\documentclass[10pt]{article}
\pdfoutput=1
\usepackage[margin=26mm]{geometry}
\usepackage{amsthm,amsmath,mathbbol,mathtools,amsfonts,amssymb,graphicx}
\usepackage{hyperref}
\usepackage{algorithm}
\usepackage{url}
\usepackage{epsfig}
\usepackage{wrapfig}
\usepackage{color}
\usepackage{blkarray}
\usepackage[textsize=scriptsize]{todonotes}
\usepackage{lmodern}
\usepackage{bigstrut}
\usepackage{booktabs}
\usepackage[numbers,sort,compress]{natbib}
\usepackage{arydshln}
\usepackage{subfig}
\usepackage{algorithm}
\usepackage{algpseudocode}
\usepackage{comment}
\usepackage{pgf,tikz}
\usepackage[all]{xy}
\usetikzlibrary{matrix, positioning}
\usetikzlibrary{patterns}
\usepackage{tikz-cd}
\usetikzlibrary{arrows}
\usepackage{pbox}
\usepackage[capitalise]{cleveref}
\usepackage{xspace}
\usepackage{xstring}
\usepackage{multicol}
\newtheorem{theorem}{Theorem}
\newtheorem{lemma}[theorem]{Lemma}
\newtheorem{example}[theorem]{Example}
\newtheorem{proposition}[theorem]{Proposition}
\newtheorem{corollary}[theorem]{Corollary}
\newtheorem{conjecture}[theorem]{Conjecture}
\newtheorem{definition}[theorem]{Definition}
\newtheorem{remark}[theorem]{Remark}

\newtheorem*{corollary*}{Corollary}

\newcommand{\R}{\mathbb{R}}
\newcommand{\Z}{\mathbb{Z}}

\newcommand{\G}{\mathbf{G}}
\newcommand{\Ss}{\mathcal{S}}
\newcommand{\T}{\mathcal{T}}

\newcommand{\Pp}{\mathcal{P}}
\newcommand{\Pb}{\mathbb{P}}

\newcommand{\Gg}{\mathcal{G}}
\newcommand{\I}{\mathcal{I}}
\newcommand{\J}{\mathcal{J}}
\newcommand{\Qq}{\mathcal{Q}}
\newcommand{\la}{\langle}
\newcommand{\ra}{\rangle}
\newcommand{\hra}{\hookrightarrow}
\newcommand{\xra}{\xrightarrow}
\newcommand{\thra}{\twoheadrightarrow}
\newcommand{\sse}{\subseteq}
\newcommand{\ol}{\overline}
\newcommand{\oEN}{\EN^\circ}
\newcommand{\ER}{\mathbf{Er}}
\DeclareMathOperator{\cok}{coker}

\DeclareMathOperator{\id}{id}

\DeclareMathOperator{\Img}{Im}
\DeclareMathOperator{\img}{im}

\DeclareMathOperator{\Ker}{Ker}
\DeclareMathOperator{\Er}{Er}
\DeclareMathOperator{\supdim}{supdim}
\DeclareMathOperator{\EN}{EN}
\DeclareMathOperator{\Pru}{Pru}

\numberwithin{theorem}{section}

\makeatletter
\renewcommand*\env@matrix[1][*\c@MaxMatrixCols c]{%
	\hskip -\arraycolsep
	\let\@ifnextchar\new@ifnextchar
	\array{#1}}
\makeatother

\title{Stabilizing decomposition of multiparameter persistence modules}
\date{}
\author{H\aa vard Bakke Bjerkevik\thanks{Department of Mathematics \& Statistics, SUNY Albany, USA\newline
Email: hbjerkevik@albany.edu}}

\begin{document}

\maketitle

\begin{abstract}
While decomposition of one-parameter persistence modules behaves nicely, as demonstrated by the algebraic stability theorem, decomposition of multiparameter modules is known to be unstable in a certain precise sense.
Until now, it has not been clear that there is any way to get around this and build a meaningful stability theory for multiparameter module decomposition.
We introduce new tools, in particular $\epsilon$-refinements and $\epsilon$-erosion neighborhoods, to start building such a theory.
We then define the $\epsilon$-pruning of a module, which is a new invariant acting like a ``refined barcode'' that shows great promise to extract features from a module by approximately decomposing it.

Our main theorem can be interpreted as a generalization of the algebraic stability theorem to multiparameter modules up to a factor of $2r$, where $r$ is the maximal pointwise dimension of one of the modules.
Furthermore, we show that the factor $2r$ is close to optimal.
Finally, we discuss the possibility of strengthening the stability theorem for modules that decompose into pointwise low-dimensional summands, and pose a conjecture phrased purely in terms of basic linear algebra and graph theory that seems to capture the difficulty of doing this.
We also show that this conjecture is relevant for other areas of multipersistence, like the computational complexity of approximating the interleaving distance, and recent applications of relative homological algebra to multipersistence.
\end{abstract}

\tableofcontents

\section{Introduction}

A fundamental reason for the interest in one-parameter persistence is the stability of decomposition of modules.
Given a one-parameter persistence module, one can decompose it uniquely into interval modules \cite{botnan2020decomposition,azumaya1950corrections} giving us an invariant called the \emph{barcode}, which is simply the multiset of intervals appearing in the decomposition.
This decomposition is stable in the sense that if the interleaving distance between two modules is small, then one can find a matching between their barcodes that matches similar intervals bijectively except possibly some short intervals \cite{chazal2009proximity,cohen2007stability}.
The nice behavior of decompositions is important for several reasons:
It is useful for computational efficiency, as one can work with small pieces instead of one large object; it makes modules easier to visualize, as a barcode is much easier to draw than a module itself (for instance using persistence diagrams \cite{chazal2012structure,cohen2007stability}); it helps to interpret data, since individual intervals may represent exactly the features of the data set that one wants to study \cite{chazal2013clustering, kanari2018topological}; and it can be used to remove noise from a module for instance by simply removing short bars \cite{edelsbrunner2002topological}.

With all these advantages, one might wonder why a similar theory of decompositions has not been developed for multiparameter modules.
Together with the added computational complexity of moving from one to several parameters \cite{bauer2021ripser,dey2022generalized}, the main reason seems to be that while one has unique decomposition of modules into indecomposables also in multiparameter persistence, this decomposition is wildly unstable.
Indeed, it has been shown that for any finitely presented $n$-parameter module $M$ and $\epsilon>0$, there is an indecomposable module $N$ within interleaving distance $\epsilon$ from $M$ \cite[Theorem A]{bauer2022generic}.
For this reason, a naive extension of the theory of stability of decompositions from the one-parameter setting is doomed to fail.

In this paper, we develop a theory for approximate decomposition of multiparameter modules and show how these stability issues can be dealt with.
As a bonus, out of this theoretical work pops a new invariant that we call the \emph{$\epsilon$-pruning} (\cref{def_pruning}) of a module, which has nice theoretical properties that allows us to consider it a stabilized multiparameter barcode.
The $\epsilon$-pruning of a module $M$ is essentially an approximation of $M$ that is particularly decomposable.
We envision the following pipeline, which closely mimics the standard pipeline in one-parameter persistent theory, where our focus is on the first two steps.
The main difference is that we add a pruning step before decomposing.
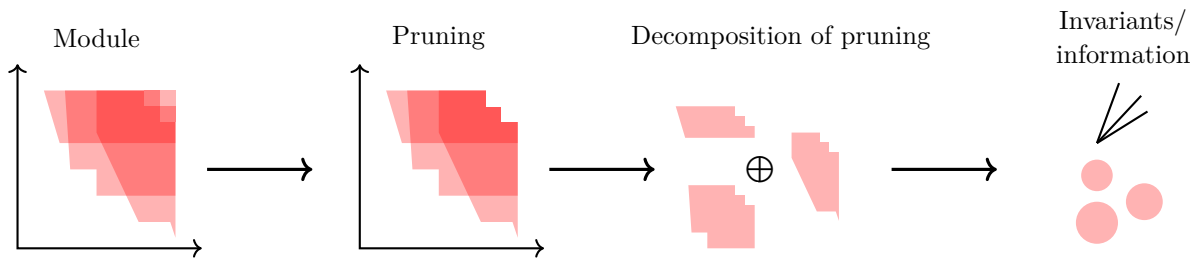
\vspace{.2cm}

\begin{tikzpicture}[scale=.7]
\begin{scope}[xshift=-6.5cm]
\node at (1.5,4){Module};
\draw[thick,<->] (0,3.5) to (0,0) to (3.5,0);
\fill[red, opacity=0.3] (.8,2) to (.5,3) to (3,3) to (3,2);
\fill[red, opacity=0.3] (1.5,1) to (1.5,1.5) to (1,1.5) to (.9,3) to (3,3) to (3,1);
\fill[red, opacity=0.3] (2.9,.5) to (2.3,.5) to (1.5,2.2) to (1.5,3) to (3,3) to (3,.2);
\fill[white] (2.4,2.4) rectangle (3,3);
\fill[red, opacity=0.3] (2.4,2.4) rectangle (3,3);
\fill[red, opacity=0.3] (2.4,2.4) rectangle (2.7,3);
\fill[red, opacity=0.3] (2.4,2.4) rectangle (3,2.7);
\draw[very thick,->] (3.6,1.5) to (5.6,1.5);
\end{scope}
\begin{scope}
\node at (1.5,4){Pruning};
\draw[thick,<->] (0,3.5) to (0,0) to (3.5,0);
\fill[red, opacity=0.3] (.8,2) to (.5,3) to (3,3) to (3,2);
\fill[red, opacity=0.3] (1.5,1) to (1.5,1.5) to (1,1.5) to (.9,3) to (3,3) to (3,1);
\fill[red, opacity=0.3] (2.9,.5) to (2.3,.5) to (1.5,2.2) to (1.5,3) to (3,3) to (3,.2);
\fill[white] (2.4,2.7) rectangle (3.1,3.1);
\fill[white] (2.7,2.4) rectangle (3.1,3.1);
\draw[very thick,->] (3.6,1.5) to (5.6,1.5);
\end{scope}
\begin{scope}[xshift = 6.5cm]
\node at (1.5,4){Decomposition of pruning};
\node at (1.1,1.5){$\bigoplus$};
\begin{scope}[yshift=.9cm, xshift=-.8cm, scale=.6]
\fill[red, opacity=0.3] (.8,2) to (.5,3) to (3,3) to (3,2);
\fill[white] (2.4,2.7) rectangle (3.1,3.1);
\fill[white] (2.7,2.4) rectangle (3.1,3.1);
\end{scope}
\begin{scope}[yshift=-.6cm, xshift=-.8cm, scale=.6]
\fill[red, opacity=0.3] (1.5,1) to (1.5,1.5) to (1,1.5) to (.9,3) to (3,3) to (3,1);
\fill[white] (2.4,2.7) rectangle (3.1,3.1);
\fill[white] (2.7,2.4) rectangle (3.1,3.1);
\end{scope}
\begin{scope}[xshift=.8cm, yshift=.4cm, scale=.6]
\fill[red, opacity=0.3] (2.9,.5) to (2.3,.5) to (1.5,2.2) to (1.5,3) to (3,3) to (3,.2);
\fill[white] (2.4,2.7) rectangle (3.1,3.1);
\fill[white] (2.7,2.4) rectangle (3.1,3.1);
\end{scope}
\draw[very thick,->] (3.6,1.5) to (5.6,1.5);
\end{scope}
\begin{scope}[xshift = 13cm]
\node at (1.5,4.3){Invariants/};
\node at (1.5,3.7){information};
\begin{scope}[xshift=.5cm, yshift=-.4]
\fill[red,opacity=.3] (.5,.5) circle (.4);
\fill[red,opacity=.3] (.5,1.4) circle (.3);
\fill[red,opacity=.3] (1.4,.9) circle (.35);
\end{scope}
\begin{scope}[yshift=2cm, xshift=1cm, scale=.6]
\draw[thick] (0,0) to (1.6,1);
\draw[thick] (0,0) to (1.4,1.5);
\draw[thick] (0,0) to (0.7,1.9);
\end{scope}
\end{scope}
\end{tikzpicture}
\vspace{.2cm}

The idea of looking at approximations of a module instead of just the module itself is similar in spirit to one of the fundamental ideas of persistence.
Just as persistent homology allows us to ignore homology classes that only show up at very limited choices of parameter values, approximating a module can let us split apart parts of a module that are only glued together in a non-persistent way.

If the original module is close to decomposing into many pieces, its $\epsilon$-pruning is guaranteed to decompose a lot (for an appropriate $\epsilon$), and we can hope to gain many of the mentioned advantages in the one-parameter case regarding computational complexity, visualization, etc.
For this reason, we believe that prunings are a very promising tool from a practical point of view.

To be able to express stability and fineness of decompositions, we first need to build up a new vocabulary, since $\epsilon$-matchings and the well-known bottleneck distance it gives rise to are not able to express a notion of approximate decomposition, and therefore do not give stability.
The first new term in our dictionary is the \emph{$\epsilon$-erosion neighborhood} $\EN_\epsilon(M)$ of a module $M$ (\cref{def_erosion_nbh}).
We show in Sections \ref{sec_erosion} and \ref{sec_red_distances} that up to a multiplicative constant, this is equivalent to the $\epsilon$-neighborhood of a module in the interleaving distance in an appropriate sense.
(\cref{thm_strong_equiv} is a way to make this statement precise.)
The erosion neighborhood of $M$ has the advantage that, unlike the interleaving neighborhood, it contains only subquotients of $M$ (\cref{def_subquotient}), which are modules that are smaller than $M$ in an certain algebraic sense.
These are much more limited than the modules that are only interleaving close to $M$, which means that the space one needs to search for approximations of $M$ shrinks dramatically.
We expect this to be useful from both a theoretical and an algorithmic point of view.

The main results of the paper concern the stability properties of the pipeline we outlined (or rather, of the composition of the first two steps).
Such stability results are essential for a theory of approximate decompositions, as without stability, we cannot know if the output we get describes the input in a reliable way.
The question we need to answer is: given two modules that are similar, are their decompositions similar?
Since any module is close to an indecomposable module, it is clear that they cannot be similar in the sense that there is a nice bijection between (large) summands.
Our solution is to introduce \emph{$\epsilon$-refinements} (\cref{def_refinement}) and $\epsilon$-prunings.
An $\epsilon$-refinement of a module $M$ is, roughly speaking, a module that is $\epsilon$-interleaved with $M$, and has a finer decomposition into indecomposables than $M$.
Such refinements allow us to further split apart summands in the decompositions we are considering, and only then try to match the indecomposables in one decomposition with those of the other.
The definition of an $\epsilon$-pruning is more opaque, but it turns out to play the role of a ``very refined'' version of $M$.
\begin{figure}
\centering
\begin{tikzpicture}[scale=.5]
\begin{scope}[xshift=0cm]
\draw[color=black,fill=black] (0,0) circle (.3);
\node at (-1.2,0){$M_3$};
\draw[color=black,fill=black] (0,2) circle (.3);
\node at (-1.2,2){$M_2$};
\draw[color=black,fill=black] (0,4) circle (.3);
\node at (-1.2,4){$M_1$};
\node at (0,6.5){$B(M)$};
\def\x{4}
\node at (\x,5.5){$B(\Pr_\epsilon(M))$};
\draw[thick, shorten <=.3cm, shorten >=.2cm] (0,0) to (\x,0);
\draw[color=black,fill=black] (\x,0) circle (.2);
\draw[thick, shorten <=.3cm, shorten >=.2cm] (0,0) to (\x,-1);
\draw[color=black,fill=black] (\x,-1) circle (.2);
\draw[thick, shorten <=.3cm, shorten >=.2cm] (0,2) to (\x,1);
\draw[color=black,fill=black] (\x,1) circle (.2);
\draw[thick, shorten <=.3cm, shorten >=.2cm] (0,2) to (\x,2);
\draw[color=black,fill=black] (\x,2) circle (.2);
\draw[thick, shorten <=.3cm, shorten >=.2cm] (0,2) to (\x,3);
\draw[color=black,fill=black] (\x,3) circle (.2);
\draw[thick, shorten <=.3cm, shorten >=.2cm] (0,4) to (\x,4);
\draw[color=black,fill=black] (\x,4) circle (.2);
\def\z{8}
\draw[color=black,fill=black] (\z,-1) circle (.3);
\node at (\z+1.2,-1){$N_4$};
\draw[color=black,fill=black] (\z,1) circle (.3);
\node at (\z+1.2,1){$N_3$};
\draw[color=black,fill=black] (\z,3) circle (.3);
\node at (\z+1.2,3){$N_2$};
\draw[color=black,fill=black] (\z,5) circle (.3);
\node at (\z+1.2,5){$N_1$};
\node at (\z,6.5){$B(N)$};
\draw[thick, shorten <=.2cm, shorten >=.3cm] (\x,0) to (\z,-1);
\draw[thick, shorten <=.2cm, shorten >=.3cm] (\x,2) to (\z,-1);
\draw[thick, shorten <=.2cm, shorten >=.3cm] (\x,1) to (\z,1);
\draw[thick, shorten <=.2cm, shorten >=.3cm] (\x,-1) to (\z,1);
\draw[thick, shorten <=.2cm, shorten >=.3cm] (\x,4) to (\z,3);
\draw[thick, shorten <=.2cm, shorten >=.3cm] (\x,3) to (\z,5);
\end{scope}
\end{tikzpicture}
\caption{Illustration of \cref{thm_main} (ii), which says that if $M$ and $N$ are $\epsilon$-interleaved, then the $\epsilon$-pruning of $M$ is a $2r\epsilon$-refinement of $N$, where $r$ is the maximum pointwise dimension of $M$.
If \cref{conj_main} holds, then the same statement holds if $r$ is the maximum pointwise dimension of the modules in $B(M)$.
(That is, the $M_i$ in the figure.)}
\label{fig_main_thm}
\end{figure}
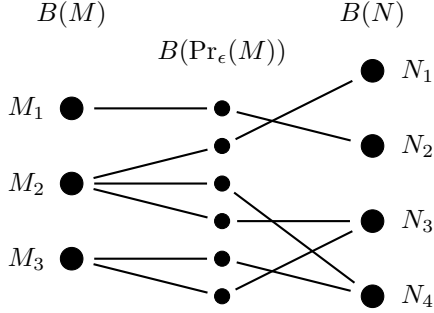
For a module $M$, let $\supdim M = \sup_{p\in \Pb} \dim M_p \in \{0, 1, \dots, \infty\}$.
\cref{cor_main} of the main theorem \cref{thm_main} says the following.
\begin{corollary*}
Let $\epsilon\geq 0$, and let $M$ and $N$ be $\epsilon$-interleaved pointwise finite dimensional modules with $r = \supdim M<\infty$.
Then $\Pru_\epsilon(M)$ is a $2r\epsilon$-refinement of both $M$ and $N$.
\end{corollary*}
This means that all the modules in the $\epsilon$-interleaving neighborhood of $M$ have a common approximate decomposition within a distance of $2r\epsilon$, where we only allow splitting and shrinking indecomposable components, and not gluing them together.
Thus, the main theorem can be interpreted as saying that the barcodes of $M$ and $N$ are similar if $M$ and $N$ are similar.
But it also suggests that the $\epsilon$-pruning of a module $M$ is a effectively a \emph{stabilized decomposition} of $M$, since it is a refinement of a whole neighborhood of $M$, not just $M$ itself.
An interpretation of this is that the pruning captures all summands (``features'') appearing in modules close to $M$, and that it is ``very decomposable'', so it could be a powerful invariant of $M$.
Compare this to the naive decomposition of $M$, which (for instance in the case of indecomposable $M$) might contain almost no information about decompositions of modules close to $N$.

A more theoretical perspective on these contributions is that they are part of a larger project of developing a theory of ``approximate'' or ``fuzzy'' algebra.
Persistence modules are objects of a type that have long been studied in representation theory, and there exists a large body of theory regarding these objects.
What makes persistence different from classical representation theory is that we work on the assumption that the objects come from noisy data sets, which means that we can only assume that any module we are given is an approximation, and does not exactly represent the phenomena that we are studying.
Therefore, instead of studying the properties of modules, we should build a theory analyzing \emph{neighborhoods} of modules.
From this perspective, Krull-Schmidt type theorems from algebra (like \cref{thm_unique_dec}) are not quite what we are looking for in persistence theory, since they say that isomorphic modules have ``isomorphic'' decompositions, and isomorphism is a notion that is too strong for us.
In persistence theory, a more appropriate Krull-Schmidt type theorem would be that similar (say, $\epsilon$-interleaved) modules have similar (approximate) decompositions.
Our notion of $\epsilon$-refinements allows us to make this statement precise, and \cref{thm_main} is exactly a theorem of this type.
As such, we argue that it is the first Krull-Schmidt theorem for multiparameter modules that is truly in the spirit of persistence.

In \cref{sec_red_distances}, we tie our work to a framework of a lot of work in persistence, namely the theory of distances.
We show that $\epsilon$-erosion neighborhoods give rise to a nicely behaved distance, and compare our work with previous work on erosion.
Then, we discuss the difficulty of defining good multiparameter distances that depend on decompositions, and introduce prunings and the pruning distance (\cref{def_pruning_dist}).
We show that the pruning distance is indeed a distance, argue that it measures the difference between the decompositions close to one module with the decompositions close to another, and conjecture (\cref{conj_pruning}) that it has nice stability properties.
We also introduce a function $f_R$ measuring the difference between modules using refinements.
This is not a distance in general, but we show that it is a distance for interval decomposable modules, and argue that it is a nicely behaved improvement on the bottleneck distance in this setting.

In \cref{sec_matchings}, we discuss how the theorem might be strengthened in the case of modules decomposing into pointwise low-dimensional summands.
We formulate a conjecture (\cref{conj_main}) that we consider to be the ultimate goal for stability of multiparameter decompositions, and speculate that what needs to be overcome to prove the conjecture is nicely expressed as a conjecture with a more combinatorial flavor (\cref{conj_CI}).
This conjecture is also relevant for stability questions in recent work on applying relative homological algebra to multipersistence, which we show in \cref{subsec_benign}, and for the computational complexity of approximating the interleaving distance.
With its simplicity and relevance for decompositions, stability and computational complexity, we are of the opinion that \cref{conj_CI} should be considered an important open problem in the theory of multipersistence, and that much can be learned from a solution to it.

It is worth highlighting that we are working in a very general setting throughout the paper.
Except in \cref{sec_matchings}, where we are seeking out modules of a simple form on purpose, we only require that the modules are pointwise finite dimensional (though for the stability results to make sense, we need a finite global bound on the pointwise dimension of the modules), and much of our work on erosion neighborhoods puts no restrictions on the modules.
In contrast, a lot of work in multipersistence requires modules to be finitely presented.
As an example, this allows us to work in the setting of tame modules due to Ezra Miller \cite[Def. 2.9]{miller2020essential}, but also in much more general settings.
More surprisingly, except one-parameter persistence being a simpler special case, the number of parameters never plays a role in the paper; all we need is a poset with a well-behaved family of shift functions.
Considering how common it is for work in multipersistence to either be done in the special case of two parameters or contain results that get weaker as the number of parameters increases, this is a significant strength of the paper.

\subsection{Other work on (approximate) decompositions}

The idea of describing a module using smaller pieces is an idea that motivates a lot of work in multiparameter persistence.
In \cite{dey2022generalized}, it is shown how to decompose distinctly graded multiparameter modules in $O(n^{2\omega+1})$, where $\omega<2.373$ is the matrix multiplication constant.
There are many recent papers on applying relative homological algebra to multiparameter persistence \cite{asashiba2023approximation, blanchette2021homological, botnan2024signed, botnan2024bottleneck, chacholski2024koszul, oudot2024stability}, where the idea is to find a multiset of intervals (some may have negative multiplicity) describing a module.
A similar sounding, but different approach is taken in \cite{loiseaux2022efficient}, where interval decomposable modules are used to approximate general modules.

Our approach is different from the latter two in that we approximate by general modules, not by collections of intervals, which allows us to get a single approximate decomposition (the $\epsilon$-pruning for small $\epsilon$) from which we can recover the original module up to a small interleaving distance.
(There is a conjecture in \cite{loiseaux2022efficient} about recovering a module from approximations, but this requires considering a whole family of approximations, not just one.)
The $\epsilon$-pruning of a module is an innovation in that it is a single approximate decomposition that has provable stability properties while only destroying a small amount of information contained in the original module.

\subsection{Acknowledgments}
The author thanks Ulrich Bauer and Steve Oudot for pointing out and helping fix an issue concerning erosion neighborhoods, and anonymous referees for valuable feedback.
Thanks are also due to the Centre for Advanced Study (CAS) at the Norwegian Academy of Science for organizing a workshop on interactions between representation theory and topological data analysis in December 2022 which inspired this work, and for hosting the author for a period in 2023 as part of the program Representation Theory: Combinatorial Aspects and Applications.
The author was funded by the Deutsche Forschungsgemeinschaft (DFG - German Research Foundation) - Project-ID 195170736 - TRR109.

\section{Motivation}
\label{sec_motivation}

We will now explain some of the challenges of generalizing the one-parameter stability theory to multipersistence and motivate our approach to building such a theory of decomposition stability.

\subsection{The algebraic stability theorem}

Though we are motivated by the study of $d$-parameter persistence modules, we will work in a more general setting, as all we really need is a poset with a collection of well-behaved shift functors.

Let $\Pb$ be a poset, which we identify with its poset category.
This category has an object for every element of $\Pb$, and a single morphism from $p$ to $q$ if $p\leq q$, and no morphism otherwise.
(Those not interested in more general posets can assume $\Pb = \R^d$ and $\R^d_\epsilon(p) = p + (\epsilon,\dots, \epsilon)$, and jump to \cref{def_module}.)
Assume that $\Pb$ is equipped with a shift function $\Pb_\epsilon \colon \Pb\to \Pb$ for each $\epsilon\in \R$ such that
\begin{itemize}
	\item $\Pb_0$ is the identity,
	\item for $p\leq q\in \Pb$, $\Pb_\epsilon(p) \leq \Pb_\epsilon(q)$,
	\item for all $p \in \Pb$ and $\epsilon\geq 0$, $\Pb_\epsilon(p)\geq p$, and
	\item for all $\epsilon,\delta\in \R$, $\Pb_{\epsilon+\delta} = \Pb_\epsilon \circ \Pb_\delta$.
\end{itemize}
It follows from the first and fourth point that all the $\Pb_\epsilon$ are bijections.
In the rest of the paper, $\Pb$ will be assumed to have these properties.

Let $\mathbf{Vec}$ be the category of vector spaces over some fixed field $k$.
\begin{definition}
\label{def_module}
A \textbf{persistence module} is a functor $M\colon \Pb\to \mathbf{Vec}$.
A morphism $f\colon M\to N$ of persistence modules is a natural transformation.
\end{definition}

Throughout the paper, ``module'' will mean persistence module.
In the case $\Pb = \R^d$ and $\R^d_\epsilon(p) = p + (\epsilon,\dots, \epsilon)$, we refer to these as $d$-parameter modules, and we sometimes separate the case $d=1$ from the case of general $d$ by talking about one-parameter modules or multiparameter modules, respectively.
We use the notation $M_p$ and $M_{p\to q}$ for $M$ applied to an object $p$ and a morphism $p\to q$, respectively.
Concretely, a module $M$ is a set of vector spaces $M_p$ and linear transformations $M_{p\to q}$ for all $p\leq q\in \Pb$ such that $M_{p\to p}$ is the identity and $M_{q\to r}\circ M_{p\to q} = M_{p\to r}$ for all $p\leq q\leq r$.

A morphism $f\colon M\to N$ can be viewed as a set of linear transformations $\{f_p\}_{p\in \Pb}$ satisfying $f_q\circ M_{p\to q} = N_{p\to q}\circ f_p$ for all $p\leq q$.
If $M_p$ is finite-dimensional for all $p$, we say that $M$ is \emph{pointwise finite-dimensional}, or \emph{pfd} for short.
Though we are mostly interested in pfd modules, we will often work with general modules, and we specify that a module is pfd only when needed.

The following theorem says that pfd persistence modules decompose into indecomposables in an essentially unique way.
\begin{theorem}[{\cite[Theorem 1.1]{botnan2020decomposition}, \cite[Theorem 1 (ii)]{azumaya1950corrections}}]
\label{thm_unique_dec}
Let $M$ be a pfd module.
Then there is a set $\{M_i\}_{i\in \Lambda}$ of nonzero indecomposable modules with local endomorphism ring such that $M\cong \bigoplus_{i\in \Lambda} M_i$.
If $M\cong \bigoplus_{j\in \Gamma} M'_j$, and each $M'_j$ is nonzero and indecomposable, then there is a bijection $\sigma\colon \Lambda\to \Gamma$ such that $M_i\cong M'_{\sigma(i)}$ for all $i\in \Lambda$.
\end{theorem}

\begin{definition}
\label{def_barcode}
Let $M$ be a pfd module isomorphic to $\bigoplus_{i\in \Lambda} M_i$, where each $M_i$ is nonzero and indecomposable.
We define the \textbf{barcode} of $M$ to be the multiset
\[
B(M)\coloneqq \{M_i\}_{i\in \Lambda},
\]
where the $M_i$ are considered as isomorphism classes of modules.
\end{definition}
By \cref{thm_unique_dec}, the barcode is well-defined, and is an isomorphism invariant.

To study properties of persistence modules that are stable with respect to small perturbations, we need a way to express closeness between modules.
The most common way of doing this in persistence theory is through \emph{$\epsilon$-interleavings} (\cref{def_interleaving}), which can be viewed as approximate isomorphisms:
If two modules are $\epsilon$-interleaved for a small $\epsilon$, then they are close to being isomorphic.
\begin{definition}
\label{def_matching}
Let $M$ and $N$ be pfd modules.
A \textbf{matching} from $M$ to $N$ is a bijection $\sigma\colon \bar B(M)\to \bar B(N)$ where $\bar B(M)\sse B(M)$ and $\bar B(N)\sse B(N)$.
If
\begin{itemize}
\item for each $Q\in \bar B(M)$, $Q$ and $\sigma(Q)$ are $\epsilon$-interleaved, and
\item each $Q\in (B(M)\setminus \bar B(M))\cup (B(N)\setminus \bar B(N))$ is $\epsilon$-interleaved with the zero module,
\end{itemize}
then $\sigma$ is an \textbf{$\epsilon$-matching} from $M$ to $N$.
The \textbf{bottleneck distance} $d_B$ is defined by
\[
d_B(M,N) = \inf\{\epsilon\mid \exists \epsilon\text{-matching between } M \text{ and } N\}.
\]
\end{definition}
By symmetry, $\sigma$ is an $\epsilon$-matching from $M$ to $N$ if and only if $\sigma^{-1}$ is an $\epsilon$-matching from $N$ to $M$.
We often say that there is an $\epsilon$-matching between $M$ and $N$ if there is an $\epsilon$-matching from one to the other.

In the one-parameter setting; that is, for modules over the poset $\R$, the barcodes only contain interval modules \cite[Theorem 1.1]{crawley2015decomposition} (see \cref{def_interval}).
Furthermore, we have the following theorem saying that decomposition into indecomposables is stable.
\begin{theorem}[The algebraic stability theorem, {\cite[Theorem 4.4]{chazal2009proximity}}]
\label{thm_ast}
If $M$ and $N$ are pfd one-parameter $\epsilon$-interleaved modules, then there is an $\epsilon$-matching between them.
As a direct consequence, $d_I(M,N)\leq d_B(M,N)$.
\end{theorem}
It was noted in \cite[Theorem~3.4]{lesnick2015theory} that the opposite inequality also holds, so we have $d_I = d_B$ for one-parameter modules.

\subsection{How to build a decomposition stability theory}

We would like to build a theory in multiparameter persistence for decomposing modules in a way that is stable.
That is, for two modules that are close in the interleaving distance, we would like to decompose them into indecomposables and match indecomposables in one module with indecomposables in the other, obtaining a statement similar to \cref{thm_ast}.
Sadly, it is not hard to come up with examples showing that this is impossible to do in a naive way.
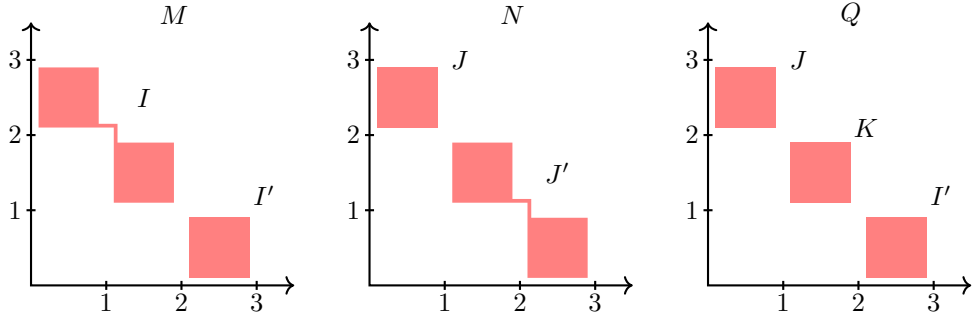
\begin{figure}
\centering
\begin{tikzpicture}[scale=1]
\node at (1.9, 3.6){$M$};
\draw[thick,<->] (0,3.5) to (0,0) to (3.5,0);
\draw[thick] (-.06,1) to (.06,1);
\node[left] at (0,1){$1$};
\draw[thick] (-.06,2) to (.06,2);
\node[left] at (0,2){$2$};
\draw[thick] (-.06,3) to (.06,3);
\node[left] at (0,3){$3$};
\draw[thick] (1,-.06) to (1,.06);
\node[below] at (1,0){$1$};
\draw[thick] (2,-.06) to (2,.06);
\node[below] at (2,0){$2$};
\draw[thick] (3,-.06) to (3,.06);
\node[below] at (3,0){$3$};
\node at (1.5,2.5){$I$};
\fill[red, opacity=0.5]
(.1,2.9) to (.1,2.1) to (1.1,2.1) to (1.1,1.1) to (1.9,1.1) to (1.9,1.9) to (1.15,1.9) to (1.15,2.15) to (.9,2.15) to (.9,2.9) to (.1,2.9);
\node at (3.1,1.2){$I'$};
\fill[red, opacity=0.5] (2.1,.1) rectangle (2.9,.9);
\begin{scope}[xshift=4.5cm]
\node at (1.9, 3.6){$N$};
\draw[thick,<->] (0,3.5) to (0,0) to (3.5,0);
\draw[thick] (-.06,1) to (.06,1);
\node[left] at (0,1){$1$};
\draw[thick] (-.06,2) to (.06,2);
\node[left] at (0,2){$2$};
\draw[thick] (-.06,3) to (.06,3);
\node[left] at (0,3){$3$};
\draw[thick] (1,-.06) to (1,.06);
\node[below] at (1,0){$1$};
\draw[thick] (2,-.06) to (2,.06);
\node[below] at (2,0){$2$};
\draw[thick] (3,-.06) to (3,.06);
\node[below] at (3,0){$3$};
\node at (1.2,3){$J$};
\fill[red, opacity=0.5] (.1,2.1) rectangle (.9,2.9);
\node at (2.5,1.5){$J'$};
\fill[red, opacity=0.5]
(1.1,1.9) to (1.1,1.1) to (2.1,1.1) to (2.1,.1) to (2.9,.1) to (2.9,.9) to (2.15,.9) to (2.15,1.15) to (1.9,1.15) to (1.9,1.9) to (1.1,1.9);
\end{scope}
\begin{scope}[xshift=9cm]
\node at (1.9, 3.6){$Q$};
\draw[thick,<->] (0,3.5) to (0,0) to (3.5,0);
\draw[thick] (-.06,1) to (.06,1);
\node[left] at (0,1){$1$};
\draw[thick] (-.06,2) to (.06,2);
\node[left] at (0,2){$2$};
\draw[thick] (-.06,3) to (.06,3);
\node[left] at (0,3){$3$};
\draw[thick] (1,-.06) to (1,.06);
\node[below] at (1,0){$1$};
\draw[thick] (2,-.06) to (2,.06);
\node[below] at (2,0){$2$};
\draw[thick] (3,-.06) to (3,.06);
\node[below] at (3,0){$3$};
\node at (1.2,3){$J$};
\fill[red, opacity=0.5] (.1,2.1) rectangle (.9,2.9);
\node at (2.1,2.1){$K$};
\fill[red, opacity=0.5] (1.1,1.1) rectangle (1.9,1.9);
\node at (3.1,1.2){$I'$};
\fill[red, opacity=0.5] (2.1,.1) rectangle (2.9,.9);
\end{scope}
\end{tikzpicture}
\caption{The module $M= I\oplus I'$ on the left, $N= J\oplus J'$ in the middle, and $Q = J\oplus K\oplus I'$ on the right.
$M$ and $N$ are $0.03$-interleaved and have a common $0.03$-refinement $Q$, but any bijection between $\{I, I'\}$ and $\{J, J'\}$ would match two components that are not $\delta$-interleaved for any $\delta<0.4$.
}
\label{fig_large_bottleneck}
\end{figure}
In \cref{fig_large_bottleneck}, we see modules $M$ and $N$ similar to those in \cite[Ex. 9.1]{botnan2018algebraic} that are $\epsilon$-interleaved for a small $\epsilon$, but $B(M)$ contains a module $I$  that is not $\delta$-interleaved with a module in $B(N)$ unless $\delta \gg \epsilon$.
In this example, however, we see that we ``almost'' have a matching between indecomposable summands of $M$ and $N$:
With a little surgery, we could remove the thin parts in $I$ and $J'$ so that $B(M)$ and $B(N)$ would contain three modules each, allowing a good matching (that is, an $\epsilon$-matching for a small $\epsilon$).
And there is in fact a known operation that would achieve this, namely \emph{erosion} (\cref{def_erosion}).
In this example, we have that $M$ and $N$ do not allow a good matching, but $\Er_\epsilon(M)$ and $\Er_\epsilon(N)$ do.

Unfortunately, eroding does not solve all our problems.
In \cref{fig_erosion_fail}, we again have modules $M$ and $N$ that are $\epsilon$-interleaved and do not allow a good matching, but would allow a good matching if we could make some small changes to the modules -- this time, $N$ splits into a direct sum of two indecomposables if we remove the upper right corner.
But now erosion does not get the job done.
For any $\delta<0.9$, $\Er_\delta(N)$ is indecomposable, while $B(\Er_\delta(M))$ has two modules coming from the two indecomposable summands of $M$.

The $0.1$-pruning of $M$, on the other hand, gives us what we want: it splits into two pieces while at the same time approximating both $M$ and $N$.
(We take a closer look at a similar example in \cref{ex_pruning}.)
Applying pruning to the example in \cref{fig_large_bottleneck}, we obtain a module decomposing into three summands, similar to $Q$.
Going by these examples, pruning seems to give good approximate decompositions of modules.
To formalize what a ``good approximate decomposition'' of a module should mean, we first introduce the \emph{$\epsilon$-erosion neighborhood} $\EN_\epsilon(N)$ (\cref{def_erosion_nbh}) of modules that ``lie between'' $N$ and its $\epsilon$-erosion in an appropriate sense.
Then we introduce $\epsilon$-refinements:
\begin{definition}
\label{def_refinement}
Let $\epsilon\geq 0$, and let $M$ be a pfd module.
An \textbf{$\epsilon$-refinement} of $M$ is a module isomorphic to a module $\bigoplus_{M'\in B(M)} N'$ such that $N'\in \EN_\epsilon(M')$ for each $M'\in B(M)$.
\end{definition}
Intuitively, $\epsilon$-refining a module allows us to split apart modules in $B(M)$ that are $\epsilon$-close to decomposing to obtain a finer barcode $B(M')$ (where $M'$ is an $\epsilon$-refinement of $M$), but we are not allowed to glue together summands.
(If we glue all the summands together into one piece, this will tell us nothing interesting about the decomposition of a module, which is why we do not permit gluing.)
Splitting apart $M_i$ corresponds to choosing an $N_i$ that decomposes into several summands.
Given two modules $M$ and $N$, we take the view that the appropriate notion of a stabilized $\epsilon$-matching in the multiparameter setting is a module $R$ that is simultaneously an $\epsilon$-refinement of $M$ and of $N$.
This view certainly works well in the toy examples of \cref{fig_large_bottleneck} and \cref{fig_erosion_fail}: in both cases there are common $\epsilon$-refinements for small $\epsilon$ ($Q$ and $\Pru_{0.1}(M)$, respectively).
By definition of refinement, these can be obtained by breaking apart the ``almost decomposable'' indecomposables in each barcode: E.g., in \cref{fig_large_bottleneck}, we associate $J\oplus K$ to $I$ (``breaking apart'' $I$ into two pieces) and $I'$ to $I'$ to obtain the $0.03$-refinement $Q$ of $M$.
It also generalizes parts of the one-parameter stability theory in a nice way:
In one proof of the algebraic stability theorem (\cref{thm_ast}), Bauer and Lesnick \cite{bauer2015induced} construct a subquotient that is essentially a common refinement of two modules to get a matching.
As we show in \cref{thm_eq_f_R_d_B} (with help from a result in \cite{bauer2015induced}), an $\epsilon$-matching between one-parameter modules gives rise to a common $\epsilon$-refinement and vice versa, and we also show this for \emph{upset decomposable} modules, a special class of multiparameter modules.
This justifies our view of common refinements as stabilized matchings:
In many nice cases, common refinements and matchings are equivalent, but common refinements avoid the unstable behavior of matchings in the examples in \cref{fig_large_bottleneck,fig_erosion_fail}.

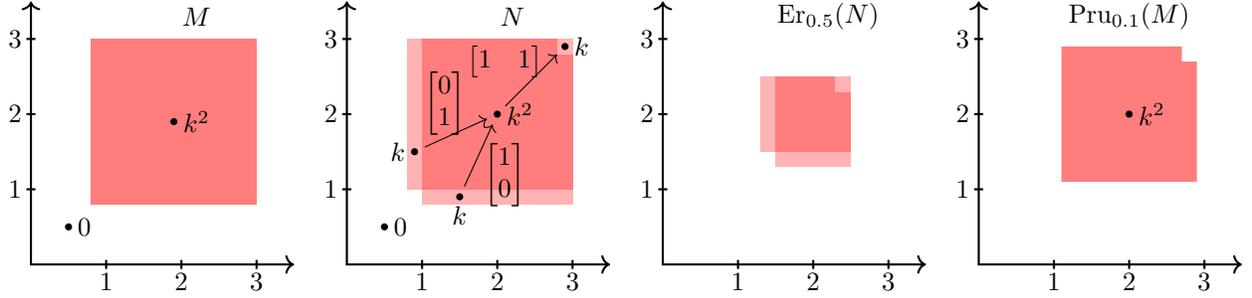
\begin{figure}
\centering
\begin{tikzpicture}[scale=1]
\draw[thick,<->] (0,3.5) to (0,0) to (3.5,0);
\draw[thick] (-.06,1) to (.06,1);
\node[left] at (0,1){$1$};
\draw[thick] (-.06,2) to (.06,2);
\node[left] at (0,2){$2$};
\draw[thick] (-.06,3) to (.06,3);
\node[left] at (0,3){$3$};
\draw[thick] (1,-.06) to (1,.06);
\node[below] at (1,0){$1$};
\draw[thick] (2,-.06) to (2,.06);
\node[below] at (2,0){$2$};
\draw[thick] (3,-.06) to (3,.06);
\node[below] at (3,0){$3$};
\node at (2.2,3.3){$M$};
\fill[red, opacity=0.3] (.8,.8) rectangle (3,3);
\fill[red, opacity=0.3] (.8,.8) rectangle (3,3);
\node[right] at (1.9,1.9){$k^2$};
\draw[color=black,fill=black] (1.9,1.9) circle (.04);
\node[right] at (.5,.5){$0$};
\draw[color=black,fill=black] (.5,.5) circle (.04);
\begin{scope}[xshift=4.2cm]
\draw[thick,<->] (0,3.5) to (0,0) to (3.5,0);
\draw[thick] (-.06,1) to (.06,1);
\node[left] at (0,1){$1$};
\draw[thick] (-.06,2) to (.06,2);
\node[left] at (0,2){$2$};
\draw[thick] (-.06,3) to (.06,3);
\node[left] at (0,3){$3$};
\draw[thick] (1,-.06) to (1,.06);
\node[below] at (1,0){$1$};
\draw[thick] (2,-.06) to (2,.06);
\node[below] at (2,0){$2$};
\draw[thick] (3,-.06) to (3,.06);
\node[below] at (3,0){$3$};
\node at (2.2,3.3){$N$};
\fill[red, opacity=0.3] (.8,1) rectangle (3,3);
\fill[red, opacity=0.3] (1,.8) rectangle (3,3);
\fill[white] (2.8,2.8) rectangle (3,3);
\fill[red, opacity=0.3] (2.8,2.8) rectangle (3,3);
\node[right] at (2,2){$k^2$};
\draw[color=black,fill=black] (2,2) circle (.04);
\node[right] at (.5,.5){$0$};
\draw[color=black,fill=black] (.5,.5) circle (.04);
\node[below] at (1.5,.9){$k$};
\draw[color=black,fill=black] (1.5,.9) circle (.04);
\node[left] at (.9,1.5){$k$};
\draw[color=black,fill=black] (.9,1.5) circle (.04);
\node[right] at (2.9,2.9){$k$};
\draw[color=black,fill=black] (2.9,2.9) circle (.04);
\draw[->, shorten <=.15cm, shorten >=.15cm] (.9,1.5) to (2,2);
\draw[->, shorten <=.15cm, shorten >=.15cm] (1.5,.9) to (2,2);
\draw[->, shorten <=.15cm, shorten >=.15cm] (2,2) to (2.9,2.9);
\node at (2.1,1.2){$\begin{bmatrix}
1\\
0
\end{bmatrix}$};
\node at (1.3,2.15){$\begin{bmatrix}
0\\
1
\end{bmatrix}$};
\node at (2.1,2.7){$\begin{bmatrix}
1&1
\end{bmatrix}$};
\end{scope}
\begin{scope}[xshift=8.4cm]
\draw[thick,<->] (0,3.5) to (0,0) to (3.5,0);
\draw[thick] (-.06,1) to (.06,1);
\node[left] at (0,1){$1$};
\draw[thick] (-.06,2) to (.06,2);
\node[left] at (0,2){$2$};
\draw[thick] (-.06,3) to (.06,3);
\node[left] at (0,3){$3$};
\draw[thick] (1,-.06) to (1,.06);
\node[below] at (1,0){$1$};
\draw[thick] (2,-.06) to (2,.06);
\node[below] at (2,0){$2$};
\draw[thick] (3,-.06) to (3,.06);
\node[below] at (3,0){$3$};
\node at (2.2,3.3){$\Er_{0.5}(N)$};
\fill[red, opacity=0.3] (1.3,1.5) rectangle (2.5,2.5);
\fill[red, opacity=0.3] (1.5,1.3) rectangle (2.5,2.5);
\fill[white] (2.3,2.3) rectangle (2.5,2.5);
\fill[red, opacity=0.3] (2.3,2.3) rectangle (2.5,2.5);
\end{scope}
\begin{scope}[xshift=12.6cm]
\draw[thick,<->] (0,3.5) to (0,0) to (3.5,0);
\draw[thick] (-.06,1) to (.06,1);
\node[left] at (0,1){$1$};
\draw[thick] (-.06,2) to (.06,2);
\node[left] at (0,2){$2$};
\draw[thick] (-.06,3) to (.06,3);
\node[left] at (0,3){$3$};
\draw[thick] (1,-.06) to (1,.06);
\node[below] at (1,0){$1$};
\draw[thick] (2,-.06) to (2,.06);
\node[below] at (2,0){$2$};
\draw[thick] (3,-.06) to (3,.06);
\node[below] at (3,0){$3$};
\node at (2,3.3){$\Pru_{0.1}(M)$};
\fill[red, opacity=0.3] (1.1,1.1) to (2.9,1.1) to (2.9,2.7) to (2.7,2.7) to (2.7,2.9) to (1.1,2.9);
\fill[red, opacity=0.3] (1.1,1.1) to (2.9,1.1) to (2.9,2.7) to (2.7,2.7) to (2.7,2.9) to (1.1,2.9);
\node[right] at (2,2){$k^2$};
\draw[color=black,fill=black] (2,2) circle (.04);
\end{scope}
\end{tikzpicture}
\caption{Left: The direct sum $M$ of two interval modules supported on $[0.8,3]^2$.
Middle left: an indecomposable module $N$.
Middle right: The $0.5$-erosion of $N$, which is also indecomposable.
Right: The $0.1$-pruning $\Pru_{0.1}(M)$ of $M$ decomposes into two summands, allowing a good matching with the components of $M$.
It turns out that $\Pru_{0.1}(M)$ is a $0.3$-refinement of both $M$ and $N$.}
\label{fig_erosion_fail}
\end{figure}

A natural goal is to use these refinements to obtain a matching theorem analogous to \cref{thm_ast}.
With our main result, \cref{thm_main}, we achieve a version of this.
Interpreting common $\delta$-refinements as approximate $\delta$-matchings between modules, the theorem says that for any $\epsilon$-interleaved modules $M$ and $N$, there is an approximate $2r\epsilon$-matching between them, where $r$ is the maximum pointwise dimension of $M$.
This is a close analogy of \cref{thm_ast}, but with an extra factor of $2r$.
Unfortunately, \cref{thm_counterex} shows that this is close to optimal in a certain sense, and therefore this factor cannot be removed.
The common refinement can be chosen independently of $N$, since we can let it be the $\epsilon$-pruning of $M$.
In other words, the relationship between refinements and prunings is that refinements formalize what an approximate decomposition of a module is, while prunings are a specific construction that turns out to give good refinements of modules (but they are not the only way to refine a module).

Distances provide a convenient language for discussing stability; for instance the algebraic stability theorem (\cref{thm_ast}) expresses stability of decomposition simply as $d_B\leq d_I$ for one-parameter modules.
For interval decomposable modules, refinements give rise to a distance $f_R$ (\cref{def_f_R}, \cref{thm_f_R_distance}).
$f_R$ is defined on all pfd modules, but is not a distance in general.
(We discuss the question of how to define a good ``decomposition distance'' for general modules in \cref{subsec_prunings}.)
In \cref{cor_f_R}, we rephrase \cref{thm_main} as $f_R\leq 2rd_I$.
It follows that $f_R\leq 2d_I$ for pointwise at most one-dimensional modules, while one can show that $d_B(N,Q)> 10d_I(N,Q)$ for the modules in \cref{fig_large_bottleneck}, and one can increase the constant $10$ arbitrarily by thinning the ``bridge'' in $J'$.
Thus, $d_B$ is completely unstable even in very simple examples, while $f_R$ allows positive stability results.

Our results arguably pretty much settle the stability question for modules that decompose into a few large components.
If the modules decompose (or are close to decomposing) into many pointwise low-dimensional summands, on the other hand, much remains to be understood about stability.
We conjecture in \cref{conj_main} that the $r$ appearing in the theorem can be replaced by a constant times the maximum pointwise dimension of the modules in $B(M)$, which might be significantly smaller than the maximum pointwise dimension of $M$.
We emphasize that \cref{conj_main} cannot be phrased in terms of the bottleneck distance due to its instability.
Before our definition of refinements, it is not clear how such a conjecture could have been expressed at all.

\section{The $\epsilon$-erosion neighborhood of a module}
\label{sec_erosion}

We now define and study the relationships between erosions and erosion neighborhoods, subquotients and interleavings.
\cref{fig_nhoods} gives a schematic summary of how these concepts and $\epsilon$-refinements are related.
The large pink region is the $\epsilon$-interleaving neighborhood of a module $M$ for some $\epsilon>0$.
The subquotients of $M$ can be thought of as modules that are ``smaller'' than $M$ (this intuition is formalized by the poset relation of \cref{lem_subq_poset}), and are illustrated as the region below the straight lines in \cref{fig_nhoods}.
The $\epsilon$-erosion neighborhood $\EN_\epsilon(M)$ (\cref{def_erosion_nbh}) only contains (isomorphism classes of) subquotients of $M$ of a specific form, and is contained in the $\epsilon$-interleaving neighborhood of $M$ by \cref{thm_erosion_containment_gives_int}.
The $\epsilon$-refinements of $M$ are even more specific: they are contained in $\EN_\epsilon(M)$ by \cref{lem_refins_are_ers} (ii).
As the figure illustrates, and this is proved in \cref{lem_int_subq_not_in_er_nhood}, there are sometimes subquotients of $M$ that are $\epsilon$-interleaved with $M$, but are not in $\EN_\epsilon(M)$, so $\EN_\epsilon(M)$ should not be thought of simply as the subquotients of $M$ that are $\epsilon$-interleaved with $M$.

Though $\EN_\epsilon(M)$ is much more limited than the $\epsilon$-interleaving neighborhood of $M$, we use results like \cref{lem_if_int_then_er} and \cref{thm_strong_equiv} to argue that it captures what happens in the vicinity of $M$ essentially just as well as what the interleaving neighborhood does.
A convenient aspect of $\EN_\epsilon(M)$ is that we can define a canonical set of modules (not isomorphisms classes!) $\oEN_\epsilon(M)$ (\cref{def_erosion_nbh}) such that every isomorphism class in $\EN_\epsilon(M)$ has at least one representative in $\oEN_\epsilon(M)$.
This makes the erosion neighborhood a lot more concrete and structured than the interleaving neighborhood.

The $\epsilon$-erosion $\Er_\epsilon(M)$ can be thought of as the ``minimal'' module in $\EN_\epsilon(M)$.
One needs to be careful, however; a somewhat more precise statement is that $\oEN_\epsilon(M)$ contains a family $\ER_\epsilon(M)$ of modules (\cref{def_ER}) canonically isomorphic to $\Er_\epsilon(M)$ that acts as a lower boundary of $\oEN_\epsilon(M)$.
That is, there is a relation $\preceq$ on subquotients of $M$ such that $Q\in \oEN_\epsilon(M)$ if and only if there is an $N\in \ER_\epsilon(M)$ with $N\preceq Q\preceq M$.
The precise version of this statement is given in \cref{lem_squeeze}.

\begin{figure}
\centering
\begin{tikzpicture}[scale=2]
\draw (-1,-2) to (0,0) to (1,-2);
\draw[thick] (-.5,-1) to [out=-40,in=180] (0,-1.6) to [out=0,in=-140] (.5,-1);
\fill[red, opacity=0.3] (-.5,-1) to [out=-40,in=180] (0,-1.6) to [out=0,in=-140] (.5,-1) to [out=40,in=-90] (1.2,0) to [out=90,in=-10] (0,1) to [out=170,in=90] (-1.5,-.5) to [out=-90,in=180] (-.7,-1.5) to [out=0,in=-60] (-.6,-1.2) to [out=120,in=-130] (-1.1,-.7) to [out=50,in=140] (-.5,-1);
\fill[red, opacity=0.3] (-.5,-1) to [out=-40,in=180] (0,-1.6) to [out=0,in=-140] (.5,-1) to (0,0);
\begin{scope}
\clip (-.5,-1) to [out=-40,in=180] (0,-1.6) to [out=0,in=-140] (.5,-1) to (0,0);
\fill[red, opacity=0.5] (0,0) to (-.2,-1.7) to (.2,-1.7);
\end{scope}
\draw[color=black,fill=black] (0,0) circle (.04);
\node[above] at (0,0){$M$};
\node at (0,-.9){$\EN_\epsilon(M)$};
\draw (.255,-1.4) to (1.3,-1.4);
\node[right] at (1.3,-1.4){$\ER_\epsilon(M)$};
\node at (-.85,-.3){$d_I(M,-)\leq \epsilon$};
\draw (0,-1.9) to (0,-1.4);
\node at (0,-2){$\epsilon$-refinements of $M$};
\end{tikzpicture}
\caption{An illustration of the relationships between the subquotients of $M$, the interleaving and erosion neighborhoods of $M$, and the $\epsilon$-refinements and $\epsilon$-erosion of $M$.
The figure is somewhat inconsistent for easier visualization:
$\EN_\epsilon(M)$ is a set of equivalence classes of modules up to isomorphism.
Since $\ER_\epsilon(M)$ is a set of isomorphic modules, it only represents one element of $\EN_\epsilon(M)$, contrary to how it is drawn.
}
\label{fig_nhoods}
\end{figure}

\subsection{Key definitions and properties}

Recall that a persistence module is a functor $M\colon \Pb\to \mathbf{Vec}$ and is described in terms of vector spaces $M_p$ and linear transformations $M_{p\to q}$.
A morphism $f\colon M\to N$ of persistence modules is a natural transformation and is made up of pointwise morphisms $f_p\colon M_p\to N_p$.

\begin{definition}
\label{def_shift}
For a module $M$ and $\epsilon\in \R$, let the \textbf{$\epsilon$-shift of $M$}, $M(\epsilon)$, be the module defined by $M(\epsilon)_p = M_{\Pb_\epsilon(p)}$ and $M(\epsilon)_{p\to q} = M_{\Pb_\epsilon(p)\to \Pb_\epsilon(q)}$.
For $\epsilon\leq \delta$, let $M_{\epsilon\to \delta}\colon M(\epsilon)\to M(\delta)$ be the morphism given by $(M_{\epsilon\to \delta})_p = M_{\Pb_\epsilon(p)\to \Pb_\delta(p)}$.
\end{definition}
For one-parameter modules, the notation $M_{\epsilon\to \delta}$ is ambiguous, as it can mean both a morphism of vector spaces from $M_\epsilon$ to $M_\delta$ and a morphism of modules from $M(\epsilon)$ to $M(\delta)$.
Except in \cref{lem_int_subq_not_in_er_nhood}, where we are working explicitly with one-parameter modules, $M_{\epsilon\to \delta}$ will always refer to the morphism of modules.

There is a category of persistence modules, namely the functor category $\mathbf{Vec}^{\Pb}$.
For any $\epsilon\geq 0$, $\epsilon$-shift is a functor on this category.
In particular, it also acts on morphisms:
$f(\epsilon)$ is given by $f(\epsilon)_p = f_{\Pb_\epsilon(p)}$ for all $p\in \Pb$.

\begin{definition}
\label{def_erosion}
Let $M$ be a module and let $\epsilon\geq 0$.
Let $\Img_\epsilon(M) \coloneqq \img(M_{-\epsilon\to 0})\subseteq M$ and $\Ker_\epsilon(M) \coloneqq \ker(M_{0\to \epsilon}) \subseteq M$.
For $\epsilon\geq 0$, we define the $\epsilon$-erosion of $M$ to be
\[
\Er_\epsilon(M) = \Img_\epsilon(M)/(\Ker_\epsilon(M)\cap \Img_\epsilon(M)).
\]
\end{definition}
Note that since $\Img_0(M) = M$ and $\Ker_0(M) = 0$, $\Er_0(M) = M/0 \cong M$.
For $\epsilon\leq\delta$ we have $\Img_\epsilon\supseteq \Img_\delta$ and $\Ker_\epsilon\sse \Ker_\delta$, so it makes sense to think of the $\epsilon$-erosion of $M$ as ``shrinking'' $M$ from the top and the bottom by more and more as $\epsilon$ increases.

For any modules $M$ and $N$ with submodules $M_1\sse M_2$ and $N_1\sse N_2$, respectively, and a morphism $f\colon M\to N$ with $f(M_1)\sse N_1$ and $f(M_2)\sse N_2$, there is a canonical morphism $\bar f\colon M_1/M_2 \to N_1/N_2$ induced by $f$.
Throughout the paper, ``induced morphism'' will refer to such a canonical morphism.
\begin{remark}
\label{rem_not_unique_er}
Our choice of the submodules $\Img_\epsilon(M)$ and $\Ker_\epsilon(M)\cap \Img_\epsilon(M)$ is somewhat arbitrary; for any submodule $K$ of $M$ with
\begin{equation}
\label{eq_er_set}
\Ker_\epsilon(M)\cap \Img_\epsilon(M) \subseteq K\subseteq \Ker_\epsilon(M),
\end{equation}
we have a canonical isomorphism $\Er_\epsilon(M) \xrightarrow{\sim} (\Img_\epsilon(M)+K)/K$ induced by the inclusion $\Img_\epsilon(M)\hookrightarrow \Img_\epsilon(M)+K$.
(We leave to the reader to check that this morphism is indeed both injective and surjective.)
Thus, there is a whole set of alternative representations of $\Er_\epsilon(M)$ as quotients of submodules of $M$ that are all equivalent up to isomorphism.
In particular, $(\Img_\epsilon(M)+\Ker_\epsilon(M))/\Ker_\epsilon(M)$ is a natural alternative to our definition.
\end{remark}
\begin{definition}
\label{def_ER}
For a module $M$ and $\epsilon\geq 0$, we define $\ER_\epsilon(M)$ as the set of modules of the form $(\Img_\epsilon(M)+K)/K$, where $K$ satisfies \cref{eq_er_set}.
\end{definition}
By \cref{rem_not_unique_er}, all the modules of $\ER_\epsilon(M)$ are isomorphic to $\Er_\epsilon(M)$ by a canonical isomorphism.

Other modules isomorphic to $\Er_\epsilon(M)$ are $M(-\epsilon)/\ker M_{-\epsilon\to \epsilon} = M(-\epsilon)/\Ker_{2\epsilon}(M(-\epsilon))$ or $\img M_{-\epsilon\to \epsilon} = \Img_{2\epsilon}(M(\epsilon))$, and there are canonical isomorphisms
\[
M(-\epsilon)/\Ker_{2\epsilon}(M(-\epsilon)) \to \Er_\epsilon(M) \to \Img_{2\epsilon}(M(\epsilon))
\]
induced by the shift morphisms $M(-\epsilon) \to M \to M(\epsilon)$.
We prefer our definition of $\Er_\epsilon(M)$ because it expresses $\Er_\epsilon(M)$ as a quotient of submodules of $M$, and not of $M(-\epsilon)$ or $M(\epsilon)$, though the isomorphism $\Er_\epsilon(M) \cong \img M_{-\epsilon\to \epsilon}$ in particular will come in handy in several proofs.

To our surprise, we were not able to find a reference for \cref{def_erosion}, though the definition seems to have been known in the persistence community for a while.
Erosions in a somewhat different sense were defined and discussed for functions from topological spaces to $\R^2$ in \cite{edelsbrunner2011stability}, and was first introduced for multiparameter modules in \cite{frosini2013stable}.
(The distance $d_T$ in the latter is effectively an erosion distance.)
In \cite{patel2018generalized} and \cite{puuska2020erosion} some of these ideas are developed further.
Though the context and generality in these papers vary, their notion of erosion distance, which we will discuss in \cref{sec_red_distances}, boils down to the same definition when specialized or adapted to our setting.
This definition does not use our notion of erosions on the nose; rather, it passes straight to the pointwise dimension of $\Er_\epsilon(M)$.
Since $\Er_\epsilon(M)\cong \img(M_{-\epsilon\to \epsilon})$, this is in fact a version of the \emph{rank invariant}, which associates to each $a\leq b\in \Pb$ the rank of $M_{a\to b}$, which is equal to $\dim(\img(M_{a\to b}))$.
For our purposes, it makes more sense to use \cref{def_erosion}, as $\Er_\epsilon(M)$ is a much more informative invariant of $M$ than the pointwise dimensions of $\Er_\epsilon(M)$.

Erosion is closely related to subquotients:
\begin{definition}
\label{def_subquotient}
Let $M$ be a module.
A module $N$ is a \textbf{subquotient} of $M$ if there is a module $M'$ with a monomorphism $f\colon M'\hookrightarrow M$ and an epimorphism $g\colon M' \twoheadrightarrow N$.
We write $N\leq M$.
\end{definition}
One can interpret $N\leq M$ as $N$ being smaller than $M$ in an algebraic sense, though this intuition can be misleading in the non-pfd setting.

Let $f$ and $g$ be as in the definition, and define $I = \img f$ and $K = f(\ker g)$.
We have $I/K\cong M'/\ker g \cong N$, so $N$ being a subquotient of $M$ is equivalent to the existence of submodules $K\subseteq I\subseteq M$ with $N\cong I/K$.
In particular, for any module $M$ and $\epsilon\geq 0$, $\Er_\epsilon(M)\leq M$.

The notation suggests that this is a poset relation between modules (up to isomorphism), and this is indeed the case for the class of pfd modules:
\begin{lemma}
\label{lem_subq_poset}
The subquotient relation $\leq$ is a partial order on the set of isomorphism classes of pfd modules.
\end{lemma}
\begin{proof}
Reflexivity is easy: just let $f$ and $g$ be the identity.

Next, we prove antisymmetry.
If $N\leq M$ and $M\ncong N$, either $f$ or $g$ is not an isomorphism.
Suppose $f$ is not an isomorphism.
Then there is a $p\in \Pb$ such that $f_p$ is a proper inclusion, which means that $\dim(N_p)<\dim(M_p)$, as the dimensions are finite.
A similar statement holds if $g$ is not an isomorphism.
If in addition $M\leq N$, then $\dim(N_p)\geq \dim(M_p)$, a contradiction.

Finally, we prove transitivity.
If $M\leq N \leq Q$, there are modules $M'$ and $Q'$ with morphisms
\[
M \twoheadleftarrow M' \hra N \twoheadleftarrow Q' \hra Q.
\]
Taking the pullback of the middle two morphisms gives us morphisms $M' \twoheadleftarrow N' \hra Q'$, since pullbacks preserve monomorphisms and epimorphisms.
This gives us
\[
M \twoheadleftarrow M' \twoheadleftarrow N' \hra Q' \hra Q.
\]
Composing morphisms, we get $M \twoheadleftarrow N' \hra Q$, so $M\leq Q$.
\end{proof}

\begin{lemma}[Erosion is well-behaved]
\label{lem_er_behaved}
Let $\epsilon,\delta\geq 0$.
\begin{itemize}
\item[(i)] If there is a monomorphism $M\hra N$, then there is a monomorphism $\Er_\epsilon(M)\hra \Er_\epsilon(N)$.
If there is an epimorphism $M\thra N$, then there is an epimorphism $\Er_\epsilon(M)\thra \Er_\epsilon(N)$.
Thus, $M\leq N$ implies $\Er_\epsilon(M)\leq \Er_\epsilon(N)$.
\item[(ii)] For any set $\{M_\ell\}_{\ell\in\Lambda}$ of modules, we have $\Er_\epsilon(\bigoplus_{\ell\in \Lambda} M_\ell) \cong \bigoplus_{\ell\in \Lambda} \Er_\epsilon(M_\ell)$.
\item[(iii)] For any module $M$, we have $\Er_\epsilon(\Er_\delta(M)) \cong \Er_{\epsilon+\delta}(M)$.
\end{itemize}
\end{lemma}
\begin{proof}
(i)
In the former case, we can assume $M\sse N$, which gives
\[
\Er_\epsilon(M)\cong \img M_{-\epsilon\to \epsilon} \sse \img N_{-\epsilon\to \epsilon} \cong \Er_\epsilon(N).
\]
In the latter case, we can assume that $M = N/A$ for some $A\sse N$.
We get
\[
\Er_\epsilon(M) \cong \img(N/A)_{-\epsilon\to \epsilon} \cong \img N_{-\epsilon\to \epsilon}/(A(\epsilon)\cap \img N_{-\epsilon\to \epsilon}) \leq \Er_\epsilon(N), 
\]
since $\Er_\epsilon(N) \cong \img N_{-\epsilon\to \epsilon}$.

(ii)
We have
\[
\Er_\epsilon\left(\bigoplus_{\ell\in \Lambda} M_\ell\right)
\cong \img\left(\bigoplus_{\ell\in \Lambda} M_\ell\right)_{-\epsilon\to \epsilon}
\cong \bigoplus_{\ell\in \Lambda} \img (M_\ell)_{-\epsilon\to \epsilon}
\cong \bigoplus_{\ell\in \Lambda} \Er(M_\ell).
\]

(iii)
We have $\Er_\delta(M) \cong \img M_{-\delta \to \delta}$, so $\Er_\epsilon(\Er_\delta(M)) \cong \img(\img M_{-\delta \to \delta})_{-\epsilon\to\epsilon}$.
Now $(\img M_{-\delta \to \delta})_{-\epsilon\to\epsilon}$ is the morphism from $\img M_{-\delta-\epsilon \to \delta-\epsilon}\sse M(\delta-\epsilon)$ to $\img M_{-\delta+\epsilon \to \delta+\epsilon}\sse M(\delta+\epsilon)$ we get by restricting $M_{\delta-\epsilon\to \delta+\epsilon}$.
Thus,\belowdisplayskip=-12pt
\begin{align*}
\img(\img M_{-\delta \to \delta})_{-\epsilon\to\epsilon} &= M_{\delta-\epsilon\to \delta+\epsilon}(\img M_{-\delta-\epsilon \to \delta-\epsilon})\\
&= \img M_{-\delta-\epsilon \to \delta+\epsilon}\\
&\cong \Er_{\epsilon+ \delta}(M).
\end{align*}
\end{proof}

\subsection{The erosion neighborhood}

\begin{definition}
\label{def_erosion_nbh}
For a module $M$ and $\epsilon\in [0,\infty)$, define $\oEN_\epsilon(M)$ as the set of modules of the form $M_1/M_2$, where $M_2\sse M_1\sse M$, $\Img_\epsilon(M)\sse M_1$ and $M_2\sse \Ker_\epsilon(M)$.
We define the \textbf{$\epsilon$-erosion neighborhood} of a module $M$ as the collection $\EN_\epsilon(M)$ of modules $N$ isomorphic to a module in $\oEN_\epsilon(M)$.
\end{definition}
The following lemma shows that the $\epsilon$-erosion is a special case of $\epsilon$-refinements, which are special cases of modules in the $\epsilon$-erosion neighborhood.
\begin{lemma}
\label{lem_refins_are_ers}
Let $M$ be a direct sum of indecomposable modules, and let $\epsilon\geq 0$.
\begin{itemize}
\item[(i)] $\Er_\epsilon(M)$ is an $\epsilon$-refinement of $M$.
\item[(ii)] Suppose $R$ is an $\epsilon$-refinement of $M$.
Then $R\in \EN_\epsilon(M)$.
\end{itemize}
\end{lemma}

\begin{proof}
(i)
This is a direct consequence of \cref{lem_er_behaved} (ii).

(ii)
Write $M = \bigoplus_{i\in \Gamma} M_i$ with each $M_i$ indecomposable.
Then $R\cong \bigoplus _{i\in \Gamma} M'_i/M''_i$, where $M'_i/M''_i \in \oEN_\epsilon(M_i)$.
We get $R\cong \bigoplus_{i\in \Gamma} M'_i/\bigoplus_{i\in \Gamma} M''_i \in \EN_\epsilon(M)$.
\end{proof}

We will now define the interleaving distance, which was introduced for one-parameter modules by Chazal et al. \cite{chazal2009proximity} and for multiparameter modules by Lesnick \cite{lesnick2015theory}.
Corollary 5.6 in the latter paper shows that the interleaving distance is in a certain sense the largest stable distance.
This is an argument for considering it the ``true'' distance between modules, and, together with the ubiquity of the interleaving distance in the persistence literature, is why we choose to study the properties of erosion neighborhoods in the context of interleavings in what follows.
\begin{definition}
\label{def_interleaving}
An $\epsilon$-interleaving between two modules $M$ and $N$ is a pair $\phi\colon M\to N(\epsilon)$ and $\psi\colon N\to M(\epsilon)$ of morphisms such that $\psi(\epsilon)\circ \phi = M_{0\to 2\epsilon}$ and $\phi(\epsilon)\circ \psi = N_{0\to 2\epsilon}$.
The interleaving distance $d_I$ is defined by
\[
d_I(M,N) = \inf\{\epsilon\mid \exists \epsilon\text{-interleaving between } M \text{ and } N\}.
\]
\end{definition}

In \cref{rem_not_unique_er} and \cref{def_ER}, we pointed out that there is a whole set $\ER_\epsilon(M)$ of modules canonically isomorphic to $\Er_\epsilon(M)$.
The following lemma, which was discovered and proved with the help of Ulrich Bauer and Steve Oudot, shows that $\oEN_\epsilon(M)$ is the set of subquotients of $M$ that are squeezed between $M$ and $\ER_\epsilon(M)$.
Thus, it makes sense to think of $\oEN_\epsilon(M)$ (and $\EN_\epsilon(M)$) as containing exactly the modules ``between'' $M$ and some version of $\Er_\epsilon(M)$.
\begin{lemma}
\label{lem_squeeze}
Define the relation $\preceq$ by $M_1/M_2 \preceq M'_1/M'_2$ whenever $M'_2\subseteq M_2\subseteq M_1\subseteq M'_1$.
For modules $M_2\sse M_1$, we have $M_1/M_2\in \oEN_\epsilon(M)$ if and only if there exists $N\in \ER_\epsilon(M)$ such that $N\preceq M_1/M_2 \preceq M$.
(Here we identify $M$ with $M/0$.)
\end{lemma}
\begin{proof}
Recall that by \cref{def_ER}, $N\in \ER_\epsilon(M)$ means that $N = (\Img_\epsilon(M)+K)/K$, where $\Ker_\epsilon(M)\cap \Img_\epsilon(M) \subseteq K\subseteq \Ker_\epsilon(M)$.

We first prove ``if''.
Assuming $N\preceq M_1/M_2 \preceq M$, we have $\Img_\epsilon(M)+K\sse M_1\sse M$ and $M_2\sse K \sse \Ker_\epsilon(M)$, and $M_1/M_2\in \oEN_\epsilon(M)$ follows.

For ``only if'', suppose $M_1/M_2\in \oEN_\epsilon(M)$, which means that $\Img_\epsilon(M)\sse M_1\sse M$ and $M_2\sse \Ker_\epsilon(M)$.
Let $K = (\Ker_\epsilon(M)\cap \Img_\epsilon(M))+M_2$.
Then $\Ker_\epsilon(M)\cap \Img_\epsilon(M) \subseteq K\subseteq \Ker_\epsilon(M)$, so $(\Img_\epsilon(M)+K)/K\in \ER_\epsilon(M)$.
Moreover, $(\Img_\epsilon(M)+K)/K = (\Img_\epsilon(M)+M_2)/K \preceq M_1/M_2 \preceq M$.
\end{proof}
With this lemma in mind, it is tempting to guess that $N$ is in $\EN_\epsilon(M)$ if and only if $\Er_\epsilon(M)\leq N\leq M$.
But while it follows from the lemma that ``only if'' holds, ``if'' is far from true.
In fact, \cref{thm_erosion_containment_gives_int} shows that any module in $\EN_\epsilon(M)$ is $\epsilon$-interleaved with $M$, while \cref{lem_dI_larger_dE} (see \cref{def_erosion_dist}) shows that for any $c>0$, there are one-parameter modules $\Er_\epsilon(M)\leq N\leq M$ with $d_I(M,N)> c\epsilon$.

A series of results in the next section and the remainder of this one show that $\EN_\epsilon(M)$ is strongly related to the $\epsilon$-neighborhood of $M$ in the interleaving distance.
\cref{thm_erosion_containment_gives_int} shows that the $\epsilon$-erosion neighborhood is contained in the $\epsilon$-interleaving neighborhood of $M$, while \cref{thm_strong_equiv} (i) shows that if $M$ and $N$ are $\epsilon$-interleaved, then their $\epsilon$-erosion neighborhoods intersect.
A word of warning:
The $\epsilon$-erosion neighborhood of $M$ does \emph{not} contain all subquotients of $M$ within interleaving distance at most $\epsilon$, and neither does the $c\epsilon$-erosion neighborhood for any fixed $c$ in general, as we will show in \cref{lem_int_subq_not_in_er_nhood}.

Since our main focus is on decompositions of modules, we would like the $\epsilon$-erosion neighborhood of $M$ to contain essentially all the possible decompositions of modules close to $M$ in the interleaving distance.
We get this from \cref{lem_if_int_then_er}:
If $M$ and $N$ are $\epsilon$-interleaved, then $\Er_\epsilon(N)$, which is an $\epsilon$-refinement of $N$ by \cref{lem_refins_are_ers} (i), is in the $2\epsilon$-erosion neighborhood of $M$.
Thus, the $2\epsilon$-erosion neighborhood of $M$ contains what we can view as an $\epsilon$-approximation of any decomposition that is $\epsilon$-close to $M$ in the interleaving distance.

\begin{lemma}
\label{lem_if_int_then_er}
Suppose $M$ and $N$ are $\epsilon$-interleaved.
Then $\Er_\epsilon(N)\in \EN_{2\epsilon}(M)$.
\end{lemma}
\begin{proof}
Let $\phi(\epsilon)$ and $\psi$ be $\epsilon$-interleaving morphisms, where $\phi\colon N(-\epsilon)\to M$ and $\psi\colon M\to N(\epsilon)$.
Let
\[
N(-\epsilon) \xrightarrow{\ol \phi} \img(\phi)/(\ker(\psi)\cap \img(\phi)) \xrightarrow{\ol \psi} N(\epsilon)
\]
be the morphisms induced by $\phi$ and $\psi$.
Since $\ol\psi \circ \ol\phi = \psi\circ \phi = N_{-\epsilon \to \epsilon}$, and $\ol \phi$ is an epimorphism and $\ol \psi$ a monomorphism, the middle term is isomorphic to $\img(N_{-\epsilon \to \epsilon})$, which is isomorphic to $\Er_\epsilon(N)$.

To show that $\img(\phi)/(\ker(\psi)\cap \img(\phi))\in \oEN_{2\epsilon}(M)$, first note that
\[
\img(\phi)\supseteq \img(\phi\circ \psi(-2\epsilon)) = \img(M_{-2\epsilon\to 0}) = \Img_{2\epsilon}(M).
\]
Secondly,
\[
\ker(\psi) \subseteq \ker(\phi(2\epsilon)\circ \psi) = \ker(M_{0\to 2\epsilon}) = \Ker_{2\epsilon}(M).
\]
This concludes the proof.
\end{proof}

\begin{theorem}
\label{thm_erosion_containment_gives_int}
Let $N\in \EN_\epsilon(M)$.
Then $M$ and $N$ are $\epsilon$-interleaved.
\end{theorem}
\begin{proof}
Possibly replacing $N$ with an isomorphic module, we can assume $N= M_1/M_2$, where $\Img_\epsilon(M)\sse M_1\sse M$ and $M_2\subseteq \Ker_\epsilon(M)$.
Let $\phi\colon M \to N(\epsilon)$ and $\psi\colon N\to M(\epsilon)$ be the morphisms induced by $M_{0\to \epsilon}$.
The former is well-defined because $\img M_{0\to \epsilon} = \Img_\epsilon(M)(\epsilon)\sse M_1(\epsilon)$, and the latter is well-defined because $M_2\sse \Ker_\epsilon(M) = \ker M_{0\to \epsilon}$.

We get that both $\psi(\epsilon)\circ \phi\colon M \to M(2\epsilon)$ and $\phi(\epsilon)\circ \psi\colon N \to N(2\epsilon)$ are induced by $M_{0\to \epsilon}(\epsilon) \circ M_{0\to \epsilon} = M_{0\to 2\epsilon}$.
Thus, $\psi(\epsilon)\circ \phi = M_{0\to 2\epsilon}$ and $\phi(\epsilon)\circ \psi = N_{0\to 2\epsilon}$.
\end{proof}

\begin{corollary}
\label{cor_if_er_then_int}
Suppose $\Er_{\epsilon}(N)\in \EN_{2\epsilon}(M)$.
Then $M$ and $N$ are $3\epsilon$-interleaved.
\end{corollary}
\begin{proof}
By \cref{thm_erosion_containment_gives_int}, $\Er_\epsilon(N)$ is $2\epsilon$-interleaved with $M$ and $\epsilon$-interleaved with $N$.
By composing interleavings, we obtain a $3\epsilon$-interleaving between $M$ and $N$.
\end{proof}

\subsection{Examples, part I}

We close the section with an example showing that \cref{cor_if_er_then_int} cannot be strengthened, namely \cref{lem_in_er_nhood_not_interl} (i).
Part (ii) of the lemma will be used in the next section.

\begin{definition}
\label{def_interval}
An \textbf{interval} is a subset $I$ of $\Pb$ such that
\begin{itemize}
\item if $p\leq q\leq r$ and $p,r\in I$, then $q\in I$, and
\item for all $p,q\in I$, there are points $p= p_1, p_2,\dots, p_\ell = q$ in $I$ for some $\ell$ such that $p_1\geq p_2\leq p_3 \geq \dots \leq p_\ell$.
\end{itemize}
For any interval $I$, we define its associated \textbf{interval module}, which we also denote by $I$, by letting $I_p = k$ for $p\in I$ and $I_p = 0$ for $p\neq I$, and $I_{p\to q} = \id_k$ for all $p\leq q\in I$.
\end{definition}
Note that this is different from another definition of interval that is also used, namely as a set of the form $\{q\in \Pb\mid p\leq q\leq r \}$ for some $p,r\in \Pb$.
To avoid notational overload, Blanchette et al. \cite{blanchette2021homological} introduce the word \emph{spread}, though these do not quite correspond to our notion of interval, as they do not enforce the connectedness condition in the second bullet point of \cref{def_interval}.

We refer to any module isomorphic to the interval module associated to an interval as an interval module.
Given an interval module, we will often take for granted that it is given on the form as described in the definition, and as the notation suggests, we will sometimes talk about $I$ with only the context making it clear if we are talking about the interval $I$ or its associated interval module.

\begin{lemma}
\label{lem_in_er_nhood_not_interl}
(i) There are pfd $2$-parameter modules $M$ and $N$ such that $\Er_1(M)\in \EN_2(N)$ and $\Er_1(N)\in \EN_2(M)$, and $M$ and $N$ are not $\delta$-interleaved for any $\delta<3$.

(ii) There are pfd $2$-parameter modules $M$ and $N$ such that $\Er_1(M)\in \EN_1(N)$, and $M$ and $N$ are not $\delta$-interleaved for $\delta<2$.
\end{lemma}
\begin{proof}
For $p\in \R^2$, let $\la p\ra = \{q\in \R^2 \mid q\geq p\}$.
Let $M$ be the $2$-parameter interval module supported on $\la (2,0) \ra \cup \la (0,2) \ra$.
At every point $p$ in the support of $N' \coloneqq M(-1) \oplus M(-1)$, we have $N'_p = k^2$.
Let $\epsilon\in [1,2]$.
Let $N_1$ be the submodule of $N'$ generated by $(1,0) \in N'_{(3,1)}$ and $(0,1) \in N'_{(1,3)}$, let $N_2$ be the submodule of $N'$ generated by $(1,1)\in N'_{(3+\epsilon, 3+\epsilon)}$, and let $N = N_1/N_2$.
See \cref{fig_ex_tight}.
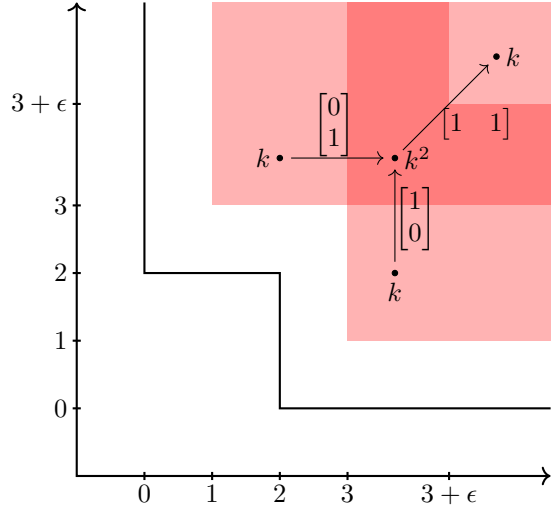
\begin{figure}
\centering
\begin{tikzpicture}[scale=.9]
\begin{scope}[xshift=-1cm, yshift=-1cm]
\draw[thick,<->] (0,7) to (0,0) to (7,0);
\draw[thick] (-.06,1) to (.06,1);
\node[left] at (0,1){$0$};
\draw[thick] (-.06,2) to (.06,2);
\node[left] at (0,2){$1$};
\draw[thick] (-.06,3) to (.06,3);
\node[left] at (0,3){$2$};
\draw[thick] (-.06,4) to (.06,4);
\node[left] at (0,4){$3$};
\draw[thick] (-.06,5.5) to (.06,5.5);
\node[left] at (0,5.5){$3+\epsilon$};
\draw[thick] (1,-.06) to (1,.06);
\node[below] at (1,0){$0$};
\draw[thick] (2,-.06) to (2,.06);
\node[below] at (2,0){$1$};
\draw[thick] (3,-.06) to (3,.06);
\node[below] at (3,0){$2$};
\draw[thick] (4,-.06) to (4,.06);
\node[below] at (4,0){$3$};
\draw[thick] (5.5,-.06) to (5.5,.06);
\node[below] at (5.5,0){$3+\epsilon$};
\end{scope}
\draw[thick] (6,0) to (2,0) to (2,2) to (0,2) to (0,6);
\fill[red, opacity=0.3] (3,1) rectangle (6,6);
\fill[red, opacity=0.3] (1,3) rectangle (6,6);
\fill[white] (4.5,4.5) rectangle (6,6);
\fill[red, opacity=0.3] (4.5,4.5) rectangle (6,6);
\coordinate (a) at (3.7,2);
\coordinate (b) at (2,3.7);
\coordinate (c) at (3.7,3.7);
\coordinate (d) at (5.2,5.2);
\node[below] at (a){$k$};
\draw[color=black,fill=black] (a) circle (.04);
\node[left] at (b){$k$};
\draw[color=black,fill=black] (b) circle (.04);
\node[right] at (c){$k^2$};
\draw[color=black,fill=black] (c) circle (.04);
\node[right] at (d){$k$};
\draw[color=black,fill=black] (d) circle (.04);
\draw[->, shorten <=.15cm, shorten >=.15cm] (a) to (c);
\draw[->, shorten <=.15cm, shorten >=.15cm] (b) to (c);
\draw[->, shorten <=.15cm, shorten >=.15cm] (c) to (d);
\node at (4,2.8){$\begin{bmatrix}
	1\\
	0
	\end{bmatrix}$};
\node at (2.8,4.2){$\begin{bmatrix}
	0\\
	1
	\end{bmatrix}$};
\node at (5.2,4.2){$\begin{bmatrix}
	1&-1
	\end{bmatrix}$};
\end{tikzpicture}
\caption{The black curve is the boundary of the support of $M$.
A module isomorphic to $N$ is drawn in red, with the shade illustrating the dimension in the four constant regions of the support.}
\label{fig_ex_tight}
\end{figure}

We will show three claims:
\begin{itemize}
\item[(a)] $M$ and $N$ are not $\delta$-interleaved for $\delta< 1+ \epsilon$,
\item[(b)] $\Er_1(M)\in \EN_\epsilon(N)$,
\item[(c)] if $\epsilon=2$, then $\Er_1(N) \cong \Er_2(M)$.
\end{itemize}
Now, (i) follows from (a), (b) and (c) with $\epsilon = 2$, and (ii) follows from (a) and (b) with $\epsilon=1$.

(a)
Suppose $\phi\colon M \to N(\delta)$ and $\psi\colon N \to M(\delta)$ form a $\delta$-interleaving for $\delta<1+\epsilon$.
Let $g$ be a generator of $M_{(2,2)}$.
Since $g\in \img M_{(2,0) \to (2,2)} \cap \img M_{(0,2) \to (2,2)}$, we must have $\phi(g) \in \img N_{(2+\delta,\delta) \to (2+\delta,2+\delta)} \cap \img N_{(\delta,2+\delta) \to (2+\delta,2+\delta)}$.
This intersection is zero, so $\phi(g) = 0$.
But $M_{0\to 2\delta}$ is injective, so $\phi$ must be injective, which contradicts $\phi(g) = 0$.
Thus, there is no $\delta$-interleaving between $M$ and $N$ for $\delta<1+\epsilon$.

(b)
Let $Q$ be the submodule of $N$ generated by $(1,1) \in N_{(3,3)}$.
We have $Q_{0\to \epsilon} = 0$, so $Q\in \Ker_\epsilon(N)$ and therefore $N/Q \in \EN_\epsilon(N)$.
But $N/Q$ is an interval module with the same support as $\Img_1(M)$, so $N/Q \cong \Img_1(M)$, which is isomorphic to $\Er_1(M)$, since $\Ker_1(M) = 0$

(c)
Let $\epsilon=2$.
Then $\Ker_1(N)$ is the submodule of $N$ generated by $(1,1)\in N_{(4,4)}$, and $\Img_1(N)$ is the submodule generated by $(1,0)\in N_{(4,2)}$ and $(0,1)\in N_{(2,4)}$.
It follows that $\Er_1(N)$ is an interval module supported on $\la (4,2)\ra \cup \la (2,4)\ra$, which is also true for $\Img_2(M)$.
We have $\Ker_2(M) = 0$, so $\Er_2(M) \cong \Img_2(M) \cong \Er_1(N)$.
\end{proof}

\section{Distances}
\label{sec_red_distances}

Distances provide a convenient language for comparing different notions of closeness between modules.
We have claimed that interleavings and erosion neighborhoods give rise to similar notions of neighborhoods of modules, and we will make this precise by showing that the distances $d_I$ (\cref{def_interleaving}) and $d_{EN}$ (\cref{def_str_er_dist}) they give rise to only differ by a multiplicative factor of at most two (\cref{thm_strong_equiv}), and are equal for one-parameter modules (\cref{thm_EN=I}).

We also discuss other versions of erosion distances (\cref{def_erosion_dist}), and observe in \cref{lem_dI_larger_dE} that another distance $\tilde d_E$ that at first glance looks similar to $d_{EN}$ in fact behaves very differently.
This lemma and \cref{lem_int_subq_not_in_er_nhood} (which we discussed in the beginning of the previous section) highlight traps one should avoid when working with erosion neighborhoods, subquotients and interleavings.
Our opinion is that $\EN_\epsilon(M)$ and $d_{EN}$ are the principled ways of defining a ``neighborhood of quotients'' and associated distance, and that one runs into these traps only when introducing deceptively similar, but less natural alternative definitions, like that of $\tilde d_E$.

In \cref{subsec_prunings}, we move on to the difficult question of how to define a good decomposition distance for multiparameter modules.
We comment that the bottleneck distance $d_B$ and a function $f_R$ defined using common refinements are not the answer due to instability and failure of the triangle inequality, respectively, and present an serious obstacle to defining a distance that depends on decompositions in a meaningful way without being extremely unstable.
We introduce prunings and an associated pruning distance that demonstrates how one might get around the obstacle.
We then show in \cref{sec_f_R_distance} that for interval decomposable modules, $f_R$ is actually a distance, and argue that it is a very well-behaved decomposition distance in this setting.

\begin{definition}
\label{def_distance}
A \emph{distance} $d$ on persistence modules associates a number $d(M,N)\in [0,\infty]$ to any pair $(M,N)$ of modules in a way such that
\begin{itemize}
\item $d(M,M) = 0$,
\item $d(M,N) = d(N,M)$ (symmetry), and
\item $d(M,Q) \leq d(M,N) + d(N,Q)$ (triangle inequality).
\end{itemize}
\end{definition}
This is also called an extended pseudometric: ``extended'' because the distance might be infinite and ``pseudo'' because we allow distance zero between different modules.
If $d$ and $d'$ are two distances and $d(M,N)\leq d'(M,N)$ hold for all modules $M$ and $N$, we write $d\leq d'$.
For any distance $d$ and constant $c\geq 0$, $cd$ is the distance defined by $(cd)(M,N) = c(d(M,N))$.

As mentioned, notions of erosion (distances) appear in several papers \cite{edelsbrunner2011stability, frosini2013stable, patel2018generalized, puuska2020erosion}, though the definition of erosion we use appears to be folklore.
To our knowledge, erosion makes its first appearance in the multiparameter persistence literature in \cite{frosini2013stable}.
In that paper, Frosini considers persistence modules valued in abelian groups instead of vector spaces, but adapting this to our setting is straightforward and gives the first of the following two distances, where we have exploited the definition of $\Er_\epsilon(M)$ to give a concise formulation.
Recall that $\leq$ denotes the subquotient relation between modules.
\begin{definition}
\label{def_erosion_dist}
For two modules $M$ and $N$, write $M\leq^* N$ if $\dim(M_p) \leq \dim(N_p)$ for all $p\in \Pb$.
We define the distance $d_E$ by
\[
d_E(M,N) = \inf\{\epsilon\mid \Er_\epsilon(M)\leq^* N \text{ and }\Er_\epsilon(N)\leq^* M\},
\]
and the distance $\tilde d_E$ by
\[
\tilde d_E(M,N) = \inf\{\epsilon\mid \Er_\epsilon(M)\leq N \text{ and }\Er_\epsilon(N)\leq M\}.
\]
\end{definition}
Frosini calls the first distance $d_T$ and in fact does not mention the word ``erosion''.
We use the notation $d_E$ (``E'' for ``erosion'') to be consistent with Puuska \cite{puuska2020erosion}, who gives a unified view of the erosion distance of Patel \cite{patel2018generalized} and the distance $d_T$.

The second distance, $\tilde d_E$, does not appear in the literature to our knowledge (though it might be known to the multipersistence community), but is a natural alternative to the erosion distance given the definition of $\Er_\epsilon(M)$.
To see that $\tilde d_E$ satisfies the triangle equality, observe that if $\Er_\epsilon(M)\leq N$ and $\Er_\delta(N)\leq Q$, then we get $\Er_{\epsilon+\delta}(M) \cong \Er_\delta(\Er_\epsilon(M)) \leq \Er_\delta(N) \leq Q$, where the isomorphism comes from \cref{lem_er_behaved} (iii), and the first ``$\leq$'' from $\Er_\epsilon(M)\leq N$ and \cref{lem_er_behaved} (i).

\subsection{The erosion neighborhood distance}

To discuss the properties of $\epsilon$-erosion neighborhoods and in particular how they behave with respect to the interleaving distance, we introduce another distance, which we call the \emph{erosion neighborhood distance}.

Even though $\epsilon$-erosion neighborhoods are not strictly speaking sets, we write $Q\in \EN_\epsilon(M) \cap \EN_\delta(N)$ when we mean $Q\in \EN_\epsilon(M)$ and $Q\in \EN_\delta(N)$.
\begin{definition}
\label{def_str_er_dist}
We define the erosion neighborhood distance $d_{EN}$ by
\[
d_{EN}(M,N) = \inf\{\epsilon\mid \exists Q\in \EN_\epsilon(M) \cap \EN_\epsilon(N)\}.
\]
\end{definition}
If $Q\in \EN_\epsilon(M) \cap \EN_\epsilon(N)$, then $\Er_\epsilon(M)\leq Q\leq N$, and similarly $\Er_\epsilon(N)\leq M$.
Thus, $\tilde d_E \leq d_{EN}$.
Furthermore, $M\leq N$ implies $M\leq^* N$, so $d_E\leq \tilde d_E$.
That is, we have $d_E\leq \tilde d_E \leq d_{EN}$.

One can ask if the opposite inequalities hold up to a constant.
That is, do we have $\tilde d_E\leq c d_E$ or $d_{EN} \leq c\tilde d_E$ for some $c$?
The answer is no in both cases.
In the first case, pick one-parameter modules $M$ and $N$ with $B(M) = \{[0,2)\}$ and $B(N) = \{[0,1),[1,2)\}$.
In this case, $d_E(M,N)=0$, while one can show that $\tilde d_E(M,N)\neq 0$.

That $d_{EN} \leq c\tilde d_E$ does not hold for any $c$ is harder to prove.
By \cref{lem_er_behaved} (i) and (iii), $\Er_\epsilon(M)\leq N$ and $\Er_\epsilon(N)\leq M$ together imply $\Er_{2\epsilon}(N)\leq \Er_\epsilon(M)\leq N$, which looks similar to $\Er_\epsilon(M)\in \EN_{2\epsilon}(N)$.
Based on this, one might be tempted to guess that one can use $Q= \Er_\epsilon(M)$ in the definition of $d_{EN}$ to prove that $d_{EN}\leq 2\tilde d_E$, but we will show in \cref{lem_dI_larger_dE} (using \cref{thm_strong_equiv}) that the two distances can actually differ by an arbitrarily large multiplicative constant, disproving $d_{EN} \leq c\tilde d_E$ for any $c$.

The erosion neighborhood distance is somewhat similar to the distance $d$ defined by Scolamiero et al. \cite[Def. 8.6]{scolamiero2017multidimensional} in that they both involve morphisms between modules with small kernels and cokernels.
We do not study how the two distances are related, but it is claimed (without proof) in \cite{scolamiero2017multidimensional} that $d_I\leq d\leq 6d_I$, which, put together with \cref{thm_strong_equiv}, implies that $d_{EN}\leq d\leq 12d_{EN}$.

We will now show that $d_{EN}$ is indeed a distance.
Symmetry and $d_{EN}(M,M) = 0$ are straightforward, which leaves only the triangle inequality to be checked.

\begin{lemma}
\label{lem_erosion_containment2}
Let $\epsilon,\delta\geq 0$ and suppose $N\in \EN_\epsilon(M)$ and $Q\in \EN_\delta(N)$.
Then $Q\in \EN_{\epsilon+ \delta}(M)$.
\end{lemma}
\begin{proof}
There are submodules $M_2\sse M_1$ of $M$ with $N\cong M_1/M_2$ and
\[
M_2\sse \Ker_\epsilon(M),\quad \Img_\epsilon(M)\sse M_1,
\]
and there are submodules $N_2\sse N_1$ of $N$ with $Q\cong N_1/N_2$ and
\[
N_2\subseteq \Ker_\delta(N),\quad \Img_\delta(N)\subseteq N_1.
\]
Let $M'_1$ and $M'_2$ be the inverse images of $N_1$ and $N_2$ under the composition
\[
\phi\colon M_1 \thra M_1/M_2 \xra{\sim} N,
\]
where the first map is the canonical projection and the second is an isomorphism.
We get $\Img_\delta(M_1)\sse M'_1$ by \cref{lem_im_to_im} (i).
Since $\Img_\epsilon(M)\sse M_1$, we get
\[
\Img_{\epsilon+ \delta}(M) = \Img_\delta(\Img_\epsilon(M)) \sse \Img_\delta(M_1) \sse M'_1.
\]
From $N_2\sse \Ker_\delta(N)$, we get $0 = N_{0\to \delta}(\delta)\circ \phi(M'_2) = \phi(\delta) \circ (M_1)_{0\to \delta}(M'_2)$, so $\Img_\delta(M'_2)\sse M_2$.
Together with $M_2\sse \Ker_\epsilon(M)$, this gives $M'_2\sse \ker M_{-\epsilon- \delta\to 0} = \Ker_{\epsilon+ \delta}(M)$.
Since $M'_1/M'_2\cong N_1/N_2$, the lemma follows.
\end{proof}

\begin{lemma}
\label{lem_str_er_triangle}
The erosion neighborhood distance $d_{EN}$ satisfies the triangle inequality.
\end{lemma}

\begin{proof}
It is enough to show that if there are modules $M$, $N$, $P$, $Q_1$ and $Q_2$ with $Q_1\in \EN_\epsilon(M)\cap \EN_\epsilon(N)$ and $Q_2\in \EN_\delta(N)\cap \EN_\delta(P)$, then there is a module $Q_3\in \EN_{\epsilon+ \delta}(M)\cap \EN_{\epsilon+ \delta}(P)$.

We can assume $Q_1 = N_1/N'_1$ and $Q_2 = N_2/N'_2$, where
\begin{align*}
N'_1\sse \Ker_\epsilon(N), \quad &\Img_\epsilon(N)\sse N_1,\\
N'_2\sse \Ker_\delta(N), \quad &\Img_\delta(N)\sse N_2.
\end{align*}
Let $N_3 = N_1\cap N_2$ and $N'_3 = (N'_1\oplus N'_2)\cap N_1\cap N_2$, and let $Q_3 = N_3/N'_3$.
We will show that $Q_3\in \EN_\delta(Q_1)$, which by \cref{lem_erosion_containment2} gives $Q_3\in \EN_{\epsilon+ \delta}(M)$.
By symmetry, we also have $Q_3\in \EN_{\epsilon+ \delta}(P)$, which finishes the proof.

Let $\pi\colon N_1\to N_1/N'_1$ be the canonical projection.
Since $N'_1\cap N_3\sse N'_3$, $Q_3 \cong \pi(N_3)/\pi(N'_3)$.
Thus, it suffices to show $\Img_\delta(Q_1)\sse\pi(N_3)$ and $\pi(N'_3)\sse \Ker_\delta(Q_1)$.
We have
\[
\Img_\delta(N_1) \sse N_1\cap \Img_\delta(N)\sse N_3
\]
and
\[
N'_3\sse (N'_1\oplus N'_2)\cap N_1\sse (N'_1\oplus \Ker_\delta(N))\cap N_1 = N'_1\oplus \Ker_\delta(N_1).
\]
Applying $\pi$ to the former gives
\[
\Img_\delta(Q_1) = \Img_\delta(\pi(N_1)) = \pi(\Img_\delta(N_1)) \sse \pi(N_3).
\]
Applying $\pi$ to the latter gives
\[
\pi(N'_3)\sse \pi(N'_1\oplus \Ker_\delta(N_1)) = \pi(\Ker_\delta(N_1)) \sse \Ker_\delta(\pi(N_1)) = \Ker_\delta(Q_1).\qedhere
\]
\end{proof}

Checking the rest of the conditions of \cref{def_distance} is straightforward, so we get the following corollary.
\begin{corollary}
\label{cor_d_EN_distance}
The erosion neighborhood distance $d_{EN}$ is a distance.
\end{corollary}

Next, we show that the erosion neighborhood distance and the interleaving distance can only differ by a multiplicative factor of up to two.
Since $d_I$ is NP-hard to approximate up to a constant of less than $3$, this means that $d_{EN}$ is NP-hard to approximate up to a constant of less than $\frac 3 2$.
In particular, computing $d_{EN}$ exactly is NP-hard.\footnote{More precisely: given $M$, $N$ and $\epsilon$, it is NP-hard to decide if $d_{EN}(M,N)\leq \epsilon$.}

Occasionally, we will use the notation $x\in M$ for a module $M$ meaning that $x\in M_p$ for some point $p\in \Pb$ that is suppressed from the notation.
\begin{theorem}
\label{thm_strong_equiv}
For any modules $M$ and $N$, $d_{EN}(M,N)\leq d_I(M,N)\leq 2d_{EN}(M,N)$.
In particular,
\begin{itemize}
\item[(i)] if $M$ and $N$ are $\epsilon$-interleaved, then there is a module $Q\in \EN_\epsilon(M)\cap \EN_\epsilon(N)$, and
\item[(ii)] if there is a module $Q\in \EN_\epsilon(M)\cap \EN_\delta(N)$, then $M$ and $N$ are $(\epsilon+\delta)$-interleaved.
\end{itemize}
\end{theorem}
By \cref{lem_in_er_nhood_not_interl} (ii), the second inequality cannot be strengthened for $2$-parameter modules.
To see that the first inequality is tight for $d$-parameter modules for $d\geq 1$, consider for instance the one-parameter interval modules supported on $[0,\infty)$ and $[1,\infty)$, between which both distances are $1$.
Such an example can be constructed also for general $d$.
\begin{proof}
The first inequality follows from (i).
By \cref{thm_erosion_containment_gives_int}, $Q$ in (ii) is $\epsilon$-interleaved with $M$ and $\delta$-interleaved with $N$, which implies that $M$ and $N$ are $(\epsilon+\delta)$-interleaved.
Putting $\epsilon=\delta$, we get the second inequality.

Let us show (i).
Suppose $\phi\colon M\to N(\epsilon)$ and $\psi\colon N\to M(\epsilon)$ form an $\epsilon$-interleaving.
Define submodules $M_1,M_2\sse M$ and $N_1, N_2\sse N$ by
\begin{align*}
M_1 &= \Img_\epsilon(M) + \img \psi(-\epsilon), \quad M_2 = \Ker_\epsilon(M) \cap \ker \phi\cap M_1,\\
N_1 &= \Img_\epsilon(N) + \img \phi(-\epsilon), \quad N_2 = \Ker_\epsilon(N) \cap \ker \psi\cap N_1.
\end{align*}
Clearly, $M_1/M_2\in \EN_\epsilon(M)$ and $N_1/N_2\in \EN_\epsilon(N)$.
It remains to prove that $M_1/M_2 \cong N_1/N_2$.

Consider the morphisms $\mu\colon M(-\epsilon)\oplus N(-\epsilon) \to M$ and $\nu\colon M(-\epsilon)\oplus N(-\epsilon) \to N$ given by $M_{-\epsilon\to 0}$ (resp. $\phi(-\epsilon)$) on the first summand and $\psi(-\epsilon)$ (resp. $N_{-\epsilon\to 0}$) on the second.
We have $\img \mu = M_1$ and $\img \nu = N_1$.
Thus, $M_1/M_2 \cong (M(-\epsilon)\oplus N(-\epsilon))/\mu^{-1}(M_2)$, and $N_1/N_2 \cong (M(-\epsilon)\oplus N(-\epsilon))/\nu^{-1}(N_2)$, so to show $M_1/M_2 \cong N_1/N_2$, it suffices to show $\mu^{-1}(M_2) = \nu^{-1}(N_2)$.

An element $(x,y)\in M(-\epsilon)\oplus N(-\epsilon)$ is in $\mu^{-1}(\Ker_\epsilon(M))$ if and only if $M_{-\epsilon\to 0}(x) + \psi(-\epsilon)(y)$ is in $\Ker_\epsilon(M)$.
Equivalently, $M_{0\to \epsilon}(M_{-\epsilon\to 0}(x)) = M_{0\to \epsilon}(\psi(-\epsilon)(-y))$.
\vspace{.2cm}

\begin{tikzpicture}
\node at (-3,2){$M(-\epsilon)$};
\node at (-3,0){$N(-\epsilon)$};
\node at (0,2){$M$};
\node at (0,0){$N$};
\node at (3,2){$M(\epsilon)$};
\node at (3,0){$N(\epsilon)$};
\draw[->, shorten <=.7cm, shorten >=.4cm] (-3,2) to (0,2);
\node[above] at (-1.5,2){$M_{-\epsilon\to 0}$};
\draw[->, shorten <=.4cm, shorten >=.6cm] (0,2) to (3,2);
\node[above] at (1.5,2){$M_{0\to \epsilon}$};
\draw[->, shorten <=.7cm, shorten >=.4cm] (-3,0) to (0,0);
\node[below] at (-1.5,0){$N_{-\epsilon\to 0}$};
\draw[->, shorten <=.4cm, shorten >=.6cm] (0,0) to (3,0);
\node[below] at (1.5,0){$N_{0\to \epsilon}$};
\draw[->, shorten <=.4cm, shorten >=.4cm] (-3,2) to (0,0);
\node at (-2.8,1.4){$\phi(-\epsilon)$};
\draw[->, shorten <=.4cm, shorten >=.4cm] (-3,0) to (0,2);
\node at (-2.8,0.6){$\psi(-\epsilon)$};
\draw[->, shorten <=.4cm, shorten >=.4cm] (0,2) to (3,0);
\node at (.5,1.4){$\phi$};
\draw[->, shorten <=.4cm, shorten >=.4cm] (0,0) to (3,2);
\node at (.5,0.6){$\psi$};
\draw[color=white,fill=white] (-7,0) circle (.001);
\end{tikzpicture}
\vspace{.2cm}

Since $\phi$ and $\psi$ form an interleaving, the above diagram commutes.
This means that the equation just stated is equivalent to $\psi(\phi(-\epsilon)(x)) = \psi(N_{-\epsilon\to 0}(-y))$.
This, in turn, is equivalent to $\nu(x,y)\in \ker \psi$.
By symmetry, $\mu(x,y)\in \ker \phi$ is equivalent to $\nu(x,y)\in \Ker_\epsilon(N)$.
Putting the two together, $\mu(x,y)\in M_2$ holds if and only if $\nu(x,y)\in N_2$, so $\mu^{-1}(M_2) = \nu^{-1}(N_2)$.
\end{proof}

Call a module $R$ \emph{$\epsilon$-trivial} if $\Img_\epsilon(R)$ is the zero module.
Equivalently, $\Ker_\epsilon(R) = R$.
To prove our next theorem, we will need the following:
\begin{theorem}\cite[Theorem~6.1, special case]{bauer2015induced}
\label{thm_induced_matchings}
Suppose $A$ and $B$ are pfd $1$-parameter modules whose barcodes only contain intervals of the form $[a,b)$.
Then any morphism $\phi\colon X\to Y$ with $\epsilon$-trivial kernel and $\delta$-trivial cokernel induces a matching $\phi$ from $X$ to $Y$ that
\begin{itemize}
	\item[(i)] matches any interval in $B(X)$ of length more than $\epsilon$ and
	\item[(ii)] any interval in $B(Y)$ of length more than $\delta$,
\end{itemize}
and such that if $\sigma([a,b)) = [c,d)$, then
\begin{itemize}
	\item[(iii)] $b-\epsilon\leq d\leq b$ and
	\item[(iv)] $c\leq a\leq c+\delta$.
\end{itemize}
\end{theorem}

\begin{theorem}
\label{thm_EN=I}
Let $\epsilon\geq 0$.
Any pfd $1$-parameter modules $M$ and $N$ are $\epsilon$-interleaved if and only if there is a module $Q\in \EN_\epsilon(M)\cap \EN_\epsilon(N)$.
Thus, $d_{EN}(M,N) = d_I(M,N)$.
\end{theorem}

\begin{proof}
By \cref{thm_strong_equiv} (i), it suffices to show that if there is a module $Q\in \EN_\epsilon(M)\cap \EN_\epsilon(N)$, then there is an $\epsilon$-interleaving between $M$ and $N$.
Since $Q\in \EN_\epsilon(M)$, we can write $Q\cong M'/M''$, where $M''\sse M' \sse M$, and we have $M''\sse \Ker_\epsilon(M)$ and $\Img_\epsilon(M)\sse M'$.
Consider the morphisms $f\colon M\to M/M''$ and $g\colon M'/M'' \to M/M''$ induced by the identity morphism of $M$.
We have that $f$ is surjective and $g$ injective, so $\ker(g)$ and $\cok(f)$ are zero and therefore $0$-trivial.
Furthermore, $\ker(f)\cong M''$ and $\cok(g)\cong M/M'$ are $\epsilon$-trivial.
For simplicity, we will assume that any interval that appears is of the form $[a,b)$ for $a\in \R$ and $b\in \R\cup \{\infty\}$.
The more general case is dealt with carefully in \cite{bauer2015induced} and is not conceptually different.

\cref{thm_induced_matchings} gives us matchings $\sigma_f$ from $M$ to $M/M''$ and $\sigma_g$ from $M'/M''$ to $M/M''$ such that if $\sigma_f([a,b)) = [c,d)$, then $a=c$ and $b-\epsilon\leq d\leq b$, and such that if $\sigma_g([a,b)) = [c,d)$, then $c\leq a\leq c+\epsilon$ and $b =d$.
Moreover, $\sigma_f$ matches all the intervals in $B(M)$ with length more than $\epsilon$, and all the intervals in $B(M/M'')$; and $\sigma_g$ matches all the intervals in $B(M/M'')$ with length more than $\epsilon$, and all the intervals in $B(M'/M'')$.

Putting all of this together, $\sigma_g^{-1}\circ \sigma_f$ is a matching from $M$ to $M'/M''$ that matches all intervals in $B(M)$ of length at least $2\epsilon$, and all intervals in $B(M'/M'')$, and if $\sigma_g^{-1}\circ \sigma_f([a,b)) = [c,d)$, then $a\leq c \leq a+\epsilon$ and $b-\epsilon\leq d \leq b$.
Since $M'/M''\cong Q$, $B(M'/M'')=B(Q)$.

We can construct a similar matching $\gamma$ from $B(N)$ to $B(Q)$.
Then $\gamma\circ\sigma_g^{-1}\circ \sigma_f$ is an $\epsilon$-matching: Firstly, every interval of length at least $2\epsilon$ is matched.
Secondly, if $\gamma\circ\sigma_g^{-1}\circ \sigma_f([a,b)) = [a',b')$, then there is a $[c,d)$ with $\sigma_g^{-1}\circ \sigma_f([a,b)) = [c,d) = \gamma([a',b'))$.
Then $a\leq c \leq a+\epsilon$ and $a'\leq c \leq a'+\epsilon$, which gives $a-\epsilon\leq a'\leq a+\epsilon$.
Similarly, $b-\epsilon\leq b'\leq b+\epsilon$.
Any pair of interval modules that are matched by this matching is $\epsilon$-interleaved, and any unmatched interval module is $\epsilon$-interleaved with the zero module.
Thus, summing all these interleavings, we obtain an $\epsilon$-interleaving between $M$ and $N$.
\end{proof}

\subsection{Examples, part II}

The following lemma states that no result like \cref{thm_strong_equiv} holds for the erosion distance or $\tilde d_E$, not even for one-parameter modules.
\begin{lemma}
\label{lem_dI_larger_dE}
For any $c\in \R$, there are one-parameter pfd modules $M$ and $N$ with $d_I(M,N)\geq c$ and $\tilde d_E(M,N)\leq 1$.
\end{lemma}

\begin{proof}
It suffices to show the lemma under the assumption that $c$ is a positive integer.
For $i\geq 0$, let $A_i = [-i,0)$ and $B_i = [-i,i)$, where as before, we identify intervals and their associated modules.
Let $M = \bigoplus_{i=1}^{2c-1} B_i$, and let $N = A_{2c}\oplus M$.
Clearly, $M\leq N$, which gives $\Er_1(M)\leq N$.
Applying the definition of erosion, one can see that for any interval $I = [a,b)$, we have $\Er_1(I) = [a+1, b-1)$.
Using this, we get that
\begin{align*}
\Er_1(M) &= \bigoplus_{i=1}^{2c-2} B_i,\\
\Er_1(N) &= [2c+1,-1) \oplus \bigoplus_{i=1}^{2c-2} B_i.
\end{align*}
We also have that for intervals $I$ and $I'$, $I$ is a subquotient of $I'$ if and only if $I\sse I'$.
Thus, $[2c+1,-1) \leq B_{2c-1}$, and it follows that $\Er_1(N)\leq M$.
Thus, $\tilde d_E(M,N)\leq 1$.

It remains to show that $d_I(M,N)\geq c$.
By \cref{thm_ast}, it suffices to show that there is no $\epsilon$-matching between $M$ and $N$ for $\epsilon<c$.
Observe that $A_{2c}$ is too large to be left unmatched in an $\epsilon$-matching.
Moreover, for each $i\geq c$, the difference between the right endpoint of $A_{2c}$ and the right endpoint of $B_i$ is at least $c>\epsilon$, and the same holds for the left endpoints when $i\leq c$.
Thus, no interval in $B(M)$ can be matched with $A_{2c}$ in an $\epsilon$-matching, either.
\end{proof}

We finish off the section with an example not strictly related to distances.
By a slight twist on the modules in the proof of the previous lemma, we show (as promised in the previous section) that for any constant $c$, the $c\epsilon$-erosion neighborhood of a module $M$ does not in general contain all subquotients within interleaving distance $\epsilon$ of $M$.

\begin{lemma}
\label{lem_int_subq_not_in_er_nhood}
For any $c>0$, there are pfd one-parameter modules $M$ and $N$ such that $N\leq M$ and $M$ and $N$ are $1$-interleaved, but $N\notin \EN_c(M)$.
\end{lemma}

\begin{proof}
Like in the proof of the previous lemma, we can assume that $c$ is a positive integer, and we will use modules decomposing into modules of the form $A_i$ and $B_i$.
Let $M = A_c\oplus \bigoplus_{i=1}^{c+1} B_i$ and $N = A_{c+1}\oplus \bigoplus_{i=1}^c B_i$.
There is a $1$-interleaving between $A_c$ and $A_{c+1}$, and between $B_i$ and $B_{i-1}$ for all $i\in \{1,\dots,c+1\}$.
($B_0$ is empty and thus corresponds to the zero module.)
We can sum these interleavings to get a $1$-interleaving between $M$ and $N$.

$A_{c+1}$ is a quotient of $B_{c+1}$, the zero module is a quotient of $A_c$, and $B_i$ is a quotient of $B_i$ for all $i$.
Putting these together, we get that $N\leq M$.

We will now show that $N\notin \EN_\epsilon(M)$ for $\epsilon< c+1$.
Assuming $N\in \EN_\epsilon(M)$, we have submodules $M_1\supseteq M_2$ of $M$ with $M_1\supseteq \Img_\epsilon(M)$ and $M_2\sse \Ker_\epsilon(M)$, and $N\cong M_1/M_2$.

Since $N_{-c-1}\neq 0$, we must have $(M_1)_{-c-1}\neq 0$, which implies that $M_1$ contains the summand $B_{c+1}$ of $M$ as a submodule.
At the same time, $B_{c+1}\cap M_2\subseteq B_{c+1}\cap \Ker_\epsilon(M)$, the support of which does not intersect $[-c-1,c+1-\epsilon)\sse [-c-1,0]$.
This implies that the linear transformation $(M_1/M_2)_{-c-1\to 0}$ is nonzero.
But $N_{-c-1\to 0}$ is zero, so $N\ncong M_1/M_2$, which contradicts our assumptions.
\end{proof}

\subsection{Prunings and the pruning distance}
\label{subsec_prunings}

In this subsection, we discuss how to define a well-behaved ``decomposition distance''; that is, a distance that measures how different the decompositions of two modules are.
We observe how a couple of natural attempts at defining such a distance fail, show that these failures are due to a serious more general obstacle (\cref{thm_uli_luis_cor}), and then explain how this obstacle might be circumvented, introducing the pruning distance (\cref{def_pruning_dist}) as a concrete suggestion of a well-behaved decomposition distance.

The most obvious candidate for a decomposition distance is the bottleneck distance $d_B$ (\cref{def_matching}).
Unfortunately, the bottleneck distance is very unstable:
In the example in \cref{fig_large_bottleneck}, the three modules are pairwise $0.03$-interleaved, but the bottleneck distance between any pair of them is $0.4$.
By making the ``bridges'' connecting the squares smaller, the interleaving distance can be made arbitrarily small while $d_B$ stays equal to $0.4$.
(Moreover, we can fix $Q$ and let every module have pointwise dimension at most one.)

Fixing such instability issues is the main motivation for our definition of $\epsilon$-refinements (\cref{def_refinement}), so a natural next attempt at defining a good decomposition distance is the following:
\begin{definition}
\label{def_f_R}
For pfd modules $M$ and $N$, let
\begin{equation}
f_R(M,N) = \inf\{\epsilon \mid \exists Q \text{ that is an }\epsilon\text{-refinement of both } M \text{ and }N\}.
\end{equation}
\end{definition}
That is, we define $f_R$ using common $\epsilon$-refinements in the same way $d_I$ and $d_B$ are defined using $\epsilon$-interleavings and $\epsilon$-matchings, respectively.
Sadly, also this definition has an issue:
$f_R$ does not satisfy the triangle inequality, and is therefore not a distance.\footnote{We omit a rigorous proof of this fact, but here is an outline:
$f_R$ agrees with the distance $d_{EN}$ on pairs of indecomposable modules, but $f_R\neq d_{EN}$ in general, and $f_R$ and $d_{EN}$ have stability properties by \cref{cor_f_R} and \cref{thm_strong_equiv}, respectively.
It then follows from \cref{thm_uli_luis_cor} that $f_R$ is not a distance.}

As a matter of fact, the instability of $d_B$ and the failure of $f_R$ to satisfy the triangle inequality are due to a fundamental obstacle to defining good decomposition distances, which is expressed by \cref{thm_uli_luis_cor} below.
This theorem is a corollary of \cite[Theorem A]{bauer2022generic}, which is phrased in terms of \emph{finitely presentable} modules.
For $p\in \Pb$, let $\la p\ra = \{q\in \Pb \mid q\geq p\}$.
\begin{definition}
A \textbf{finitely presentable} module is a module isomorphic to the cokernel of a morphism $M\to N$, where $B(M)$ and $B(N)$ are finite and only contain interval modules whose underlying intervals are of the form $\la p\ra$.
\end{definition}
The theorem says that under a technical condition (ii), any distance on finitely presentable modules is fully determined by the values it takes on pairs of indecomposable modules.
The condition (ii) can be interpreted as saying that the distances $d$ and $d'$ are not extremely unstable:
If (ii) fails, then there are $M$ and $M_1,M_2, \dots$ with $\lim_{i\to \infty} d_I(M,M_i) = 0$, but $\lim_{i\to \infty} d(M,M_i) \neq 0$ (or the same for $d'$).
This means that there is a constant $\delta$ such that for any $\epsilon>0$, there is a module $M_\epsilon$ with $d_I(M,M_\epsilon)<\epsilon$ and $d(M,M_\epsilon)>\delta$.
As a consequence, there is no $c_M\in \R$ such that $d(M,N)\leq c_Md_I(M,N)$ for all $N$.
In particular, $d_B$ fails condition (ii).
\begin{theorem}[{Follows from \cite[Theorem~A]{bauer2022generic}}]
\label{thm_uli_luis_cor}
Suppose $d$ and $d'$ are distances on $m$-parameter modules for $m\geq 2$ such that
\begin{itemize}
\item[(i)] for all indecomposable $M$ and $N$, $d(M,N) = d'(M,N)$,
\item[(ii)] for any module $M$ and indecomposable $M_1, M_2,\dots$ such that $\lim_{i\to \infty} d_I(M,M_i) = 0$, we have
\[
\lim_{i\to \infty} d(M,M_i) = \lim_{i\to \infty} d'(M,M_i) = 0.
\]
\end{itemize}
Then $d(M,N)=d'(M,N)$ for any finitely presentable modules $M$ and $N$.
\end{theorem}

\begin{proof}
Let $M$ and $N$ be finitely presentable modules.
By \cite[Theorem A]{bauer2022generic}, there are sequences $M_1, M_2,\dots$ and $N_1, N_2,\dots$ of finitely presentable indecomposable modules such that
\[
\lim_{i\to \infty} d_I(M,M_i) = \lim_{i\to \infty} d_I(N,N_i) = 0.
\]
By (ii), we also have
\begin{align*}
\lim_{i\to \infty} d(M,M_i) &= \lim_{i\to \infty} d(N,N_i) = 
\lim_{i\to \infty} d'(M,M_i) = \lim_{i\to \infty} d'(N,N_i) = 0.
\end{align*}
By the triangle inequality, this gives $d(M,N) = \lim_{i\to \infty} d(M_i,N_i)$ and $d'(M,N) = \lim_{i\to \infty} d'(M_i,N_i)$.
But by (i), $d(M_i,N_i)=d'(M_i,N_i)$ for all $i$, so $d(M,N) = d'(M,N)$.
\end{proof}
This theorem demonstrates the difficulty of defining a well-behaved distance that depends on decompositions:
We want to work with stable distances, but imposing even a very weak stability statement like (ii) seems to remove decomposable modules completely from the picture, since the distance is then completely determined by the values it takes on pairs of indecomposable modules.
This explains why neither $d_B$ nor $f_R$ is a stable distance: For indecomposable $M$ and $N$, we have $d_B(M,N) = d_I(M,N)$, but we have $d_B\neq d_I$ for general modules, and $d_I$ is stable.
Thus, by \cref{thm_uli_luis_cor}, $d_B$ is either not a distance, or it is extremely unstable.
By the same argument, replacing $d_B$ by $f_R$ and $d_I$ by $d_{EN}$, we get that $f_R$ is not a stable distance, either.
The moral is that the reason $d_I$ and $d_{EN}$ are stable distances and $d_B$ and $f_R$ are not, is that $d_B$ and $f_R$ depend on decompositions, while $d_I$ and $d_{EN}$ do not.

Thus, \cref{thm_uli_luis_cor} paints a bleak picture for decomposition distances: it almost seems like extreme instability is an inevitable property of any distance depending on decompositions.
We conjecture that the situation is less hopeless than it seems.
The ``loophole'' that we exploit is that even though any stable distance is determined by the values it takes on pairs of indecomposable modules $M$ and $N$, there is nothing stopping us from defining a distance that depends on the decompositions of modules in a neighborhood of $M$ and $N$, not just the decompositions of $M$ and $N$ themselves.
To this end, we define \emph{prunings} below, and use these to define the \emph{pruning distance}.
The $\epsilon$-pruning of a module plays a crucial part in our main stability result \cref{thm_main}, where we show that the $\epsilon$-pruning of a module $M$ is a refinement of all the modules in a neighborhood of $M$.
Thus, one can view the $\epsilon$-pruning of $M$ as an approximate decomposition of $M$ that is finer than all decompositions of modules close to $M$.
As a consequence, the $\epsilon$-pruning contains information about all decompositions in a neighborhood of $M$.

The definition of pruning is quite technical, and understanding the full motivation for the definition requires a careful reading of \cref{sec_stability_theorem}.
For now, the motivation we give is that the definition introduces submodules $I$ and $K$ and makes sure that they are well-behaved with regards to morphisms $f\colon M\to M(2\epsilon)$, which we view as approximate endomorphisms.
That an object encoding (approximate) decompositions is defined in terms of (approximate) endomorphisms should not be surprising, since there is a strong relation between decompositions and endomorphisms:
Any decomposition $M=M'\oplus M''$ gives rise to idempotent endomorphisms by projecting onto $M'$ and $M''$, while the Fitting lemma \cite[Sec.~3.4]{jacobson2012basic} shows how to obtain a decomposition from an endomorphism under suitable assumptions.
\begin{definition}
\label{def_pruning}
Let $M$ be a module.
Let $I$ be the largest submodule of $M$ such that for any morphism $f\colon M\to M(2\epsilon)$, $f(I)\sse M_{0\to 2\epsilon}(I)$.
Let $K$ be the smallest submodule of $I$ such that for any morphism $f\colon M\to M(2\epsilon)$, $I_{0\to 2\epsilon}^{-1}(f(K))\sse K$.
We define the \textbf{$\epsilon$-pruning pair} of $M$ to be $(I,K)$, and the \textbf{$\epsilon$-pruning} of $M$ to be $\Pru_\epsilon(M)\coloneqq (I/K)(-\epsilon)$.
\end{definition}
We show that the $\epsilon$-pruning pair is well-defined in \cref{lem_pruning_pair}.
We have that the $0$-pruning pair of a module $M$ is $(M,0)$, and the $0$-pruning of $M$ is $M/0\cong M$.

\begin{example}
\label{ex_pruning}
We construct an almost-decomposable module and describe its $\epsilon$-pruning for an appropriate $\epsilon$, illustrating how pruning a module can split it into smaller pieces.
The construction is illustrated in \cref{fig_pruning_ex}.
Let $0\leq \epsilon < \frac 12$, and let $I_1 = [1,3]\times [1-2\epsilon,3]$, $I_2 = [1-2\epsilon,3]\times [1,3]$ and $I_3 = [3-2\epsilon,3]^2$.
For $p\in [1,3]^2$, we have $(I_1\oplus I_2)_p=k^2$ with the two canonical basis elements of $k^2$ coming from $I_1$ and $I_2$, respectively.
Let $\phi \colon I_3\to I_1\oplus I_2$ be the map with $\phi_p\colon k\to k^2$ being given by the matrix $\begin{bmatrix}
1\\
1
\end{bmatrix}$ for $p\in I_3$, and let $M = (I_1\oplus I_2)/\img \phi$.
One can show that $M$ is indecomposable.

Let us find the $\epsilon$-pruning pair $(I,K)$ of $M$.
There is a morphism $f\colon M \to M(2\epsilon)$ with $f_p = \begin{bmatrix}
0\\
1
\end{bmatrix}$ and $f_q = \begin{bmatrix}
1\\
0
\end{bmatrix}$ for $p\in J_1\coloneqq [1,3-2\epsilon]\times [1-2\epsilon,1]$ and $q\in J_2\coloneqq [1-2\epsilon,1]\times [1,3-2\epsilon]$.
Using the definition of $I$, one can use this to show that the support of $I$ does not intersect $J_1 \cup J_2$.
Letting $I_p = M_p$ for $p\notin J_1\cup J_2$ and $I_p=0$ otherwise, we get that $f(I)\sse M_{0\to 2\epsilon}(I)$ for all $f\colon M\to M(2\epsilon)$, so we have found $I$.

By definition, $\Ker_{2\epsilon}(I)\sse K$, so $K_{(p_1,p_2)}=I_{(p_1,p_2)}$ if $p_1\geq 3-2\epsilon$ or $p_2\geq 3-2\epsilon$.
For $p\in [3-4\epsilon,3-2\epsilon]^2$, we have that $K_p\supseteq \Ker_{2\epsilon}(I)_p \neq 0$, and one can construct an $f\colon M\to M(2\epsilon)$ such that the restriction of $f_p$ to $\Ker_{2\epsilon}(I)_p$ is surjective.
It follows that $K_p=I_p$ for $p\in [3-4\epsilon,3-2\epsilon]^2$.
It turns out that nothing more needs to be added to $K$, so $K_p=0$ for
\[
p\in [1-2\epsilon,3-2\epsilon]^2 \setminus [3-4\epsilon,3-2\epsilon]^2,
\]
and $K_p=I_p$ otherwise.
The $\epsilon$-pruning $(I/K)(-\epsilon)$ is isomorphic to $L\oplus L$, where
\[
L = [1+\epsilon,3-\epsilon]^2 \setminus [3-3\epsilon,3-\epsilon]^2.
\]
\end{example}

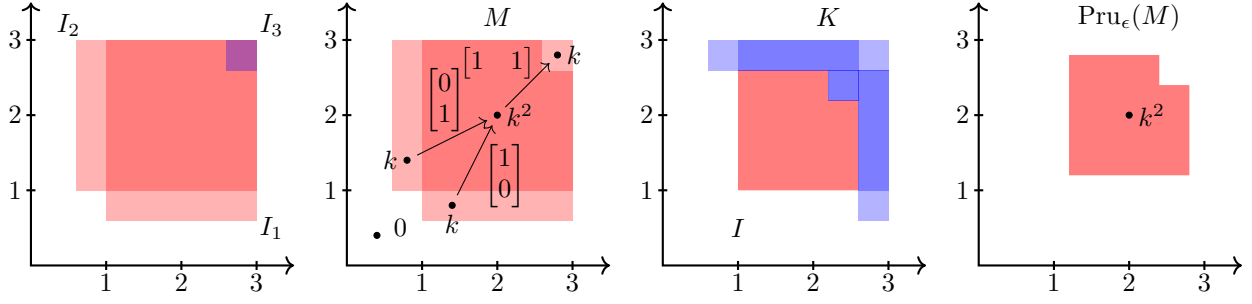
\begin{figure}
\centering
\begin{tikzpicture}[scale=1]
\draw[thick,<->] (0,3.5) to (0,0) to (3.5,0);
\draw[thick] (-.06,1) to (.06,1);
\node[left] at (0,1){$1$};
\draw[thick] (-.06,2) to (.06,2);
\node[left] at (0,2){$2$};
\draw[thick] (-.06,3) to (.06,3);
\node[left] at (0,3){$3$};
\draw[thick] (1,-.06) to (1,.06);
\node[below] at (1,0){$1$};
\draw[thick] (2,-.06) to (2,.06);
\node[below] at (2,0){$2$};
\draw[thick] (3,-.06) to (3,.06);
\node[below] at (3,0){$3$};
\node at (3.2,.5){$I_1$};
\node at (.5,3.2){$I_2$};
\node at (3.2,3.2){$I_3$};
\fill[red, opacity=0.3] (1,.6) rectangle (3,3);
\fill[red, opacity=0.3] (.6,1) rectangle (3,3);
\fill[blue, opacity=0.3] (2.6,2.6) rectangle (3,3);
\begin{scope}[xshift=4.2cm]
\draw[thick,<->] (0,3.5) to (0,0) to (3.5,0);
\draw[thick] (-.06,1) to (.06,1);
\node[left] at (0,1){$1$};
\draw[thick] (-.06,2) to (.06,2);
\node[left] at (0,2){$2$};
\draw[thick] (-.06,3) to (.06,3);
\node[left] at (0,3){$3$};
\draw[thick] (1,-.06) to (1,.06);
\node[below] at (1,0){$1$};
\draw[thick] (2,-.06) to (2,.06);
\node[below] at (2,0){$2$};
\draw[thick] (3,-.06) to (3,.06);
\node[below] at (3,0){$3$};
\node at (2,3.3){$M$};
\fill[red, opacity=0.3] (.6,1) rectangle (3,3);
\fill[red, opacity=0.3] (1,.6) rectangle (3,3);
\fill[white] (2.6,2.6) rectangle (3,3);
\fill[red, opacity=0.3] (2.6,2.6) rectangle (3,3);
\node[right] at (2,2){$k^2$};
\draw[color=black,fill=black] (2,2) circle (.04);
\node[right] at (.5,.5){$0$};
\draw[color=black,fill=black] (.4,.4) circle (.04);
\node[below] at (1.4,.8){$k$};
\draw[color=black,fill=black] (1.4,.8) circle (.04);
\node[left] at (.8,1.4){$k$};
\draw[color=black,fill=black] (.8,1.4) circle (.04);
\node[right] at (2.8,2.8){$k$};
\draw[color=black,fill=black] (2.8,2.8) circle (.04);
\draw[->, shorten <=.15cm, shorten >=.15cm] (.8,1.4) to (2,2);
\draw[->, shorten <=.15cm, shorten >=.15cm] (1.4,.8) to (2,2);
\draw[->, shorten <=.15cm, shorten >=.15cm] (2,2) to (2.8,2.8);
\node at (2.1,1.2){$\begin{bmatrix}
1\\
0
\end{bmatrix}$};
\node at (1.3,2.2){$\begin{bmatrix}
0\\
1
\end{bmatrix}$};
\node at (3.1,2.2){$\begin{bmatrix}
1&-1
\end{bmatrix}$};
\end{scope}
\begin{scope}[xshift=8.4cm]
\draw[thick,<->] (0,3.5) to (0,0) to (3.5,0);
\draw[thick] (-.06,1) to (.06,1);
\node[left] at (0,1){$1$};
\draw[thick] (-.06,2) to (.06,2);
\node[left] at (0,2){$2$};
\draw[thick] (-.06,3) to (.06,3);
\node[left] at (0,3){$3$};
\draw[thick] (1,-.06) to (1,.06);
\node[below] at (1,0){$1$};
\draw[thick] (2,-.06) to (2,.06);
\node[below] at (2,0){$2$};
\draw[thick] (3,-.06) to (3,.06);
\node[below] at (3,0){$3$};
\node at (2.2,3.3){$K$};
\node at (1,.5){$I$};
\fill[red, opacity=0.3] (1,1) to (2.6,1) to (2.6,2.2) to (2.2,2.2) to (2.2,2.6) to (1,2.6);
\fill[red, opacity=0.3] (1,1) to (2.6,1) to (2.6,2.2) to (2.2,2.2) to (2.2,2.6) to (1,2.6);
\fill[blue, opacity=0.3] (.6,2.6) rectangle (3,3);
\fill[blue, opacity=0.3] (1,2.6) rectangle (2.6,3);
\fill[blue, opacity=0.3] (2.6,.6) rectangle (3,2.6);
\fill[blue, opacity=0.3] (2.6,1) rectangle (3,2.6);
\fill[blue, opacity=0.3] (2.2,2.2) rectangle (2.6,2.6);
\fill[blue, opacity=0.3] (2.2,2.2) rectangle (2.6,2.6);
\end{scope}
\begin{scope}[xshift=12.6cm]
\draw[thick,<->] (0,3.5) to (0,0) to (3.5,0);
\draw[thick] (-.06,1) to (.06,1);
\node[left] at (0,1){$1$};
\draw[thick] (-.06,2) to (.06,2);
\node[left] at (0,2){$2$};
\draw[thick] (-.06,3) to (.06,3);
\node[left] at (0,3){$3$};
\draw[thick] (1,-.06) to (1,.06);
\node[below] at (1,0){$1$};
\draw[thick] (2,-.06) to (2,.06);
\node[below] at (2,0){$2$};
\draw[thick] (3,-.06) to (3,.06);
\node[below] at (3,0){$3$};
\node at (2,3.3){$\Pru_\epsilon(M)$};
\fill[red, opacity=0.3] (1.2,1.2) to (2.8,1.2) to (2.8,2.4) to (2.4,2.4) to (2.4,2.8) to (1.2,2.8);
\fill[red, opacity=0.3] (1.2,1.2) to (2.8,1.2) to (2.8,2.4) to (2.4,2.4) to (2.4,2.8) to (1.2,2.8);
\node[right] at (2,2){$k^2$};
\draw[color=black,fill=black] (2,2) circle (.04);
\end{scope}
\end{tikzpicture}
\caption{Left: The intervals used to construct $M$ in \cref{ex_pruning}.
Middle left: $M$, which is indecomposable.
Middle right: The $\epsilon$-pruning pair $(I,K)$.
Right: The $\epsilon$-pruning of $M$, which decomposes.}
\label{fig_pruning_ex}
\end{figure}

For the following, we will find it convenient to extend our definition of $\epsilon$-refinement to $\epsilon=\infty$.
Let $\epsilon\geq 0$ and $M\cong \bigoplus_{i\in \Lambda} M_i$, where each $M_i$ is indecomposable.
An \emph{$\infty$-refinement} of $M$ is a module isomorphic to a module $\bigoplus_{i\in \Lambda} N_i$ such that $N_i$ is a subquotient of $M_i$ for each $i\in \Lambda$.
Let 
\begin{definition}
\label{def_pruning_dist}
The \emph{pruning distance} $d_P$ is defined by
\begin{align*}
d_P(M,N) = \inf\{\epsilon\geq 0\mid \forall \delta\geq 0, &\Pru_{\epsilon+\delta}(M) \text{ is an } \infty\text{-refinement of } \Pru_\delta(N) \text{ and }\\ &\Pru_{\epsilon+\delta}(N) \text{ is an } \infty\text{-refinement of } \Pru_\delta(M)\}
\end{align*}
for any pfd modules $M$ and $N$.
\end{definition}
For $M$ and $N$ to have a small pruning distance, we require $\Pru_\epsilon(M)$ to be more finely decomposed than $N$ and vice versa for a small $\epsilon$, which is similar in spirit to $M$ and $N$ having a common $\epsilon$-refinement.
But to get around the problem of indecomposable $M$ and $N$, for which ``more finely decomposed'' is not an interesting property, we also require $\Pru_{\epsilon+\delta}(M)$ to be more finely decomposed than $\Pru_\delta(N)$ for $\delta>0$, which contains information about decompositions of modules in a neighborhood of $N$.
This is how we exploit the loophole mentioned above, and get a distance that depends meaningfully on decompositions.

\begin{lemma}
\label{lem_refine_add}
Let $\epsilon,\delta\in [0,\infty]$, and let $M$ be a pfd module.
Suppose $N$ is an $\epsilon$-refinement of $M$ and $Q$ is a $\delta$-refinement of $N$.
Then $Q$ is an $(\epsilon+\delta)$-refinement of $M$.
\end{lemma}
\begin{proof}
To unify notation for finite and infinite $\epsilon, \delta$, for any module $A$, let $\EN_\infty(A)$ be the collection of subquotients of $A$.
Recall that we use the notation $B\leq A$ if $B$ is a subquotient of $A$.

Write $M \cong \bigoplus_{i\in \Lambda} M_i$ and $N \cong \bigoplus_{i\in \Lambda} N_i$ with $M_i$ indecomposable and $N_i\in \EN_\epsilon(M_i)$ for all $i\in \Lambda$.
Let also $N \cong \bigoplus_{j\in \Gamma} N'_j$ and $Q \cong \bigoplus_{j\in \Gamma} Q_j$ with $N'_j$ indecomposable and $Q_j\in \EN_\delta(N'_j)$ for all $j\in \Gamma$.
By \cref{thm_unique_dec}, there is a partition of $\Gamma$ into subsets $\Gamma_i$ for $i\in \Lambda$ such that $N_i \cong \bigoplus_{j\in \Gamma_i} N'_j$ for all $i\in \Lambda$.
We get that $Q_i\coloneqq \bigoplus_{j\in \Gamma_i} Q_j\in \EN_\delta(N_i)$ for all $i\in \Lambda$.
Putting together $N_i\in \EN_\epsilon(M_i)$ and $Q_i \in \EN_\delta(N_i)$, we get $Q_i\in \EN_{\epsilon+\delta}(M_i)$ as follows:
If $\epsilon$ and $\delta$ are finite, then this is \cref{lem_erosion_containment2}.
If at least one of $\epsilon$ or $\delta$ is infinite, then $Q_i \leq N_i \leq M_i$, and then $Q_i\leq M_i$ follows from the transitivity of \cref{lem_subq_poset}.

Thus, since $Q \cong \bigoplus_{i\in \Lambda} \bigoplus_{j\in \Gamma_i} Q_j = \bigoplus_{i\in \Lambda} Q_i$ and $M \cong \bigoplus_{i\in \Lambda} M_i$, we get that $Q$ is an $(\epsilon+\delta)$-refinement of $M$ regardless of whether $\epsilon$ and $\delta$ are finite.
\end{proof}

\begin{proposition}
The pruning distance is a distance.
\end{proposition}
\begin{proof}
Symmetry and $d_P(M,M) = 0$ are immediate.
For the triangle inequality, let $\epsilon,\epsilon'\geq 0$ and suppose that $\Pru_{\epsilon+\delta}(M)$ is an $\infty$-refinement of $\Pru_\delta(N)$ and $\Pru_{\epsilon'+\delta'}(N)$ is an $\infty$-refinement of $\Pru_{\delta'}(Q)$ for all $\delta,\delta'\geq 0$.
Let $\gamma\geq 0$.
Picking $\delta=\epsilon'+\gamma$ and $\delta'=\gamma$, we get that $\Pru_{\epsilon+\epsilon'+\gamma}(M)$ is an $\infty$-refinement of $\Pru_{\epsilon'+\gamma}(N)$, which is in turn an $\infty$-refinement of $\Pru_\gamma(Q)$.
By \cref{lem_refine_add} applied to $\infty$-refinements, we get that $\Pru_{\epsilon+\epsilon'+\gamma}(M)$ is an $\infty$-refinement of $\Pru_\gamma(Q)$.
Running the same argument in the opposite direction from $Q$ to $M$, we get that $\Pru_{\epsilon+\epsilon'+\gamma}(Q)$ is an $\infty$-refinement of $\Pru_\gamma(M)$, and the triangle inequality of $d_P(M,N)$ follows.
\end{proof}

We now state our conjecture that this distance is stable, and even Lipschitz equivalent to the interleaving distance if we restrict ourselves to modules of a bounded pointwise dimension.
\begin{conjecture}
\label{conj_pruning}
Let $M$ and $N$ be pfd modules with $r = \supdim M<\infty$.
Then
\[
d_P(M,N)\leq d_I(M,N)\leq 2rd_P(M,N).
\]
\end{conjecture}
We can make some initial progress towards proving the conjecture:
Since a $c$-refinement for a finite $c$ is an $\infty$-refinement and $\Pru_0(N)\cong N$, \cref{thm_main} (ii) shows that for $\epsilon$-interleaved $M$ and $N$, $\Pru_{\epsilon+\delta}(M)$ is an $\infty$-refinement of $\Pru_\delta(N)$ for $\delta=0$.
To show $d_P(M,N)\leq d_I(M,N)$, we need to extend this result to all $\delta\geq 0$.

By \cref{thm_main} (ii), $\Pru_{\epsilon}(M)$ is a $2r\epsilon$-refinement of $M$, so by \cref{lem_refins_are_ers} (ii) and \cref{thm_erosion_containment_gives_int}, $M$ and $\Pru_{\epsilon}(M)$ are $2r\epsilon$-interleaved.
Suppose $d_P(M,N)<\epsilon$.
Then $\Pru_{\epsilon}(M)$ is an $\infty$-refinement of $N$, and $\Pru_{\epsilon}(N)$ is an $\infty$-refinement of $M$.
The former means that, informally speaking, $M$ is within interleaving distance $2r\epsilon$ of being smaller than $N$.
If one can also show that $N$ is within interleaving distance $2r\epsilon$ of being smaller than $M$, one could hope to combine and make precise these informal statements to show that $M$ and $N$ are $2r\epsilon$-interleaved, which would imply the inequality $d_I(M,N)\leq 2rd_P(M,N)$.

Another interesting question is the computability of prunings and the pruning distance.
We conjecture that computing prunings is similar in flavor to computing direct decompositions and as a result can be done in polynomial time.
The pruning distance is superficially related to both interleavings (if \cref{conj_pruning} holds) and decompositions, so it is not clear to us if it is more likely to be NP-hard or polynomial time computable.
In the examples used to show NP-hardness of computing $d_I$ \cite{bjerkevik2019computing}, we believe that computing $d_P$ can be done in polynomial time, so there is some hope that a general polynomial time algorithm can be found.

The definition of the pruning distance is a reasonable one, but one can argue that it does not behave quite the way a decomposition distance should.
Instead of having $d_P\leq d_I$, as in \cref{conj_pruning}, it might make more sense to define a decomposition distance $d$ in such a way that $d(M,N)> d_I(M,N)$ in cases where $M$ and $N$ are close in the interleaving distance, but decompose in very different ways.
With this in mind, the definition of $d_P$ might need some adjustments to tick off all the boxes for an optimal decomposition distance.
In any case, $d_P$ demonstrates how to get around the obstacle of \cref{thm_uli_luis_cor}, and how prunings can be useful to answer this question.

\subsection{A good decomposition distance for interval decomposable modules}
\label{sec_f_R_distance}

Though $f_R$ is not a distance in the general setting, we will now see in \cref{thm_f_R_distance} that it is a distance when restricted to a more limited class of modules that includes interval decomposable modules.
We also show in \cref{thm_eq_f_R_d_B} that $f_R$ agrees with the bottleneck distance on one-parameter modules and multiparameter modules that decompose into certain interval modules called \emph{upset} modules, which we will study in \cref{sec_matchings}.

With this in mind, $f_R$ is arguably a satisfying answer to the question of how to define a good decomposition distance for interval decomposable modules.
It agrees with $d_B$ in many cases, but unlike $d_B$, it is flexible enough to behave in a stable way in cases like the one in \cref{fig_large_bottleneck}.
We also conjecture that for interval decomposable modules, it approximates the interleaving distance up to a multiplicative constant.
(This follows from \cref{conj_main} with $r=1$.)

\begin{lemma}
\label{lem_subq_to_ref_for_intervals}
Suppose $M$ is pointwise at most one-dimensional; that is, for every $p\in \Pb$, $\dim(M_p)\leq 1$.
Then any $N\in \EN_\epsilon(M)$ is an $\epsilon$-refinement of $M$.
\end{lemma}
\begin{proof}
Write $M = \bigoplus_{i\in \Lambda} M_i$ with each $M_i$ indecomposable.
Since $N\in \EN_\epsilon(M)$, there are $N''\sse N'\sse M$ with $N \cong N'/N''$, $\Img_\epsilon(M)\sse N'$ and $N''\sse \Ker_\epsilon(M)$.
For all $i\in \Lambda$, let $N'_i = N'\cap M_i$ and $N''_i = N''\cap M_i$.

We now show that $N' = \bigoplus_{i\in \Lambda}N'_i$.
It is enough to show that this holds at an arbitrary $p\in \Pb$, and $N' \supseteq \bigoplus_{i\in \Lambda}N'_i$ is immediate, so it suffices to show $N'_p \sse \bigoplus_{i\in \Lambda}(N'_i)_p$.
This is trivial if $N'_p = 0$, so assume $N'_p \neq 0$.
Then $M_p\neq 0$, so there is an $i\in \Lambda$ with $(M_i)_p\neq 0$.
We have $N'_p\sse M_p \supseteq (M_i)_p$, so since $\dim(M_p)\leq 1$ and $\dim(M'_p), \dim(N'_p)\geq 1$, both inclusions are equalities.
Thus, $N'_p = (M_i)_p$, which gives $N'_p = (N'_i)_p$.
In other words, for any point $p$, there is an $i$ with $N'_p = (N'_i)_p$, so $N'_p \sse \bigoplus_{i\in \Lambda}(N'_i)_p$ follows.
As pointed out above, this gives us $N' = \bigoplus_{i\in \Lambda}N'_i$.
A similar argument gives $N'' = \bigoplus_{i\in \Lambda}N''_i$.
Thus,
\[
N \cong \bigoplus_{i\in \Lambda}N'_i/\bigoplus_{i\in \Lambda}N''_i \cong \bigoplus_{i\in \Lambda}(N'_i/N''_i).
\]

It remains to show that for every $i\in \Lambda$, $N'_i/N''_i\in \EN_\epsilon(M_i)$.
We know that $N''_i\sse N''\sse \Ker_\epsilon(M)$ and $N''_i\sse M_i$, so $N''_i\sse \Ker_\epsilon(M)\cap M_i = \Ker_\epsilon(M_i)$.
We also know that $\Img_\epsilon(M)\sse N'$, which gives $\Img_\epsilon(M)\cap M_i\sse N'\cap M_i = N'_i$.
But $\Img_\epsilon(M_i)\sse \Img_\epsilon(M)\cap M_i$, so we get $\Img_\epsilon(M_i) \sse N'_i$.
Thus, $N'_i/N''_i\in \EN_\epsilon(M_i)$, which completes the proof.
\end{proof}

$f_R$ trivially satisfies the first two conditions of a distance (\cref{def_distance}), and the following shows that it also satisfies the last condition, namely the triangle inequality, in a restricted setting.
\begin{theorem}
\label{thm_f_R_distance}
Let $M$, $N$ and $Q$ be pfd modules, and write $N = \bigoplus_{i\in \Lambda} N_i$ with each $N_i$ indecomposable..
Suppose each $N_i$ is pointwise at most one-dimensional; that is, $\dim((N_i)_p)\leq 1$ for all $p\in \Pb$.
Then \[f_R(M,Q)\leq f_R(M,N)+f_R(N,Q).\]
\end{theorem}
\begin{proof}
Pick $\epsilon>f_R(M,N)$ and $\delta> f_R(N,Q)$.
(If $f_R(M,N)+f_R(N,Q)=\infty$, the theorem is trivial.)
Then there exists an $\epsilon$-refinement $A$ of both $M$ and $N$, and a $\delta$-refinement $B$ of both $N$ and $Q$.
We can write $A = \bigoplus_{i\in \Lambda} A_i$ and $B = \bigoplus_{i\in \Lambda} B_i$ with $A_i\in \EN_\epsilon(N_i)$ and $B_i\in \EN_\delta(N_i)$ for every $i\in \Lambda$.
This means that we can write $A_i \cong A'_i/A''_i$ and $B_i \cong B'_i/B''_i$, where
\begin{align*}
A''_i\sse A'_i\sse N_i, & & \Img_\epsilon(N_i)\sse A'_i, & & A''_i\sse \Ker_\epsilon(N_i),\\
B''_i\sse B'_i\sse N_i, & & \Img_\delta(N_i)\sse B'_i, & & B''_i\sse \Ker_\delta(N_i).
\end{align*}
For any $S\sse A'_i$, let $\ol S$ denote the image of $S$ under the projection $A'_i\to A'_i/A''_i$.
For $i\in \Lambda$, let $C_i = \ol{A'_i\cap B'_i}/\ol{A'_i\cap B'_i \cap (B''_i\cup A_i'')}$.
We claim that $C\coloneqq\bigoplus_{i\in \Lambda}C_i$ is a $\delta$-refinement of $A$.
By \cref{lem_refine_add}, it follows that $C$ is an $(\epsilon+\delta)$-refinement of $M$.
A symmetric proof gives that $C$ is also an $(\epsilon+\delta)$-refinement of $Q$.
Thus, $f_R(M,Q)\leq \epsilon+\delta$, and the theorem follows.

Note that for any $S\sse S'\sse A'_i$, we have $\ol S\sse \ol{S'}$.
We have $\Img_\delta(N_i)\sse B'_i$, so $\Img_\delta(A'_i)\sse A'_i\cap B'_i$ by \cref{lem_im_to_im} (i) with $f$ the inclusion $A'_i\hookrightarrow N_i$, which gives $\ol{\Img_\delta(A'_i)} \sse \ol{A'_i\cap B'_i}$.
Since $\Img_\delta(A'_i/A''_i) = \ol{\Img_\delta(A'_i)}$, we get $\Img_\delta(A'_i/A''_i) \sse \ol{A'_i\cap B'_i}$.

We also have $B''_i\sse \Ker_\delta(N_i)$, which gives $A'_i\cap B''_i \sse A'_i \cap \Ker_\delta(N_i) = \Ker_\delta(A'_i)$.
Observe that $\ol{A'_i\cap B''_i} = \ol{A'_i\cap B'_i\cap B''_i} = \ol{A'_i\cap B'_i\cap (B''_i\cup A''_i)}$, where the first equality follows from $B''_i\sse B'_i$ and the second from $\ol{A''_i} = 0$.
We get
\[
\ol{A'_i\cap B'_i\cap (B''_i\cup A''_i)} = \ol{A'_i\cap B''_i} \sse \ol{\Ker_\delta(A'_i)}\sse \Ker_\delta(A'_i/A''_i).
\]
It follows that $C_i\in \EN_\delta(A'_i/A''_i) = \EN_\delta(A_i)$.
Since $A_i$ is a subquotient of $N_i$, $A_i$ is pointwise at most one-dimensional.
By \cref{lem_subq_to_ref_for_intervals}, $C_i$ is a $\delta$-refinement of $A_i$.
Thus, $C$ is a $\delta$-refinement of $A$.
\end{proof}

\begin{theorem}
\label{thm_eq_f_R_d_B}
Let $M$ and $N$ be pfd and either both one-parameter modules or both upset decomposable $n$-parameter modules.
For any $\epsilon\geq 0$, there is an $\epsilon$-matching between $M$ and $N$ if and only if there is a module $Q$ that is an $\epsilon$-refinement of both $M$ and $N$.
Thus, $f_R(M,N) = d_B(M,N)$.
\end{theorem}
\begin{proof}
The upset decomposable case is exactly \cref{lem_match_eq_to_refine}.

Assume that $M$ and $N$ are one-parameter modules.
Suppose first that there is an $\epsilon$-matching between them.
The definition of $\epsilon$-matching (\cref{def_matching}) gives us a bijection $\sigma\colon \bar B(M)\to \bar B(N)$ between subsets of the barcodes of $M$ and $N$ such that matched elements are $\epsilon$-interleaved, and non-matched elements are $\epsilon$-interleaved with the zero module.
\cref{thm_strong_equiv} (i) says that for each $R\in \bar B(M)$, there is a module $Q_R \in \EN_\epsilon(R)\cap \EN_\epsilon(\sigma(R))$.
It also says that for $R\in B(M)\setminus \bar B(M)$, we have a $Q_R \in \EN_\epsilon(R)\cap \EN_\epsilon(0)$, which has to be the zero module, since it is a subquotient of the zero module.
Let $Q = \bigoplus_{R\in B(M)} Q_R$.
By construction, $Q$ is an $\epsilon$-refinement of $M$.
Note that $Q \cong \bigoplus_{R\in \bar B(M)} Q_R$, since we have only removed zero summands.
Thus, up to isomorphism, $Q$ would have been defined in exactly the same way if $M$ and $N$ had been switched (and we had used $\sigma^{-1}$ instead of $\sigma$), so by symmetry, $Q$ is also an $\epsilon$-refinement of $N$.

Next, we show that if $M$ and $N$ have a common $\epsilon$-refinement $Q$, then there is an $\epsilon$-matching between $M$ and $N$.
By definition of $\epsilon$-refinement, we have that if we write $M = \bigoplus_{i\in\Lambda} M_i$ with each $M_i$ nonzero and indecomposable, then we can write $Q\cong \bigoplus_{i\in\Lambda} Q_i$ with each $Q_i = Q'_i/Q''_i\in \oEN_\epsilon(M_i)$.
It follows that $Q\cong (\bigoplus_{i\in\Lambda} Q'_i)/(\bigoplus_{i\in\Lambda} Q''_i) \in \oEN_\epsilon(M)$.
Similarly, $Q\in \EN_\epsilon(N)$.
By \cref{thm_EN=I}, it follows that $M$ and $N$ are $\epsilon$-interleaved, and then \cref{thm_ast} says that there is an $\epsilon$-matching between $M$ and $N$.
\end{proof}

\section{Stability}
\label{sec_main}

In this section, we prove \cref{thm_main}, which says that the $\epsilon$-pruning $\Pr_\epsilon(M)$ of a module $M$ is a $2r\epsilon$-refinement of all the modules in an interleaving neighborhood of $M$ ($r$ being the maximum pointwise dimension of $M$).
Thus, the pruning of a module can be viewed as a common approximate barcode of a whole neighborhood of a module.
As explained in \cref{sec_motivation}, we view a common refinement between two modules as an approximate matching between them.
Therefore, we interpret \cref{thm_main} as showing stability of approximate decompositions up to a certain factor.

In \cref{thm_counterex}, we show that up to a multiplicative constant, the factor $2r\epsilon$ appearing in \cref{thm_main} is optimal.
We believe that this expresses a fundamental instability property of multiparameter decomposition; that no reasonable definition of a ``decomposition distance'' $d$ (like the pruning distance) will allow a theorem of the form $\frac 1cd_I \leq d\leq cd_I$ for a constant $c$.
Therefore, even though the factor $r$ showing up in the theorem is larger than one might hope for, it is the result of encountering what seems to be some hard limits of decomposition stability.
The result of Theorems \ref{thm_main} and \ref{thm_counterex} is that we have a much more nuanced picture of decomposition (in)stability:
\cref{thm_main} shows that positive decomposition stability results for general modules are possible, which was not at all clear before this paper.
\cref{thm_counterex}, on the other hand, suggests that equivalence of meaningful decomposition metrics with $d_I$ ``up to a constant'' is impossible, because it would contradict fundamental properties of decompositions.
Thus, we have stability up to a factor depending on pointwise dimension, but stability up to a \emph{constant} factor looks unrealistic.
Our guess is that these upper and lower limits for decomposition stability meet in \cref{conj_main}, where the factor $r$ depends on the pointwise dimensions of summands of $M$, instead of the pointwise dimensions of $M$.

Due to the factor $2r$, the theorem is mainly of interest when comparing modules $M$ and $N$ such that $2rd_I(M,N)$ is not too big compared to the size of the supports of $M$ and $N$, since if $M$ and $N$ have bounded support, then for large enough $\delta$, the zero module is a $\delta$-refinement of both $M$ and $N$.
On the other hand, putting together \cref{lem_refins_are_ers} (ii), \cref{thm_erosion_containment_gives_int} and \cref{cor_main}, we get $d_I(M,\Pr_\epsilon(M))\leq 2r\epsilon$.
Thus, if $d_I(M,0)$ is large compared to $2r\epsilon$, then $d_I(\Pr_\epsilon(M),0)$ is large, which means that \cref{thm_main} guarantees that $M$ and $N$ have a large approximate summand (or several) in common.
(See \cref{ex_pruning} for an illustration.)

For practical purposes, it is worth remarking that even though the theorem only promises that the $\epsilon$-pruning of $M$ is a $2r\epsilon$-refinement of $M$ and $N$, this is a worst-case scenario that is not always realized.
The $r$ in the theorem is the result of an iterative process of constructing submodules in the proof of \cref{lem_I_K}.
If this process stabilizes in $r'<r$ steps, then $\Pru_\epsilon(M)$ is a $2r'\epsilon$-refinement of $M$ and $N$, not just a $2r\epsilon$-refinement.
Therefore, if one is given a specific module and constructs its $\epsilon$-pruning (leaving aside the question of algorithms for such a construction), one might be able to check that the pruning has better stability properties than what is promised in the theorem.
For large $r$, we suspect that the mentioned process will usually stabilize in much fewer than $r$ steps, but this should be tested empirically.

\subsection{The stability theorem}
\label{sec_stability_theorem}

Before getting into the details, we give some motivation for the proof of \cref{thm_main}, which also motivates our definition of $\epsilon$-pruning.
In \cref{thm_main} (i), we assume that we are given modules $M$ and $N$ and morphisms
\[
M \xra{\phi} N(\epsilon) \xra{\psi} M(2\epsilon)
\]
for some $\epsilon\geq 0$ whose composition is $M_{0\to 2\epsilon}$, and we want to use these morphisms to construct a $\delta$-refinement $\tilde N$ of $N$ for a certain $\delta\geq 2\epsilon$.
Moreover, we want $\tilde N$ to have a direct summand that is isomorphic to a $\delta$-refinement of $M$.

Suppose first that $M$ and $N$ are indecomposable.
Then $\tilde N$ being a $\delta$-refinement of $N$ is equivalent to $\tilde N\in \EN_\delta(N)$, and similarly for $M$.
In this case, we can choose $\tilde N = N'/N''$ with $N' = \psi^{-1}(\img(M_{0\to 2\epsilon}))$ and $N'' = \phi(\ker M_{0\to 2\epsilon})$, and consider the morphisms
\[
M/\ker(M_{0\to 2\epsilon}) \xra{\ol \phi} \tilde N \xra{\ol\psi} \img(M_{0\to 2\epsilon})
\]
induced by $\phi$ and $\psi$.
It is possible to prove that $M/\ker(M_{0\to 2\epsilon})\in \EN_{2\epsilon}(M)$, that $\tilde N\in \EN_{2\epsilon}(N)$, and that $\ol \phi\circ \ol \psi$ is an isomorphism.
A version of the splitting lemma then tells us that $M/\ker(M_{0\to 2\epsilon})$ is a direct summand of $\tilde N$, and we are done.

Next, suppose that $N=N_1\oplus N_2$ with $N_1$ and $N_2$ indecomposable.
In this case, the procedure above does not work, because we have no guarantee that $\tilde N$ is a refinement of $N$.
To be sure that we get a refinement, we need to take subquotients of $N_1$ and $N_2$ separately.
A reasonable attempt to do this is to let $\tilde N = \iota_1^{-1}(N')/\pi_1(N'')\oplus \iota_2^{-1}(N')/\pi_2(N'')$ instead, where $\pi_i$ and $\iota_i$ denote the canonical projection and inclusion between $N_i$ and $N$.
But now new problems appear: for instance, $\pi_1(N'')$ is not necessarily in the kernel of $\psi$, which means that we do not always get a well-defined morphism $\tilde N\to \img(M_{0\to 2\epsilon})$ induced by $\psi$.

Issues like these of having to respect direct sums is the main challenge that needs to be overcome to prove \cref{thm_main}, and prunings are specifically constructed to deal with this.
In the proof, we consider the $\epsilon$-pruning $I/K$ of $M$, where $K\sse I$ are the submodules of $M$ forming the $\epsilon$-pruning pair of $M$.
Writing $N = \bigoplus_{j\in \Gamma} N_j$ with each $N_j$ indecomposable, we define
\[
\ol{N_j} = (\psi\circ \iota_j)(\epsilon)^{-1}(M_{0\to 2\epsilon}(I))/\pi_j\circ \phi(K)
\]
and $\ol{N} = \bigoplus_{j\in \Gamma} \ol{N_j}$, and consider the sequence
\[
I/K\xra{\ol{\phi_j}} \ol{N_j} \xra{\ol{\psi_j(\epsilon)}} M_{0\to 2\epsilon}(I)/ M_{0\to 2\epsilon}(K)
\]
of morphisms induced by $\pi_j\circ \phi$ and $\psi\circ \iota_j$.
This time, it turns out that we do get valid morphisms because of the particular properties of prunings.
For instance, we need to check that $\psi(\pi_j\circ \phi(K))\sse M_{0\to 2\epsilon}(K)$, but since the left hand side is of the form $f(K)$, this holds by \cref{lem_I_K} (iv).
This type of argument is the motivation for the perhaps mysterious-looking conditions on $I$ and $K$ in \cref{def_pruning}.

After having proved that $\ol N$ is a direct summand of a $\delta$-pruning in \cref{thm_main} (i), we assume that $\phi$ and $\psi$ form an interleaving between $M$ and $N$ in part (ii) of the theorem, and show that $I/K$ is a $\delta$-refinement of $N$.
We are able to write $\ol{N} = N' \oplus \ker \ol{\psi(\epsilon)}$ with $N'\cong I/K$, and because of the interleaving condition, we know that $\ol{\psi(\epsilon)}$ is small in a precise sense.
Again, the case with $N$ indecomposable is not too hard: one can show that $\ol{N}/\ker \ol{\psi(\epsilon)}\cong I/K$ is a $\delta$-refinement of $N$.
And again, the case where $N$ decomposes is significantly more involved, because to obtain a refinement, we need to be careful that the submodules we quotient out respect the direct decomposition of $N$.
We solve this problem by using Zorn's lemma to replace indecomposable summands of $\ker \ol{\psi(\epsilon)}$ with isomorphic modules that do respect the decomposition of $N$ into indecomposables, so that we can take a quotient of $\ol N$ to get a $\delta$-refinement of $N$ that is isomorphic to $I/K$.

Recall the convention that writing $x\in M$ means that $x\in M_p$ for some point $p\in \Pb$.
\begin{lemma}
\label{lem_im_to_im}
Let $f\colon M \to N$ be a morphism of modules, $I$ a submodule of $N$ and $K$ a submodule of $M$.
\begin{itemize}
\item[(i)] If $\Img_\epsilon(N)\sse I$, then $\Img_\epsilon(M)\sse f^{-1}(I)$.
\item[(ii)] If $K\subseteq \Ker_\epsilon(M)$, then $f(K)\subseteq \Ker_\epsilon(N)$.
\end{itemize}
\end{lemma}
\begin{proof}
(i) Let $x\in \Img_\epsilon(M)$.
Then there exists $y\in M(-\epsilon)$ such that $M_{-\epsilon\to 0}(y) = x$.
We get
\[
f(x) = f(M_{-\epsilon\to 0}(y)) = N_{-\epsilon\to 0}(f(-\epsilon)(y)),
\]
which means that $f(x)\in \Img_\epsilon(N) \subseteq I$, so $x\in f^{-1}(I)$.

To prove (ii), suppose $x\in K$.
Then $M_{0\to \epsilon}(x)=0$, so
\[
N_{0\to \epsilon}(f(x)) = f(\epsilon)(M_{0\to \epsilon}(x)) = f(\epsilon)(0) = 0,
\]
which gives $f(x)\in \Ker_\epsilon(N)$.
\end{proof}

\begin{lemma}
\label{lem_pruning_pair}
Let $M$ be a module.
Then the $\delta$-pruning pair $(I,K)$ of $M$ is well-defined.
Let $\{f_j\}_{j\in \Gamma}$ be the set of all morphisms from $M$ to $M(2\delta)$.
Explicitly, $I$ is equal to $\bigcap_{i=0}^\infty I_i$, where we define
$I_0=M$, and for all $i\geq 1$,
\begin{itemize}
\item $I_i = \bigcap_{j\in \Gamma} f_j^{-1}(M_{0\to 2\delta}(I_{i-1}))\sse M$.
\end{itemize}
$K$ is equal to $\bigcup_{i=0}^\infty K_i$, where we define $K_0 = 0$, and for all $i\geq 1$,
\begin{itemize}
\item $K_i= \sum_{j\in \Gamma}I_{0\to 2\delta}^{-1}( f_j(K_{i-1})) \subseteq I$.
\end{itemize}
Moreover, for all $i\geq 0$, we have $I_i\supseteq I_{i+1}$ and $K_i\sse K_{i+1}$.
\end{lemma}

\begin{proof}
It is trivially true that $I_0\supseteq I_1$ and $K_0\sse K_1$.
Assuming $I_i\supseteq I_{i+1}$ for some $i\geq 0$, we get $I_{i+1}\supseteq I_{i+2}$ from the definition, so by induction, $I_i\supseteq I_{i+1}$ holds for all $i\geq 0$.
A similar inductive argument shows that $K_i\sse K_{i+1}$ holds for all $i\geq 0$.

Let $f\colon M\to M(2\delta)$.
To show $f(I)\sse M_{0\to 2\delta}(I)$, it suffices to show $f(I)\sse M_{0\to 2\delta}(I_i)$ for all $i\geq 0$, since $I_0\supseteq I_1\supseteq \dots$.
Fix $i\geq 0$.
We have
\begin{align*}
f(I) &\sse f(I_{i+1})\\
&= f(\bigcap_{j\in \Gamma} f_j^{-1}(M_{0\to 2\delta}(I_i)))\\
&\sse M_{0\to 2\delta}(I_i).
\end{align*}

We now show that $I$ has the maximality property, which implies that $I$ satisfies the conditions in the definition of $\delta$-pruning pair.
Let $\tilde I$ be a submodule of $M$ with $f_j(\tilde I)\sse M_{0\to 2\delta}(\tilde I)$ for all $j\in \Gamma$, which implies $\tilde I\sse f_j^{-1}(M_{0\to 2\delta}(\tilde I))$ for all $j\in \Gamma$, and thus,
\begin{equation}
\label{eq_inters}
\tilde I\sse \bigcap_{j\in \Gamma} f_j^{-1}(M_{0\to 2\delta}(\tilde I)).
\end{equation}
We have $\tilde I\sse I_0$ trivially.
If $\tilde I\sse I_i$ for some $i\geq 0$, then we immediately get $\tilde I\sse I_{i+1}$ by definition of $I_{i+1}$.
By induction, $\tilde I\sse I_i$ holds for all $i$, which gives $\tilde I\sse I$.

Let $f\colon M\to M(2\delta)$.
To show $I_{0\to 2\delta}^{-1}(f(K))\sse K$, it is enough to show that $I_{0\to 2\delta}^{-1}(f(K_i))\sse K$ for an arbitrary $i\geq 0$, since $K_0\sse K_1\sse \dots$.
But this follows from the definition of $K_{i+1}$:
\begin{align*}
I_{0\to 2\delta}^{-1}(f(K_i)) \sse K_{i+1} \sse K.
\end{align*}

Lastly, we show minimality of $K$ by induction.
Let $\tilde K$ be a module with $I_{0\to 2\delta}^{-1}(f_j(\tilde K))\sse \tilde K$ for all $j\in \Gamma$.
We have $K_0\sse \tilde K$ trivially.
If $K_i\sse \tilde K$ for some $i\geq 0$, then we immediately get $K_{i+1}\sse \tilde K$ by definition of $K_{i+1}$.
By induction, $K_i\sse \tilde K$ holds for all $i$, which gives $K\sse \tilde K$.
\end{proof}

For a module $M$, let $\supdim M = \sup_{p\in \Pb} \dim M_p \in \{0, 1, \dots, \infty\}$.
\begin{lemma}
\label{lem_I_K}
Let $M$ be a module with $\supdim M =r <\infty$ and let $\epsilon\geq 0$, and let $(I,K)$ be the $\epsilon$-pruning pair of $M$.
Then
\begin{itemize}
\item[(i)] $\Img_{2r\epsilon}(M(2\epsilon)) \sse M_{0\to2\epsilon}(I)$,
\item[(ii)] $K\sse \Ker_{2r\epsilon}(I)$,
\item[(iii)] $K = I_{0\to 2\epsilon}^{-1}(I_{0\to 2\epsilon}(K))$,
\item[(iv)] for any morphism $f\colon M\to M(2\epsilon)$, we have $f(K)\subseteq M_{0\to 2\epsilon}(K)$.
\end{itemize}
\end{lemma}

\begin{proof}
Observe that in \cref{lem_pruning_pair}, $I_i$ and $K_i$ are defined pointwise.
That is, for all $p\in \Pb$,
\begin{align}
\label{eq_I_p}
(I_i)_p &= \bigcap_{j\in \Gamma} (f_j)_p^{-1}((M_{0\to 2\epsilon})_p((I_{i-1})_p)),\\
\label{eq_I_p2}
(K_i)_p &= \sum_{j\in \Gamma}(I_{0\to 2\epsilon})_p^{-1}((f_j)_p((K_{i-1})_p)).
\end{align}
We also have the pointwise inclusions $(I_i)_p\supseteq (I_{i+1})_p$ and $(K_i)_p\sse (K_{i+1})_p$ for all $i\geq 0$.
Since $\dim M_p\leq r$, there can be at most $r$ strict inclusions in the sequence $(I_0)_p\supseteq (I_1)_p \supseteq (I_2)_p \supseteq \dots$, so there is a $t\leq r$ with $(I_t)_p = (I_{t+1})_p$.
From this equality and \cref{eq_I_p}, we get $(I_{t+1})_p = (I_{t+2})_p$.
Iterating this argument, we get $(I_t)_p = (I_{t'})_p$ for all $t'\geq t$, from which it follows that $(I_r)_p = (I_{r'})_p$ for all $r'\geq r$.
Thus, $I_p = (I_r)_p$.
Since $p$ was arbitrary, $I=I_r$ follows.
A completely analogous argument using \cref{eq_I_p2} gives $K = K_r$.

Let $I'_0=M(2\epsilon)$, and define $I'_i = M_{0\to 2\epsilon}(I_{i-1})$ for $i\geq 1$.
Observe that $I'_i = M_{0\to 2\epsilon}(\bigcap_{j\in \Gamma} f_j^{-1}(I'_{i-1}))$.
By arguments similar to those for $I_i$ and $K_i$, we get $I'_r = I'_{r'}$ for all $r'\geq r$.
We claim that for all $i\geq 0$, $\Img_{2i\epsilon}(M(2\epsilon)) \sse I'_i$.
We proceed by induction.
For $i=0$, both sides are equal to $M(2\epsilon)$.
Assume $\Img_{2(i-1)\epsilon}(M(2\epsilon)) \sse I'_{i-1}$, for some $i\geq 1$.
By \cref{lem_im_to_im} (i), $\Img_{2(i-1)\epsilon}(M) \sse f_j^{-1}(I'_{i-1})$ for all $j\in \Gamma$, so
\begin{align*}
I'_i &= M_{0\to 2\epsilon}(\bigcap_{j\in \Gamma} f_j^{-1}(I'_{i-1}))\\
&\supseteq M_{0\to 2\epsilon}(\Img_{2(i-1)\epsilon}(M))\\
&= \Img_{2i\epsilon}(M(2\epsilon)).
\end{align*}
The claim follows by induction, which gives us (i):
\[
M_{0\to 2\epsilon}(I) = M_{0\to 2\epsilon}(I_r) = I'_{r+1} = I'_r \supseteq \Img_{2r\epsilon}(M(2\epsilon)).
\]

Wo now show (ii).
By \cref{lem_pruning_pair}, it suffices to show $K_i\sse \Ker_{2r\epsilon}(I)$ for all $i\geq 0$, which we will prove by induction.
The statement is trivial for $i=0$.
Assuming $i\geq 1$ and $K_{i-1} \sse \Ker_{2(i-1)\epsilon}(I)$, we have
\[
f_j(K_{i-1}) \sse f_j(\Ker_{2(i-1)\epsilon}(I)) \sse \Ker_{2(i-1)\epsilon}(I(2\epsilon))
\]
by \cref{lem_im_to_im} (ii).
This gives
\begin{align*}
K_i &= \sum_{j\in \Gamma}I_{0\to 2\epsilon}^{-1}( f_j(K_{i-1}))\\
&\sse I_{0\to 2\epsilon}^{-1}( \Ker_{2(i-1)\epsilon}(I(2\epsilon)))\\
&= I_{0\to 2\epsilon}^{-1}(I_{2\epsilon\to 2i\epsilon}^{-1}(0))\\
&= I_{0\to 2i\epsilon}^{-1}(0)\\
&= \Ker_{2i\epsilon}(I).
\end{align*}
This concludes the proof by induction.
Since $K=K_r$, we get $K\sse \Ker_{2r\epsilon}(M)$.

In (iii), $I_{0\to 2\epsilon}^{-1}(I_{0\to 2\epsilon}(K))\sse K$ is the only nontrivial inclusion.
But this follows immediately by choosing $f=I_{0\to 2\epsilon}$ in $I_{0\to 2\epsilon}^{-1}(f(K))\sse K$ in \cref{def_pruning}.

Since $K\sse I$, $I_{0\to 2\epsilon}(K)$ is well-defined and equal to $M_{0\to 2\epsilon}(K)$.
Applying $I_{0\to 2\epsilon}$ to both sides in $I_{0\to 2\epsilon}^{-1}(f(K))\sse K$, we get (iv).
\end{proof}

This brings us to the stability theorem.
One can state a weaker version of the theorem in terms of interleavings.
Say that $N$ is an \emph{$\epsilon$-interleaving refinement} of $M$ if $N$ is of the form $\bigoplus_{i\in \Gamma} N_i$ (the $N_i$ do not have to be indecomposable), where there is a bijection between $\{N_i\}_{i\in \Gamma}$ and $B(M)$ such that matched modules are $\epsilon$-interleaved.
By \cref{thm_erosion_containment_gives_int}, an $\epsilon$-refinement of a module is also an $\epsilon$-interleaving refinement, while by $\epsilon$-eroding the $N_i$ and applying \cref{lem_if_int_then_er}, one gets a $2\epsilon$-refinement from an $\epsilon$-interleaving refinement.
From this, it follows that the theorem below and a version using interleaving refinements instead of refinements are equivalent up to a constant of $2$.
Thus, the counterexample of \cref{thm_counterex} is a counterexample also to a strengthening of the interleaving refinement version by an appropriate constant.

\begin{theorem}
\label{thm_main}
Let $\epsilon\geq 0$, and let $M$ and $N$ be pfd modules with morphisms $\phi\colon M\to N(\epsilon)$ and $\psi\colon N\to M(\epsilon)$ satisfying $\psi(\epsilon)\circ \phi = M_{0\to 2\epsilon}$.
Let $r = \supdim M<\infty$.
Then
\begin{itemize}
\item[(i)] there is a $2r\epsilon$-refinement $\tilde N$ of $N$ such that $B(\Pru_\epsilon(M))\sse B(\tilde N)$,
\item[(ii)] if $\phi(\epsilon)\circ \psi = N_{0\to 2\epsilon}$ (so $\phi$ and $\psi$ form an $\epsilon$-interleaving), then $\Pru_\epsilon(M)$ is a $2r\epsilon$-refinement of $N$.
\end{itemize}
\end{theorem}

Putting $N=M$ in the theorem, we get that $\Pru_\epsilon(M)$ is a $2r\epsilon$-refinement of $M$, which gives the following corollary.
\begin{corollary}
\label{cor_main}
Let $\epsilon\geq 0$, and let $M$ and $N$ be $\epsilon$-interleaved pfd modules with $r = \supdim M<\infty$.
Then $\Pru_\epsilon(M)$ is a $2r\epsilon$-refinement of both $M$ and $N$.
\end{corollary}
Recall the definition of $f_R$ in \cref{def_f_R}.
The following is an immediate consequence of \cref{cor_main}.
A popular alternative way of stating the algebraic stability theorem \cref{thm_ast} for 1-parameter modules is as the inequality $d_B\leq d_I$.
The following is an immediate consequence of \cref{cor_main}, and is similar in spirit to the inequality $d_B\leq d_I$: it is an upper bound on the ``decomposition dissimilarity'' between $M$ and $N$ by their interleaving distance.
\begin{corollary}
\label{cor_f_R}
Let $\epsilon\geq 0$, and let $M$ and $N$ be pfd modules with $r = \supdim M<\infty$.
Then
\[f_R(M,N)\leq 2r d_I(M,N).\]
\end{corollary}
The drawback of the phrasing of this corollary is that $f_R$ is not a distance, as pointed out in the previous section, so it does not quite fit into the tradition of stating stability results as inequalities of distances.

We now move on to the proof of \cref{thm_main}.
Write $N = \bigoplus_{j\in \Gamma} N_j$ with each $N_j$ indecomposable.
We write $\pi_-$ and $\iota_-$ for the canonical projections and inclusion from and into direct sums.
For each $j\in \Gamma$, let $\phi_j$ and $\psi_j$ be the compositions
\[
M \xra{\phi} N(\epsilon) \xra{\pi_j} N_j(\epsilon), \quad N_j \xrightarrow{\iota_j} N \xrightarrow{\psi} M(\epsilon),
\]
respectively.
Let $f_j = \psi_j(\epsilon)\circ \phi_j$, and note that
\begin{align*}
\sum_{j\in \Gamma} f_j &= \psi(\epsilon) \circ \left(\sum_{j\in \Gamma} \iota_j\circ\pi_j\right)(\epsilon)\circ \phi\\
&= \psi(\epsilon) \circ \id_{N(\epsilon)} \circ \phi\\
&= M_{0\to 2\epsilon}.
\end{align*}

Let $(I,K)$ be the $\epsilon$-pruning pair of $M$.
The following lemma gives us a $2r\epsilon$-refinement of $N(\epsilon)$.
Then, in \cref{lem_main_2}, we show that we have the inclusion of barcodes needed for part (i) of the theorem.
After \cref{lem_main_2}, we assume that $\phi$ and $\psi$ form an $\epsilon$-interleaving, and show part (ii) of the theorem.
\begin{lemma}
\label{lem_main_1}
Let $(I,K)$ be the $\epsilon$-pruning pair of $M$.
Define the subquotient
\[
\ol{N_j} = \psi_j(\epsilon)^{-1}(M_{0\to 2\epsilon}(I))/\phi_j(K)
\]
of $N_j(\epsilon)$.
Then $\ol{N}\coloneqq \bigoplus_{j\in \Gamma} \ol{N_j}$ is a $2r\epsilon$-refinement of $N(\epsilon)$.
\end{lemma}

\begin{proof}
By \cref{def_pruning}, $\psi_j(\epsilon)\circ \phi_j(K) \sse \psi_j(\epsilon)\circ \phi_j(I) \sse M_{0\to 2\epsilon}(I)$.
It follows that $\phi_j(K) \sse \psi_j(\epsilon)^{-1}(M_{0\to 2\epsilon}(I))$, so $\ol{N_j}$ is well-defined.
By \cref{lem_I_K} (i) and (ii), we have $\Img_{2r\epsilon}(M(2\epsilon)) \sse M_{0\to 2\epsilon}(I)$ and $K\sse \Ker_{2r\epsilon}(M)$.
By \cref{lem_im_to_im}, these inclusions imply that $\ol{N_j}\in \oEN_{2r\epsilon}(N_j(\epsilon))$ for all $j\in \Gamma$.
Thus, $\ol{N}$ is a $2r\epsilon$-refinement of $N(\epsilon)$.
\end{proof}

\begin{lemma}
\label{lem_main_2}
Let $\ol N$ be as defined in the previous lemma.
We have an inclusion $B(I/K)\sse B(\ol{N})$ of barcodes.
\end{lemma}

\begin{proof}
We identify $\ol{N}$ with the subquotient
\[
\bigoplus_{j\in \Gamma} \psi_j(\epsilon)^{-1}(M_{0\to 2\epsilon}(I))/\bigoplus_{j\in \Gamma} \phi_j(K)
\]
of $N(\epsilon)$ in the obvious way.
We claim that $\phi$ and $\psi(\epsilon)$ induce morphisms
\[
I/K\xra{\ol \phi} \ol{N} \xra{\ol{\psi(\epsilon)}} M_{0\to 2\epsilon}(I)/ M_{0\to 2\epsilon}(K).
\]
To prove this, it suffices to show that the induced morphisms
\[
I/K\xra{\ol{\phi_j}} \ol{N_j} \xra{\ol{\psi_j(\epsilon)}} M_{0\to 2\epsilon}(I)/ M_{0\to 2\epsilon}(K)
\]
are well-defined, as we obtain $\ol{\phi}$ and $\ol{\psi(\epsilon)}$ by summing over $j$.
The nontrivial properties we need to check are $\phi_j(I) \sse \psi_j(\epsilon)^{-1}(M_{0\to 2\epsilon}(I))$ and $\psi_j(\epsilon)(\phi_j(K)) \sse M_{0\to 2\epsilon}(K)$.
Equivalently, $\psi_j(\epsilon)\circ \phi_j$ has to map $I$ into $M_{0\to 2\epsilon}(I)$ and $K$ into $M_{0\to 2\epsilon}(K)$.
But this follows from \cref{def_pruning} and \cref{lem_I_K} (iv), respectively.

Since $\ol{\psi(\epsilon)} \circ \ol \phi$ is equal to the isomorphism $\ol{M_{0\to 2\epsilon}}$, the splitting lemma implies that $\ol{N} = N' \oplus \ker \ol{\psi(\epsilon)}$, where $N' = \img \ol \phi \cong I/K$.
Thus, $B(\ol{N}) = B(I/K \oplus \ker \ol \psi)$, so since $B(I/K \oplus \ker \ol \psi) = B(I/K) \sqcup B(\ker \ol \psi)$, we have $B(I/K)\sse B(\ol{N})$.
\end{proof}

\begin{proof}[proof of \cref{thm_main}]
Since $\Pru_\epsilon(M) = (I/K)(-\epsilon)$, part (i) of the theorem follows from \cref{lem_main_1} and \cref{lem_main_2}.
We now prove part (ii).
Recall that $\ol{N} = N' \oplus \ker \ol{\psi(\epsilon)}$, where $N'\cong I/K$ and $\ol{\psi(\epsilon)}$ is the morphism from $\ol{N}$ to $M_{0\to 2\epsilon}(I)/ M_{0\to 2\epsilon}(K)$ induced by $\psi(\epsilon)$.
Let $Z_j = \pi_j(\psi(\epsilon)^{-1}(M_{0\to 2\epsilon}(K)))$.
Because $K\sse \Ker_{2r\epsilon}(M)$, we have $M_{0\to 2\epsilon}(K) \sse \Ker_{2(r-1)\epsilon}(M)$.
Since $\ker \psi \sse \Ker_{2\epsilon}(N)$ by the interleaving properties, we get
\begin{align*}
\psi(\epsilon)^{-1}(M_{0\to 2\epsilon}(K)) &\sse \psi(\epsilon)^{-1}(\Ker_{2(r-1)\epsilon}(M(2\epsilon)))\\
&= \ker(M_{2\epsilon\to 2r\epsilon}\circ \psi(\epsilon))\\
&= \ker(\psi((2r-1)\epsilon)\circ N_{\epsilon\to (2r-1)\epsilon})\\
&\sse \ker(\phi(2r\epsilon)\circ\psi((2r-1)\epsilon)\circ N_{\epsilon\to (2r-1)\epsilon})\\
&= \ker(N_{(2r-1)\epsilon\to (2r+1)\epsilon}\circ N_{\epsilon\to (2r-1)\epsilon})\\
&= \ker(N_{\epsilon\to (2r+1)\epsilon})\\
&= \Ker_{2r\epsilon}(N(\epsilon)),
\end{align*}
and thus also $Z_j \sse \Ker_{2r\epsilon}(N(\epsilon))$ for all $j$.
We define $\ol{Z_j}$ as the submodule of $\ol{N_j}$ induced by $Z_j$; that is,
\[
\ol{Z_j} = (Z_j\cap \psi_j(\epsilon)^{-1}(M_{0\to 2\epsilon}(I)))/\phi_j(K).
\]

We want to obtain a $2r\epsilon$-refinement of $N(\epsilon)$ by starting with $\ol{N}$ and peeling away $\ker \ol{\psi(\epsilon)}$ so that we are left with $N'$.
Unfortunately, we cannot simply quotient out $\ker \ol{\psi(\epsilon)}$ directly:
To obtain a refinement of $N(\epsilon)$, we need to take subquotients of the individual $N_j(\epsilon)$ (or rather the $\ol{N_j}$, at this stage of the proof), and we have no guarantee that $\ker{\psi(\epsilon)}$ is a submodule of any $\ol{N_j}$, or even that we can write $\ker \ol{\psi(\epsilon)} = \bigoplus_{j\in \Gamma}(\ker \ol{\psi(\epsilon)})_j$ with each $(\ker \ol{\psi(\epsilon)})_j$ a submodule of $\ol{N_j}$.
Our strategy for dealing with this problem is to start with a decomposition $\ker \ol{\psi(\epsilon)} = \bigoplus_{X \in \Xi} X$ and replace each $X$ in turn by an isomorphic module $X^\Ss$ that we are allowed to quotient out because $X^\Ss\sse \ol{Z_j}$ for some $j\in \Gamma$.
(Below, at the stage $\Ss$, we have replaced each $X\in S^\Ss$ with an isomorphic $X^\Ss$, and the indecomposables in $\Xi\setminus S^\Ss$ are yet to be dealt with.)
In the end, we will have a submodule $\bigoplus_{X \in \Xi} X^\Ss \cong \ker \ol{\psi(\epsilon)}$ of $\ol N$ that we can quotient out to get the refinement that we want.
Since $\Xi$ can be uncountable, usual induction is not powerful enough to make this idea work, but Zorn's Lemma (which one can think of as induction for large sets) will get the job done.

Write $\ker \ol{\psi(\epsilon)}$ as a direct sum $\bigoplus_{X \in \Xi} X$ of indecomposables.
Define a \emph{partial decomposition} $\mathcal S$ as a choice of a set $S^\Ss_j\sse \Xi$ and a module $N^{\mathcal S}_j \sse \ol{N_j}$ for each $j\in \Gamma$ such that $S^\Ss_j\cap S^\Ss_{j'}$ is empty for $j\neq j'$, and a module $X^\Ss\sse \ker \ol{\psi(\epsilon)}$ for each $X\in S^\Ss\coloneqq \bigcup_{j\in\Gamma} S^\Ss_j$ such that
\begin{itemize}
\item[(i)] for each $X\in S^\Ss$, $X^\Ss\cong X$,
\item[(ii)] for each $X\in S^\Ss_j$, $X^\Ss\sse \ol{Z_j}$,
\item[(iii)] $\ol{N} = \bigoplus_{j\in \Gamma} N^\Ss_j \oplus \bigoplus_{X \in S^\Ss} X^\Ss = \bigoplus_{j\in \Gamma} N^\Ss_j \oplus \bigoplus_{X \in S^\Ss} X$.
\end{itemize}
Suppose there is a partial decomposition $\Ss$ with $S^\Ss= \Xi$.
Then we have $\ol{N} = \bigoplus_{j\in \Gamma} N^\Ss_j \oplus \bigoplus_{X \in \Xi} X = \bigoplus_{j\in \Gamma} N^\Ss_j \oplus \ker \ol{\psi(\epsilon)}$.
Since $\ol{N} = N' \oplus \ker \ol{\psi(\epsilon)}$, \cref{thm_unique_dec} gives $\bigoplus_{j\in \Gamma} N^\Ss_j \cong N' \cong I/K$.
At the same time, (ii) and the first equality in (iii) give $\ol{N} = \bigoplus_{j\in \Gamma} (N^\Ss_j \oplus K_j)$ for $K_j = \bigoplus_{X\in S_j^\Ss} X^\Ss\sse \ol{Z_j}$, so $N^\Ss_j \cong \ol{N_j}/K_j$.
Because $Z_j \sse \Ker_{2r\epsilon}(N(\epsilon))$ and $\ol{N_j} \in \oEN_{2r\epsilon}(N_j(\epsilon))$, we get $N^\Ss_j \in \EN_{2r\epsilon}(N_j(\epsilon))$, so $\bigoplus_{j\in \Gamma} N^\Ss_j$ is a $2r\epsilon$-refinement of $N(\epsilon)$, and we are done.

Let $\Ss$ and $\T$ be partial decompositions.
We write $\Ss \leq \T$ if $S^\Ss\sse S^\T$, $N^\T_j\sse N^\Ss_j$ for each $j\in\Gamma$, and for each $X\in S^\Ss$, we have $X^\Ss = X^\T$.
This equips the set of partial decompositions with a partial order.
There is a valid (minimal) partial decomposition $\Ss$ defined by $S^\Ss= \emptyset$ and $N^\Ss_j = \ol{N_j}$ for all $j$, so the poset is nonempty.
We want to show that there is a partial decomposition $\Ss$ with $S^\Ss = \Xi$.
By Zorn's lemma, it suffices to show that every chain in the poset of partial decompositions has an upper bound, and that there is no maximal element $\Ss$ with $S^\Ss\nsubseteq \Xi$.

Suppose $C$ is a chain; that is, a set of partial decompositions that is totally ordered by $\leq$.
We claim that an upper bound $\T$ is given by $S^\T=\bigcup_{\Ss\in C} S^\Ss$, $N^\T_j = \bigcap_{\Ss\in C} N^\Ss_j$ and $X^\T = X^\Ss$ for any $\Ss\in C$ such that $X\in \Ss$.
The only nontrivial condition to check is (iii).
Since $N$ is pfd, $\ol{N}_p$ is finite-dimensional for any $p\in \Pb$.
This means that fixing $p$, there is an $\Ss\in C$ such that $(N^\Ss_j)_p = (N^\T_j)_p$ for all $j$, and such that $X_p = X^\T_p = 0$ for all $X\in S^\T\setminus S^\Ss$.
Thus, since (iii) holds for $\Ss$, it also holds for $\T$ at $p$, and for all other $p'\in \Pb$ by the same argument.
$\T$ is therefore indeed a partial decomposition, and an upper bound of $C$.

Now assume $\Ss$ is a maximal element and $S^\Ss\neq \Xi$.
Consider the sequence of identity maps
\[
\ol{N} = N' \oplus \bigoplus_{X \in \Xi} X \to \bigoplus_{j\in \Gamma} N^\Ss_j \oplus \bigoplus_{X \in S^\Ss} X \to N' \oplus \bigoplus_{X \in \Xi} X.
\]
For each $X' \in S^\Ss$, let $\pi_{X'}$ be the canonical projection from $\bigoplus_{j\in \Gamma} N^\Ss_j \oplus \bigoplus_{X \in S^\Ss} X$ to the submodule $X'$, and similarly let $\pi_j$ be the projection onto $N^\Ss_j$.
Pick a module $Y \in \Xi\setminus S^\Ss$.
Let $\pi_Y\colon N' \oplus \bigoplus_{X \in \Xi} X \to Y$ be the canonical projection and $\iota_Y\colon Y\to \ol{N}$ the inclusion.
We have that the sum $\sum_{j\in \Gamma} \pi_j + \sum_{X\in S^\Ss} \pi_X$ is the identity on $\ol{N}$, so $\pi_Y\circ (\sum_{j\in \Gamma} \pi_j + \sum_{X\in S^\Ss} \pi_X)\circ \iota_Y$ is the identity on $Y$.
But $\pi_Y\circ \pi_X = 0$ for all $X\in S^\Ss$, so also $\pi_Y\circ (\sum_{j\in \Gamma} \pi_j)\circ \iota_Y = \sum_{j\in \Gamma} \pi_Y \circ \pi_j\circ \iota_Y$ is the identity on $Y$.
Since $Y$ is indecomposable, its endomorphism ring is local by \cref{thm_unique_dec}.
Thus, at least one of the $\pi_Y \circ \pi_j\circ \iota_Y$ is a unit; i.e., an isomorphism.
Fix a $j = i$ with this property.
We now define a partial decomposition $\T$ with $S^\T = S^\Ss\cup \{Y\}$, $X^\T = X^\Ss$ for all $X\in S^\Ss$ and $Y^\T = \pi_i(Y)$.
We have $\pi_i(Y)\cong Y$, and by the splitting lemma, $\pi_i(Y)$ is a direct summand of $N^\Ss_i$.
Let $N^\T_i$ be a module such that $N^\Ss_i = N^\T_i \oplus \pi_i(Y)$.
It follows that
\begin{equation}
\label{eq_dir_sums}
\bigoplus_{j\in \Gamma} N^\T_j \oplus \bigoplus_{X \in S^\T} X^\T = \bigoplus_{j\in \Gamma} N^\Ss_j \oplus \bigoplus_{X \in S^\Ss} X^\Ss = \ol{N}.
\end{equation}

Next, we show $\ol{N} = \bigoplus_{j\in \Gamma} N^\T_j \oplus \bigoplus_{X \in S^\T} X$.
It suffices to show
\begin{equation}
\label{eq_N'}
\ol{N} = \sum_{j\in \Gamma} N^\T_j + \sum_{X \in S^\T} X,
\end{equation}
since by \cref{eq_dir_sums}, for every $p\in \Pb$,
\[
\sum_{j\in \Gamma} \dim((N^\T_j)_p) + \sum_{X \in S^\T} \dim(X_p) = \sum_{j\in \Gamma} \dim((N^\T_j)_p) + \sum_{X \in S^\T} \dim(X^\T_p) = \dim(\ol{N}_p).
\]
Thus, if \cref{eq_N'} holds, then this sum is pointwise direct and thus direct.
By condition (iii) for $\Ss$, we know that $\pi_i(Y) +\sum_{j\in \Gamma} N^\T_j + \sum_{X \in S^\T} X = \ol{N}$, since $\pi_i(Y) + N^\T_i = N^\Ss_i$, and $N_j^\T = N_j^\Ss$ for $j\neq i$.
This means that to show \cref{eq_N'}, we only have to show that $\pi_i(Y)\sse \sum_{j\in \Gamma} N^\T_j + \sum_{X \in S^\T} X$.
Let $x\in \pi_i(Y)$ and $y = \pi_i^{-1}(x)\in Y$ (viewing $\pi_i$ as an isomorphism $Y \xra \sim \pi_i(Y)$).
We have
\begin{equation}
\label{eq_pi_sum}
y = \sum_{j\in \Gamma} \pi_j(y) + \sum_{X\in S^\Ss} \pi_X(y).
\end{equation}
Now, $y$ and all $\pi_X(y)$ in the sum are elements of $\sum_{X \in S^\T} X$, and for $j\neq i$, $\pi_j(y)$ is an element of $N_j^\T = N_j^\Ss$.
Thus, by rearranging, we can express $x$ as an element of $\sum_{j\in \Gamma} N^\T_j + \sum_{X \in S^\T} X$.

Finally, we need to show $\pi_i(Y)\sse \ol{Z_i}$.
Let $x$ and $y$ be as above.
In \cref{eq_pi_sum}, we know that $y$ and the second sum are in $\ker \ol{\psi(\epsilon)}$, so also $\sum_{j\in \Gamma} \pi_j(y)$ is in $\ker \ol{\psi(\epsilon)}$.
Let $z\in \psi(\epsilon)^{-1}(M_{0\to 2\epsilon}(K))$ be an inverse image of this sum under the projection $\pi\colon \bigoplus_{j\in \Gamma} \psi_j(\epsilon)^{-1}(M_{0\to 2\epsilon}(I)) \to \bigoplus_{j\in \Gamma} \psi_j(\epsilon)^{-1}(M_{0\to 2\epsilon}(I))/\bigoplus_{j\in \Gamma} \phi_j(K) = \ol{N}$.
Since this projection respects the direct sum $\bigoplus_{j\in \Gamma} N_j$, we get $\pi_i(y) = \pi(\pi_i(z))\in \ol{Z_i}$.
\end{proof}

\subsection{Counterexample to a strengthening of the stability theorem}

We now give an example showing that strengthening \cref{thm_main} in a certain way is impossible.
By \cref{lem_refins_are_ers} (ii), the $2r\epsilon$-refinement $\Pru_{2r\epsilon}(M)$ in \cref{thm_main} (i) is in $\EN_{2r\epsilon}(M)$.
By restricting the interleaving morphisms in the following theorem, we get $\phi'\colon M \to (N_1\oplus N_2)(1)$ and $\psi'\colon N_1\oplus N_2 \to M(1)$ with $\psi'\circ \phi' = M_{0\to 2}$.
Putting $n=\frac{r}{2}$, the theorem shows that \cref{thm_main} (i) does not hold for any even integer $r\geq 4$ if $2r$ is replaced by $\delta <\frac{r}{2} -1$, and putting $n=\frac{r+1}{4}$, we get that part (ii) does not hold for $r\geq 7$ with $r\equiv 3 \pmod 4$ if $2r$ is replaced by $\delta<\frac{r+1}{4} -1$.
In other words, \cref{thm_main} (i) cannot be strengthened by a factor of more than roughly $4$, and (ii) by a factor of more than roughly $8$.
\begin{theorem}
\label{thm_counterex}
For any integer $n\geq 2$, there are $1$-interleaved pfd $2$-parameter modules $M\oplus Q$ and $N_1\oplus N_2$ with $\supdim M = 2n$ and $\supdim Q \leq 2n-1$ such that for any $\delta< n-1$ and $R \in \EN_\delta(M)$, there is an $R'\in B(R)$ such that $R'\nleq N_1$ and $R'\nleq N_2$.
\end{theorem}

\begin{figure}
\centering
\begin{tikzpicture}[scale=.23]
\coordinate (q) at (0,12);
\coordinate (x) at (8,8);
\coordinate (y) at (12,4);
\coordinate (z) at (16,0);
\node[left] at (q){$q=q_4$};
\node[right] at (6,18){$q_1$};
\node[left] at (x){$x$};
\node[right] at (9,9){$x_1$};
\node[right] at (12,12){$x_4$};
\node[left] at (y){$y$};
\node[right] at (13,5){$y_1$};
\node[right] at (16,8){$y_4$};
\node[left] at (z){$z$};
\node[right] at (17,1){$z_1$};
\node[right] at (20,4){$z_4$};
\foreach \x in {0,...,4}{
	\foreach \y in {(x),(y),(z)}{
		\draw[color=black,fill=black] \y+(\x,\x) circle (.1);
}
}
\foreach \x in {0,...,3}{
	\draw[color=black,fill=black] (q)+(2*\x,2*\x) circle (.1);
}
\draw[thick] (28,25) to (28,0) to (z) to (16,4) to (y) to (12,8) to (x) to (8,12) to (q) to (0,28) to (25,28);
\foreach \x in {0,...,3}{
	\draw[dashed,thick] (25,28) to (25+\x,28) to (25+\x,25+\x) to (28,25+\x) to (28,25);
}
\node at (24.5,24.5){$I_1$};
\node at (31,28.2){$I_{2n} = I_4$};
\draw[color=black,fill=black] (22,22) circle (.1);
\node[right] at (22,22){$p$};
\end{tikzpicture}
\caption{The curves show the boundary of the $I_i$ for $n=2$.
The solid part is the part that is equal for all $i$, while the dashed curves in the top right corner show the different boundaries for different $i$.
The point $p$ is used in \cref{lem_dim_2n}.}
\label{fig_big_ex}
\end{figure}
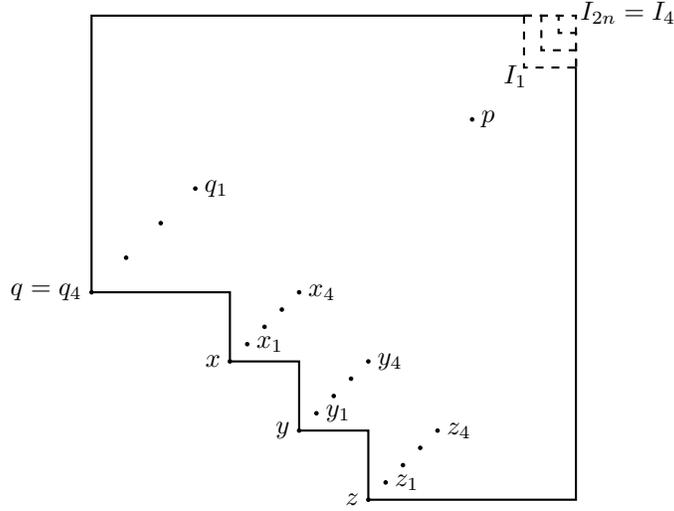

We begin by fixing $n$ and constructing $M$.
Recall that for $p\in \R^2$, we defined $\la p\ra = \{q\in \R^2 \mid q\geq p\}$.
For a set $S$ of points in $\R^2$, let $\la S \ra = \bigcup_{s\in S}\la s \ra$.

Let $q=(0,6n)$, $x=(2n^2,4n)$, $y= (2n^2+2n,2n)$ and $z = (2n^2+4n,0)$.
For $1\leq i\leq 2n$, define intervals $I_i\subseteq \R^2$ as follows.
\begin{align*}
I_i = &\la \{q, x, y, z\}\ra\\
&\setminus \la\{(0,2n^2+10n), (2n^2+8n+i,2n^2+8n+i), (2n^2+10n,0)\}\ra.
\end{align*}
Let $I = \bigoplus_{i=1}^{2n} I_i$.
Let $e_i$ be the $i^\text{th}$ unit vector of $k^{2n}$.
Possibly replacing $I$ by an isomorphic module, we can assume that $I_p$ is the subspace of $k^{2n}$ generated by $\{e_i\mid p\in I_i\}$, and that $I_{p\to q}(v_1,\dots,v_{2n}) = (\bar v_1,\dots,\bar v_{2n})$, where $\bar v_i$ is $v_i$ if $q\in I_i$ and $0$ otherwise.

For $i\in \{1,\dots,2n\}$, let
\begin{align*}
q_i &= q+(2n^2-in,2n^2-in)\\
x_i &= x+(i,i)\\
y_i &= y+(i,i)\\
z_i &= z+(i,i).
\end{align*}
We define $M$ as the submodule of $I$ with the following set of generators:
For each $i\in \{1,\dots, 2n-1\}$, $M$ has a generator $e_i$ at $q_i$, a generator $e_i$ at $x_i$, a generator $e_i-\tilde e_{i+n}$ at $y_i$,
and a generator $e_i-\tilde e_{i+n-1}$ at $z_i$, where we define $\tilde e_i$ as $e_i$ for $i\leq 2n$ and as $e_{2n}$ for $i\geq 2n$.
In addition, $M$ has a generator $e_{2n}$ at $q_{2n}$, $x_{2n}$, $y_{2n}$ and $z_{2n}$.
\begin{lemma}
\label{lem_dim_2n}
Let $0\leq \delta<n-1$, and let $p=(2n^2+7n, 2n^2+7n)$.
If $M_1 \oplus M_2 \in \EN_\delta(M)$, then either $(M_1)_p$ or $(M_2)_p$ has dimension $2n$.
\end{lemma}

\begin{proof}
We can assume that $M_1\oplus M_2$ is of the form $M'/M''$, where $\Img_\delta(M)\sse M'$ and $M''\subseteq \Ker_\delta(M)$.
Note that for any $p'\leq p$, $\Ker_\delta(M)_{p'} = 0$, so $(M_1)_{p'}$ and $(M_2)_{p'}$ are subspaces of $M_{p'}$.
Let $V=(M_1)_p$ and $W=(M_2)_p$.
Because $\Img_\delta(M)_p = M_p$, we have $M_p = V\oplus W$.
Since the map $M_{p'\to p}$ is an inclusion into $k^{2n}$, this means that $(M_1)_{p'} \sse M_{p'}\cap V$ and $(M_2)_{p'} \sse M_{p'}\cap W$ for all $p'\leq p$.
For the rest of the proof, we will only consider points $p'$ with $p'\leq p$, so the observations above hold.

For $i\in \{1,2,\dots,2n\}$, let $q'_i = q_i+(\delta,\delta)$.
By construction, $M_{q_i \to q'_i}$ is an isomorphism.
Thus, $\Img_\delta(M)_{q'_i} = M_{q'_i}$, which gives $(M_1\oplus M_2)_{q'_i} = M_{q'_i}$.
For $i\geq 0$, let $k^i$ denote the subspace of $k^{2n}$ generated by $\{e_1,\dots,e_i\}$ and $k^{-i}$ the subspace generated by $\{e_i,\dots,e_{2n}\}$.
Let $V_i= (M_1)_{q'_i}\sse V$ and $W_i=(M_2)_{q'_i}\sse W$ for $i\leq 2n$ and $V_{2n+1} = W_{2n+1} = 0$.
For each $i\leq 2n$, we must have $V_i\oplus W_i = (M_1\oplus M_2)_{q'_i} = M_{q'_i} = k^{-i}$, and either $\dim(V_i) > \dim(V_{i+1})$ or $\dim(W_i) > \dim(W_{i+1})$.
This means that either $V$ or $W$ must contain a nonzero vector $b_i\in k^{-i}\setminus k^{-(i+1)}$; i.e., a vector whose first nonzero coordinate is in position $i$.
These vectors actually form bases for $V$ and $W$:
There are sets $\beta_V$ and $\beta_W$ with $\beta_V\sqcup \beta_W = \{1,\dots,2n\}$ such that $B_V \coloneqq \{b_i\mid i\in \beta_V\}$ and $B_W \coloneqq \{b_i\mid i\in \beta_W\}$ are bases for $V$ and $W$, respectively.
We claim that if $c\in \beta_V$, then $W$ contains no vector $w$ whose first nonzero coordinate is in position $c$.
To see this, write $w$ as a linear combination over the basis $B_W$ and observe that the first nonzero coefficient of $w$ is in position $i$, where $i$ is the smallest number such that the coefficient of $b_i$ in the linear combination is nonzero.

For $c,d\in \{1,\dots,2n\}$, write $c\sim d$ if $c$ and $d$ are either both in $\beta_V$ or both in $\beta_W$, and assume without loss of generality that $1\in \beta_V$.
We first prove that
\begin{itemize}
\item[(i)] if  $c\leq n-1$, then $c\sim c+n$ and $c\sim c+n-1$.
\end{itemize}
It follows that $n\sim 1\sim n+1\sim 2\sim n+2\sim \dots \sim n-1 \sim 2n-1$, so $\beta_V$ contains $\{1,\dots, 2n-1\}$.
We will also prove that
\begin{itemize}
\item[(ii)] assuming $\{1,\dots, 2n-1\}\subseteq \beta_V$, then $2n\in \beta_V$,
\end{itemize}
which concludes the proof that $\beta_V = \{1,\dots, 2n\}$.
The lemma follows.

To prove the claims, let $c\in \{1,\dots,n\}$, and without loss of generality, assume $c\in \beta_V$.
First,
\begin{align*}
e_c\in M_{x_c}&= \Img_\delta(M)_{x_c +(\delta,\delta)}\\
&\subseteq(M_1)_{x_c+(\delta,\delta)}\oplus (M_2)_{x_c+(\delta,\delta)}\\
&\subseteq M_{x_c+(\delta, \delta)}\\
&= k^{c+n-2}
\end{align*}
This means that there are $v\in V\cap k^{c+n-2}$ and $w\in W\cap k^{c+n-2}$ such that $e_c = v+w$.
Similarly, letting $E = \la\{e_1-\tilde e_{1+n},\dots,e_{c+n-2}-\tilde e_{c+2n-2}\}\ra$, we have
\[
e_c-e_{c+n}\in (M_1)_{y_c+(\delta,\delta)}\oplus (M_2)_{y_c+(\delta,\delta)} \subseteq E,
\]
so $e_c-e_{c+n} = v'+w'$ for some $v'\in V\cap E$ and $w'\in W\cap E$.
We get $e_{c+n} = (v-v')+(w-w')$.
Since $v-v'\in V$ and $w-w'\in W$, the vectors $v-v'$ and $w-w'$ do not have the same first coordinate.
For the same reason, $v'$ and $w'$ do not have the same first coordinate.
The only way this is possible is if the first nonzero coordinate of either $v'-v$ or $w'-w$ is in position $c+n$, and the first coordinate of $v'$ is in position $c$ (and thus $w'$ is zero in position $c$), using $c\in \beta_V$.

At this point, the arguments for (i) and (ii) diverge.
For (i), we can assume $c\leq n-1$, so the only vectors in $E$ whose support contain either $c$ or $c+n$ are in $\la\{e_{c}- e_{c+n}\}\ra$.
Thus, $w'$ is zero also in position $c+n$, so since $v,w\in k^{c+n-2}$, the only vector in $\{v,v',w,w'\}$ supported on $c+n$ is $v'$.
Therefore, the first nonzero coordinate of $v'-v$ is in position $c+n$, so $c+n\in \beta_V$.

The argument for $c+n-1\in \beta_V$ is very similar to what we just did, so we only sketch it.
Defining $F = \la\{e_1-\tilde e_{1+n-1},\dots,e_{c+n-2}-\tilde e_{c+2n-3}\}\ra$ and using $z_c$ instead of $y_c$, we get $e_c - e_{c+n-1} \in F$ and thus $e_c - e_{c+n-1} = v'' + w''$ for some $v''\in V\cap F$ and $w''\in W\cap F$.
This allows us to prove that the first nonzero coordinate of either $v''-v$ or $w''-w$ is in position $c+n-1$, and that out of $\{v,v',w,w'\}$, only $v''$ is supported in position $c+n-1$.
Thus, the first nonzero coordinate of $v''-v$ is in position $c+n-1$, so $c+n-1\in \beta_V$.

Lastly, we prove (ii).
We assume $c=n$, and we have allowed ourselves to assume that $\beta_W$ is empty or contains only $2n$.
If $\beta_W = \{2n\}$, the first nonzero coordinate of $v'-v\in V$ cannot be in position $c+n=2n$, so the first nonzero coordinate of $w'-w$ must be in position $2n$.
Thus, either $w'$ or $w$ is in $\la\{e_{2n}\}\ra$.
But neither $k^{2n-1}$ nor $\la\{e_1-\tilde e_{1+n},\dots,e_{c+n-2}-\tilde e_{c+2n-2}\}\ra$ contain $e_{2n}$, so we have a contradiction.
We conclude that $\beta_W$ is empty.
\end{proof}

\begin{lemma}
\label{lem_interval_morph}
Let $G,R,G',R'\sse \R^2$.
Suppose $I = \la G\ra \setminus \la R\ra$ and $I' = \la G'\ra \setminus \la R'\ra$.
If for all $g\in G$ there is a $g'\in G'$ with $g'\leq g$ and for all $r\in R$ there is a $r'\in R'$ with $r'\leq r$, then for all $c\in k$, there is a valid morphism $f_c\colon I\to I'$ given by multiplication by $c$ in $I\cap I'$ and the zero morphism everywhere else.
\end{lemma}
\begin{proof}
If $c=0$, $f_c$ is simply the zero morphism, so assume $c\neq 0$.
We need to check that for all $p\leq q$, $(f_c)_q\circ I_{p\to q} = I'_{p\to q}\circ (f_c)_p$.
We can assume $p\in I$ and $q\in I'$, as otherwise both sides are zero.
Since $p\in I$, there is a $g\in G$ with $g\leq p$, and thus, by one of the assumptions in the lemma, there is a $g'\in G'$ with $g'\leq g$.
Since $q\in I'$, there is no $r'\in R'$ such that $r'\leq q$, and thus no $r'\in R'$ such that $r'\leq p$.
Therefore, $p\in I'$.
A dual argument gives $q\in I$.

We get that $I_{p\to q}$ and $I'_{p\to q}$ are the identity on $k$, and both $(f_c)_p$ and $(f_c)_q$ are multiplication by $c$.
The lemma follows.
\end{proof}

\begin{proof}[Proof of \cref{thm_counterex}]
Using \cref{lem_interval_morph}, one can see that there are well-defined nonzero morphisms from $I_i$ to $I_j(1)$ whenever $j\in \{i,i+1\}$.
Thus, the morphism $f\colon I\to I(1)$ given by $e_i\mapsto e_{i+1}$ for $1\leq i \leq 2n-1$ and $e_{2n}\mapsto e_{2n}$ is well-defined.
One can also check that $f$ restricts to a well-defined morphism $M\to M(1)$:
It is enough to check that $f(M)\subseteq M(1)$, and for this, it is enough to check $f$ applied to the generators.
For $i\leq 2n-1$, we have $f_{q_i}(e_i) = e_{i+1}\in I_{q_i+(1,1)}$, which is in $M$, as there is a generator $e_{i+1}\in M(1)_{q_{i+1}-(1,1)}$, and $q_{i+1}-(1,1)\leq q_i$.
We leave the almost identical check of the other generators of $M$ to the reader.

Observe that at no point is $e_1$ in the image of $f$.
Nor is $e_{2n}$ in the image of $M_{0\to 1}-f$ at any point, as this image is contained in $\la \{-e_1+e_2,\dots, -e_{2n-1}+e_{2n}\} \ra$.
Thus, both $\img f$ and $\img(M_{0\to\epsilon}-f)$ have pointwise dimension at most $2n-1$.
Therefore, the following lemma shows that there are $1$-interleaved modules of the form $M\oplus Q$ and $N_1\oplus N_2$ where $N_1$, $N_2$ and $Q$ each have pointwise dimension at most $2n-1$.
By \cref{lem_dim_2n}, this means that there is no $R \in \EN_\delta(M)$ such that each $R'\in B(R)$ is a subquotient of $N_1$ or $N_2$, which concludes the proof of the theorem.
\end{proof}

\begin{lemma}
Let $M$ be a module and $f\colon M\to M(\epsilon)$ be a morphism.
Then
\begin{align*}
A &\coloneqq M\oplus \img ((M_{-\epsilon\to 0} -f(-\epsilon))\circ f(-2\epsilon)) \text{ and }\\
B &\coloneqq \img f(-\epsilon)\oplus \img(M_{-\epsilon\to 0}-f(-\epsilon))
\end{align*}
are $\epsilon$-interleaved.
\end{lemma}

\begin{proof}
Let the morphisms $\phi\colon A \to B(\epsilon)$ and $\psi\colon B \to A(\epsilon)$ be defined by
\[
\phi =
\begin{bmatrix}
f&M_{0\to \epsilon}\\
M_{0\to \epsilon}-f&-M_{0\to \epsilon}
\end{bmatrix},
\quad \psi =
\begin{bmatrix}
M_{0\to \epsilon}&M_{0\to \epsilon}\\
M_{0\to \epsilon}-f&-f
\end{bmatrix},
\]
Where we use the same notation for a morphism and its restriction to a submodule.
We need to check that the images of these morphisms are actually contained in the given submodules.
We have $\phi(M\oplus 0) \sse B(\epsilon)$, and
\begin{align*}
\phi(0\oplus \img ((M_{-\epsilon\to 0} -f(-\epsilon))\circ f(-2\epsilon))) \sse &\img(M_{0\to\epsilon}\circ (M_{-\epsilon\to 0} -f(-\epsilon))\circ f(-2\epsilon))\\
&\oplus \img(-M_{0\to\epsilon}\circ (M_{-\epsilon\to 0} -f(-\epsilon))\circ f(-2\epsilon))\\
= &\img(f\circ M_{-\epsilon\to 0}\circ (M_{-2\epsilon\to 0} -f(-2\epsilon)))\\
&\oplus \img((M_{0\to\epsilon} -f)\circ M_{-\epsilon\to 0}\circ f(-2\epsilon))\\
\sse &B(\epsilon).
\end{align*}
Thus, $\img \phi\sse B(\epsilon)$, and a similar calculation shows that $\img \psi\sse A(\epsilon)$.

Having checked that $\phi$ and $\psi$ are well-defined, we now check that they form an interleaving.
\begin{align*}
\phi(\epsilon)\circ \psi &= \begin{bmatrix}
f&M_{0\to \epsilon}\\
M_{0\to \epsilon}-f&-M_{0\to \epsilon}
\end{bmatrix}(\epsilon)
\begin{bmatrix}
M_{0\to \epsilon}&M_{0\to \epsilon}\\
M_{0\to \epsilon}-f&-f
\end{bmatrix}\\
&= \begin{bmatrix}
f(\epsilon)\circ M_{0\to \epsilon} + M_{\epsilon\to 2\epsilon} \circ (M_{0\to \epsilon}-f)&f(\epsilon)\circ M_{0\to \epsilon} - M_{\epsilon\to 2\epsilon} \circ f\\
(M_{\epsilon\to 2\epsilon}-f(\epsilon))\circ M_{0\to \epsilon}- M_{\epsilon\to 2\epsilon}\circ (M_{0\to \epsilon}-f)& (M_{\epsilon\to 2\epsilon}-f(\epsilon))\circ M_{0\to \epsilon}+ M_{\epsilon\to 2\epsilon}\circ f
\end{bmatrix}\\
&= \begin{bmatrix}
M_{\epsilon\to 2\epsilon} \circ M_{0\to \epsilon}&0\\
0& M_{\epsilon\to 2\epsilon} \circ M_{0\to \epsilon}
\end{bmatrix}\\
&= B_{0\to 2\epsilon}.
\end{align*}
Similarly, one can calculate $\psi(\epsilon)\circ \phi = A_{0\to 2\epsilon}$.
\end{proof}

\section{Matchings between interval decomposable modules}
\label{sec_matchings}

We observed that \cref{thm_counterex} shows that \cref{thm_main} is close to optimal in the sense that if the $r$ appearing in the theorem is replaced by something less than roughly $\frac{r}{4}$ in (i) and $\frac{r}{8}$ in (ii), then the theorem is false.
However, there are other ways in which \cref{thm_main} can potentially be significantly improved.
Specifically, we conjecture the following:
\begin{conjecture}
\label{conj_main}
Let $\epsilon\geq 0$, and let $M$ and $N$ be pfd $\epsilon$-interleaved modules.
Let $r = \sup_{M'\in B(M)}\supdim M' < \infty$.
Then there is a constant $c$ independent of $M$, $N$, $\epsilon$ and $r$, and a module $Q$ that is both a $cr\epsilon$-refinement of $M$ and a $cr\epsilon$-refinement of $N$.
Equivalently, there exists a constant $c$ such that $f_R(M,N)\leq crd_I(M,N)$.
\end{conjecture}
This differs from \cref{thm_main} in that it is the maximum dimension of the direct summands of $M$ that matters instead of the maximum dimension of $M$ itself.
If $M$ is indecomposable, the two versions are the same, but if $M$ is a direct sum of many modules with small pointwise dimensions, this can make a large difference.
Moreover, considering the pipeline we outlined in the introduction, we are particularly interested in situations where we have a complicated module with large pointwise dimension that we can (approximately) decompose into many small components.
In this situation, \cref{conj_main} gives a much stronger statement (at least if $c$ is small) than \cref{thm_main}.

In this section, we will first consider a special class of modules, namely those that are \emph{upset decomposable}, and then extend our work to the more general setting of what we refer to as \emph{benign} sets of intervals.
Our main motivation for working with upsets is that these are much less complicated than general modules, while at the same time our feeling is that the main difficulties in proving \cref{conj_main} are preserved in this restricted setting.
The idea is that we want to avoid the ``noisy'' setting of general modules and get a cleaner picture of what it will take to prove \cref{conj_main}.

The main results of this section are \cref{thm_eq_conjs} and \cref{thm_benign}.
The former states that three conjectures are equivalent.
The first is the restriction of \cref{conj_main} to upset decomposable modules; the second, \cref{conj_upset}, is a straightforward statement about matchings between interleaved modules similar to the algebraic stability theorem (\cref{thm_ast}); and the third, \cref{conj_CI}, is phrased in terms of simple graph theory and linear algebra without any reference to persistence modules, which makes the conjectures accessible to a much broader audience than the topological data analysis community.
\cref{thm_benign} says that the equivalence between \cref{conj_upset} and \cref{conj_CI} holds even if we replace upsets in \cref{conj_upset} with a class of benign intervals.
The motivation behind \cref{thm_benign} is that it allows us to consider (among other intervals) \emph{hooks} and \emph{rectangles}, which have made appearances in recent work applying relative homological algebra to multipersistence.
This has become a very active subfield of multipersistence \cite{asashiba2023approximation, blanchette2021homological, botnan2024signed, botnan2024bottleneck, chacholski2024koszul, oudot2024stability} where the idea is to define a simplified barcode that describes a module in terms of intervals.
To show that invariants obtained this way are stable, one often ends up needing a statement analogous to \cref{conj_upset} (see e.g. \cite[Prop. 4.3]{botnan2024bottleneck}) for some class of intervals.
Due to \cref{thm_benign}, we believe that a proof of \cref{conj_CI} would have a major impact on stability questions in this new research direction combining homological algebra and multipersistence.

We make the informal conjecture that proving \cref{conj_main} requires solving an algebraic and a more combinatorial problem, where the algebraic problem is taken care of in \cref{thm_main} and the combinatorial problem is \cref{conj_CI}.
That is, if \cref{conj_CI} is proved, we suspect that we will have overcome most of the difficulties in proving \cref{conj_main}.

Much of the work in this section is inspired by and generalizes results from \cite{bjerkevik_botnan, bjerkevik2019computing} used to prove NP-hardness of approximating the interleaving distance up to a constant of less than $3$.
In \cref{thm_NP}, we give a precise statement describing what it would take to extend the strategy from the NP-hardness proof to obtain a stronger result, where the $3$ is replaced by a larger number.

Some immediate progress towards proving \cref{conj_main} is possible:
Given modules $M$ and $N$, and an $\epsilon$-interleaving $\phi\colon M\to N(\epsilon)$, $\psi\colon N\to M(\epsilon)$, we can apply \cref{thm_main} (i) to each $M'\in B(M)$ and the restricted morphisms $\phi'\colon M'\to N(\epsilon)$ and $\psi'\colon N\to M'(\epsilon)$ to obtain $2r\epsilon$-refinements $R$ of $N$ and $R'$ of $M$, where $r= \supdim M_i$ and $B(R')\sse B(R)$.
What remains, and this seems to be the hardest part, is to combine these refinements of $N$ into a refinement that contains a separate summand associated to each $M'$.
By considering upset decomposable modules, we highlight what might be the main difficulty in proving \cref{conj_main}.
In fact, we will see that even in this restricted setting, \cref{conj_main} is false for $c<3$, meaning that in combining the $2r\epsilon$-refinements as suggested above, we have to sacrifice something.

\subsection{Upset decomposable modules}

We now fix $d\geq 2$.
Recall that we assume that the shift function $\R^d_\epsilon$ is the standard one; that is, the one that sends $p$ to $p+ (\epsilon, \dots, \epsilon)$.
\begin{definition}
An \textbf{upset} is a set $U\sse \R^d$ such that if $p\leq q$ and $p\in U$, then $q\in U$.
Observe that an upset is an interval.
An \textbf{upset module} is an interval module whose supporting interval is an upset.
A module is \textbf{upset decomposable} if it is isomorphic to a direct sum of upset modules.
\end{definition}

\begin{lemma}
\label{lem_upset_er}
Let $U$ be an upset (module) and let $\epsilon\geq 0$.
Then a module $M$ is in $\EN_\epsilon(U)$ if and only if $M$ is isomorphic to an upset module $V$ with $U(-\epsilon)\sse V \sse U$ as upsets.
\end{lemma}

\begin{proof}
Since the shift morphisms $U(-\epsilon) \to U \to U(\epsilon)$ are injective, $\Ker_\epsilon(U)$ is the zero module, while $\Img_\epsilon(U) = U(-\epsilon)$.
``If'' follows from this.
It also follows that if $M\in \EN_\epsilon(U)$, then $M$ is isomorphic to a submodule $V$ of $U$ with $U(-\epsilon)\sse V$ as submodules of $U$.
Since $U$ has pointwise dimension at most one, $V$ is completely determined by its support.
If $M_p\neq 0$, then, since $V$ is a submodule of $U$, $V_q\supseteq \img U_{p\to q} = k$ for all $q\geq p$, which means that $V$ is supported on an upset.
It follows that $V$ is an upset module.
\end{proof}

\begin{lemma}
\label{lem_match_eq_to_refine}
Let $M$ and $N$ be pfd upset decomposable modules.
Then there is an $\epsilon$-matching between $M$ and $N$ if and only if there is a module $Q$ that is an $\epsilon$-refinement of both $M$ and $N$.
\end{lemma}

\begin{proof}
Write $M \cong \bigoplus_{i\in \Lambda} U_i$ and $N \cong \bigoplus_{j\in \Gamma} V_j$, where the $U_i$ and $V_j$ are upsets.
By \cref{lem_upset_er}, $Q$ is an $\epsilon$-refinement of both $M$ and $N$ if and only if there is a bijection $\sigma\colon \Lambda\to \Gamma$ such that $Q$ is isomorphic to a module of the form $\bigoplus_{i\in \Lambda} Q_i$ with $Q_i$ upsets, $U_i(-\epsilon) \sse Q_i \sse U_i$ and $V_{\sigma(i)}(-\epsilon) \sse Q_i \sse V_{\sigma(i)}$.
The existence of such a $Q$ is equivalent to the existence of a $\sigma\colon \Lambda\to \Gamma$ such that $U_i(-\epsilon) \sse V_{\sigma(i)}$ and $V_{\sigma(i)}(-\epsilon) \sse U_i$.
But this is equivalent to $U_i$ and $V_{\sigma_i}$ being $\epsilon$-interleaved (essentially because this allows nonzero morphisms in both directions), so this $\sigma$ gives us an $\epsilon$-matching between $M$ and $N$.
\end{proof}

By \cref{lem_match_eq_to_refine}, the restriction of \cref{conj_main} to the special case where $M$ and $N$ are upset decomposable modules is equivalent to the following conjecture, even if $c$ is fixed:
\begin{conjecture}
\label{conj_upset}
There is a constant $c\geq 1$ such that for any $\epsilon$-interleaved pfd upset decomposable modules $M$ and $N$, there is a $c\epsilon$-matching between $M$ and $N$.
\end{conjecture}

Nothing is known about this conjecture except that the statement is false for $c<3$.
That is, there are $\epsilon$-interleaved upset decomposable modules $M$ and $N$ that do not allow a $c\epsilon$-matching for $c<3$.
Such an example can be constructed by adapting \cite[Ex. 5.2]{bjerkevik2016stability}, or by using \cref{ex_CI} and \cref{lem_CI_to_mod} (with $1$-interleavings and $c$-matchings for $c<3$) below.
A consequence of this is that \cref{conj_main} is false for $c<3$.

\subsection{CI problems}

In \cite{bjerkevik_botnan}, upset decomposable modules (or rather the special case of \emph{staircase decomposable} modules, but the distinction is irrelevant here) are studied, and it is observed that there is a close connection between interleavings between upset decomposable modules and solutions of \emph{constrained invertibility problems}, or \emph{CI problems} for short, which are also introduced in that paper.
In \cite{bjerkevik_botnan} and \cite{bjerkevik2019computing}, CI problems are used to prove that the interleaving distance is NP-hard to compute, and to approximate up to a constant of less than $3$.
\begin{definition}
A \textbf{CI problem} is a pair $(P,Q)$ of $n\times n$-matrices for some $n$ such that each entry of $P$ and $Q$ is either $0$ or $*$.
We say that $n$ is the \textbf{size} of $(P,Q)$.
A \textbf{solution} of $(P,Q)$ is a pair $(A,B)$ of $n\times n$-matrices over $k$ such that $AB$ is the identity matrix, $A$ is zero at each entry corresponding to a zero entry of $P$, and $B$ is zero at each entry corresponding to a zero entry of $Q$.
A \textbf{simple} solution of $(P,Q)$ is a solution $(A,B)$ such that each row of $A$ has exactly one $1$ and the rest of the entries of $A$ are zero.
\end{definition}
It follows from the definition of simple solution that each column of $A$ has exactly one $1$, and that the same holds for the rows and columns of $B$.

The following example is borrowed from \cite[Ex. 4.1]{bjerkevik_botnan}, which is adapted from \cite[Ex. 5.2]{bjerkevik2016stability}.
\begin{example}
\label{ex_CI}
Let
\[
P =
\begin{bmatrix}
*&*&*\\
*&0&*\\
*&*&0
\end{bmatrix},
\quad
Q =
\begin{bmatrix}
*&*&*\\
*&*&0\\
*&0&*
\end{bmatrix}.
\]
Then the CI problem $(P,Q)$ has a solution $(A,B)$, where
\[
A =
\begin{bmatrix}
1&1&1\\
1&0&1\\
1&1&0
\end{bmatrix},
\quad
B =
\begin{bmatrix}
-1&1&1\\
1&-1&0\\
1&0&-1
\end{bmatrix}.
\]
However, $(P,Q)$ does not have a simple solution.
\end{example}

In what follows, we fix two infinite sets $\{u_1,u_2, \dots\}$ and $\{v_1,v_2, \dots\}$ that we will use as vertices in graphs.
For any CI problem $(P,Q)$ of size $n$, we construct an associated bipartite directed graph $\mathcal G_{(P,Q)}$ which has a set $\{u_1,\dots, u_n,v_1,\dots, v_n\}$ of vertices, and an edge from $u_i$ to $v_j$ if and only if $P_{j,i} = *$ and an edge from $v_j$ to $u_i$ if and only if $Q_{i,j} = *$.
Given a graph of this form, we can recover the CI problem in the obvious way, so this is an equivalent description of a CI problem.
A simple solution of $(P,Q)$ can be viewed as a bijection $\sigma\colon \{u_1, \dots, u_n\}\to \{v_1, \dots, v_n\}$ such that for every $i$, there is an edge in both directions between $u_i$ and $\sigma(u_i)$ in $\mathcal G_{(P,Q)}$

Let $U = (U_1,\dots, U_n)$ and $V = (V_1,\dots, V_n)$ be tuples of upsets.
Let $\mathcal P_{U,V,\epsilon}$ be the CI problem $(P,Q)$ of size $n$ with $P_{i,j} = *$ if and only if $V_j\sse U_i(\epsilon)$ and $Q_{j,i} = *$ if and only if $U_i\sse V_j(\epsilon)$.

\cref{lem_bij_CI} below essentially follows from the work in \cite[Sec. 4]{bjerkevik_botnan}.
There, $U$ and $V$ are constructed from a CI problem instead of the other way around, but the only property that matters for the proof is that edges $u_i \to v_j$ correspond to inclusions $U_i\sse V_j(\epsilon)$ (and the corresponding property with $u_i$ and $v_j$ swapped).
Concretely, the morphisms from $U_i$ to $V_j(\epsilon)$ can be identified with elements of the field $k$ if $U_i\sse V_j$ (\cite[Lemma 4.4]{bjerkevik_botnan}), which means that morphisms from $\bigoplus_{i=1}^n U_i$ to $\bigoplus_{j=1}^n V_j(\epsilon)$ and vice versa can be identified with $(n\times n)$-matrices.
One can use this understanding to show how $\epsilon$-matchings correspond to simple solutions of the CI problem:
Namely, pairing $U_i$ with $V_j$ in an $\epsilon$-matching corresponds to pairing vertices $u_i$ and $v_j$ with arrows going in both directions in $\G_{\Pp_{U,V,\epsilon}}$.
\begin{lemma}
\label{lem_bij_CI}
There is a bijection between the $\epsilon$-interleavings between $\bigoplus_{i=1}^n U_i$ and $\bigoplus_{j=1}^n V_j$, and the solutions of $\mathcal P_{U,V,\epsilon}$.
There is also a bijection between the $\epsilon$-matchings between $\bigoplus_{i=1}^n U_i$ and $\bigoplus_{j=1}^n V_j$, and the simple solutions of $\mathcal P_{U,V,\epsilon}$
\end{lemma}

To be able to rephrase \cref{conj_upset} in terms of CI problems, we need the following definition.
\begin{definition}
Let $c\geq 1$ be an odd integer and $(P,Q)$ a CI problem of size $n$.
The \textbf{$c$-weakening} of $(P,Q)$ is the CI problem corresponding to the graph with vertex set $\{u_1,\dots, u_n,v_1,\dots, v_n\}$ and an edge from $u_i$ to $v_j$ (resp. $v_j$ to $u_i$) if and only if there is a path from $u_i$ to $v_j$ (resp. $v_j$ to $u_i$) in $\mathcal G_{(P,Q)}$ of length at most $c$.
\end{definition}

The case $c=3$ of the following is a slight weakening of a conjecture in the appendix of the author's PhD thesis \cite{bjerkevik2020stability}.
\begin{conjecture}
\label{conj_CI}
There is an odd integer $c\geq 1$ such that for any CI problem $\Pp$, if $\Pp$ has a solution, then the $c$-weakening of $\Pp$ has a simple solution.
\end{conjecture}
Our goal is to show that this is equivalent to \cref{conj_upset}.

\subsection{Translating between upset decomposable modules and CI problems}

For $c\geq 1$, let $\lfloor c \rfloor^{\text{odd}}$ be the smallest odd integer less than or equal to $c$.
The following two lemmas let us translate between (simple) solutions of CI problems and interleavings (matchings) between modules.

\begin{lemma}
\label{lem_mod_to_CI}
Let $U = (U_1,\dots, U_n)$ and $V = (V_1,\dots, V_n)$ be tuples of upsets and let $c>0$ and $\epsilon\geq 0$.
If the $\lfloor c \rfloor^{\text{odd}}$-weakening of $\Pp_{U,V,\epsilon}$ has a solution (simple solution), then $\bigoplus_{i=1}^n U_i$ and $\bigoplus_{j=1}^n V_j$ are $c\epsilon$-interleaved ($c\epsilon$-matched).
\end{lemma}

\begin{proof}
It suffices to show the lemma under the assumption that $c$ is an odd integer, so assume $c = \lfloor c \rfloor^{\text{odd}}$.
Let $\Qq$ be the $c$-weakening of $\Pp_{U,V,\epsilon}$.
If there is an edge from $u_i$ to $v_j$ in $\Qq$, then there is a path $u_i = u_{r_0} \to v_{r_1} \to \dots \to u_{r_{c'-1}} \to v_{r_{c'}} = v_j$ of length $c'\leq c$ from $u_i$ to $v_j$ in $\Pp_{U,V,\epsilon}$.
The edges in this path correspond to containments $U_{r_0} \sse V_{r_1}(\epsilon) \sse \dots \sse U_{r_{c'-1}}((c'-1)\epsilon) \sse V_{r_{c'}}(c'\epsilon)$.
Thus, $U_i \sse V_j(c\epsilon)$, so there is an edge from $u_i$ to $v_j$ in $\Pp_{U,V,c\epsilon}$.
This means that $\Qq$ is a subgraph of $\Pp_{U,V,c\epsilon}$ (on the same set of vertices), so if $\Qq$ has a (simple) solution, then $\Pp_{U,V,c\epsilon}$ has a (simple) solution.
But by \cref{lem_bij_CI}, $\Pp_{U,V,c\epsilon}$ has a (simple) solution if and only if $\bigoplus_{i=1}^n U_i$ and $\bigoplus_{j=1}^n V_j$ are $c\epsilon$-interleaved ($c\epsilon$-matched).
\end{proof}

The following is a generalization of the construction in \cite[Sec. 4.2]{bjerkevik2019computing}.

\begin{lemma}
\label{lem_CI_to_mod}
Let $\Qq$ be a CI problem and $C\geq 1$ a constant.
Then there are pfd upset decomposable modules $M$ and $N$ such that for all $1\leq c\leq C$, there is a bijection
\begin{itemize}
\item between the $c$-interleavings between $M$ and $N$ and the solutions of the $\lfloor c \rfloor^{\text{odd}}$-weakening of $\Qq$, and
\item between the $c$-matchings between $M$ and $N$ and the simple solutions of the $\lfloor c \rfloor^{\text{odd}}$-weakening of $\Qq$,
\end{itemize}
and such that there is no $\epsilon$-interleaving between $M$ and $N$ for $\epsilon<1$.
\end{lemma}
\begin{proof}
We will prove the lemma under the assumption that we are working with two-parameter modules.
For $d$-parameter modules with $d>2$, simply take the product of each upset with $[0,\infty)$ $d-2$ times.

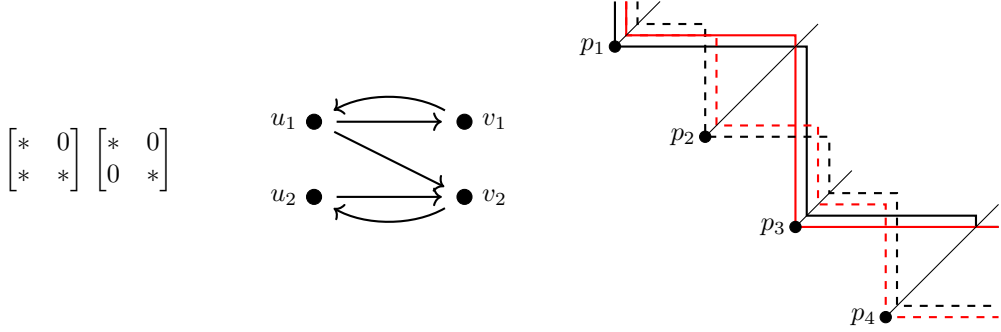
\begin{figure}
\centering
\begin{tikzpicture}[scale=1]
\node at (-3,.5){$\begin{bmatrix}
*&0\\
*&*
\end{bmatrix}$
$\begin{bmatrix}
*&0\\
0&*
\end{bmatrix}$};
\draw[color=black,fill=black] (0,0) circle (.1);
\node at (-.4,0){$u_2$};
\draw[color=black,fill=black] (0,1) circle (.1);
\node at (-.4,1){$u_1$};
\draw[color=black,fill=black] (2,0) circle (.1);
\node at (2.4,0){$v_2$};
\draw[color=black,fill=black] (2,1) circle (.1);
\node at (2.4,1){$v_1$};
\draw[thick, ->, shorten <=.3cm, shorten >=.3cm] (0,1) to (2,1);
\draw[thick, ->, shorten <=.3cm, shorten >=.3cm] (0,1) to (2,0);
\draw[thick, ->, shorten <=.3cm, shorten >=.3cm] (0,0) to (2,0);
\draw[thick, ->, shorten <=.3cm, shorten >=.3cm] (2,1) to [out=150, in=30] (0,1);
\draw[thick, ->, shorten <=.3cm, shorten >=.3cm] (2,0) to [out=-150, in=-30] (0,0);
\begin{scope}[xshift=4cm, yshift=2cm, scale=.15]
\draw[thick] (0,4) to (0,0) to (17,0) to (17,-15) to (32,-15) to (32,-16) to (34,-16);
\draw[thick,color=red] (1,4) to (1,1) to (16,1) to (16,0) to (16,-16) to (32,-16) to (34,-16);
\draw[thick, dashed] (2,4) to (2,2) to (8,2) to (8,-8) to (19,-8) to (19,-13) to (25,-13) to (25,-23) to (34,-23);
\draw[thick,dashed,color=red] (1,4) to (1,1) to (9,1) to (9,-7) to (18,-7) to (18,-14) to (24,-14) to (24,-24) to (34,-24);
\draw[color=black,fill=black] (0,0) circle (.5);
\node[left] at (0,0){$p_1$};
\draw (0,0) to (4,4);
\draw[color=black,fill=black] (8,-8) circle (.5);
\node[left] at (8,-8){$p_2$};
\draw (8,-8) to (18,2);
\draw[color=black,fill=black] (16,-16) circle (.5);
\node[left] at (16,-16){$p_3$};
\draw (16,-16) to (21,-11);
\draw[color=black,fill=black] (24,-24) circle (.5);
\node[left] at (24,-24){$p_4$};
\draw (24,-24) to (34,-14);
\end{scope}
\end{tikzpicture}
\caption{A CI problem $\Qq$ on the left and the graph $\G_\Qq$ in the middle.
On the right is an example of a valid choice of $p_i$ for $C=4$ in the proof of \cref{lem_CI_to_mod}, where the solid black, dashed black, solid red and dashed red curves are the boundaries of the upsets associated to $u_1$, $u_2$, $v_1$ and $v_2$, respectively.
The vectors $w_1 = (0,8,1,8)$, $w_2 = (2,0,3,1)$, $z_1 = (1,8,0,8)$, $z_2 = (1,1,2,0)$ determine where the curves cross the diagonal rays from the $p_i$.
}
\label{fig_CI_to_upsets}
\end{figure}
Let $n$ be the size of $\Qq$.
See \cref{fig_CI_to_upsets} for the construction ahead.
We first construct vectors $w_1, \dots, w_n, z_1, \dots, z_n\in \Z^{2n}$, where we say that the vertex associated to $w_i$ is $u_i$ and the vertex associated to $z_j$ is $v_j$.
We construct the vector $x = (x_1,\dots, x_{2n})$ associated to the vertex $y$ by letting $x_i$ for $1\leq i\leq n$ be the length of the shortest path from $u_i$ to $y$ if such a path exists and $C+2n$ otherwise, and for $n+1\leq i\leq 2n$, letting $x_i$ be the length of the shortest path from $v_{i-n}$ to $y$ if such a path exists and $C+2n$ otherwise.
Observe that $0\leq x_i\leq C+2n$.

Pick $p_1 = (p^1_1,p^2_1),\dots, p_{2n} = (p^1_{2n},p^2_{2n}) \in \R^2$ such that for all $i$, we have $p^1_{i+1}\geq p^1_i+C+ 2n$ and $p^2_{i+1}\leq p^2_i-C- 2n$.
For $x = (x_1,\dots, x_{2n}) \in \{w_1, \dots, w_n, z_1, \dots, z_n\}$, let
\[
x' = \{p_1 + (x_1,x_1),\dots, p_{2n} + (x_{2n}, x_{2n})\}.
\]
For $1\leq i\leq n$, let $U_i$ be the smallest upset containing $w'_i$ and $V_i$ the smallest upset containing $z'_i$.
Because of our assumptions on the $p_j$ and the observation that $0\leq x_j\leq C+2n$ for all $j$, all the points of $w'_i$ (resp. $z'_i$) are on the boundary of $U_i$ (resp. $V_i$).
Thus, for all $\epsilon$, $U_i \sse V_j(\epsilon)$ if and only if $w_i \geq z_j - (\epsilon,\dots,\epsilon)$.
Similarly, $V_j \sse U_i(\epsilon)$ if and only if $z_j \geq w_i - (\epsilon,\dots,\epsilon)$.
But because of the way we constructed $w_i$ and $z_j$, for $0\leq c\leq C$, we have that $w_i \geq z_j - (c,\dots,c)$ holds if and only if there is a path from $u_i$ to $v_j$ of length $\leq c$.
Thus, if $U = (U_1,\dots, U_n)$ and $V = (V_1,\dots, V_n)$, then $\Pp_{U,V,c}$ is the $\lfloor c \rfloor^{\text{odd}}$-weakening of $\Qq$.
As the $c$-interleavings ($c$-matchings) between $\bigoplus_{i=1}^n U_i$ and $\bigoplus_{j=1}^n V_j$ are in bijection with the (simple) solutions of $\Pp_{U,V,c}$ by \cref{lem_bij_CI}, and by construction, there are no $\epsilon$-interleavings between $\bigoplus_{i=1}^n U_i$ and $\bigoplus_{j=1}^n V_j$ for $\epsilon<1$, the lemma follows.
\end{proof}

Though our main focus is not on computational complexity, we include the following theorem, as it is fairly easy to prove using the lemmas we have shown so far.
It describes precisely how to convert the hardness question of $c$-approximating $d_I$ for upset decomposable modules into an equivalent statement about CI problems, which is similar to how we are translating stability questions between upset decomposable modules and CI problems.
The theorem is effectively a generalization of the approach used in \cite{bjerkevik_botnan, bjerkevik2019computing} to prove NP-hardness of $c$-approximating $d_I$ for $c<3$.
\cref{thm_NP} and \cref{conj_CI} give insight into the difficulty of proving NP-hardness of $c$-approximating $d_I$ for $c\geq 3$.
If \cref{conj_CI} holds for some $c$, then the second computational problem in the theorem is easily computable in polytime for the same $c$:
Check if the $c$-weakening has a simple solution.
If yes, we are done; if no, then we know that the original CI problem does not have a solution, and we are done.
Thus, \cref{conj_CI} is highly relevant for the hardness of $c$-approximating $d_I$ for upset decomposable modules, which is the setting of the only proof we have of NP-hardness of computing $d_I$.

For the purpose of determining input size, assume that a CI problem of size $n$ is described in space $O(n^2)$, and that an upset $U$ is described as a list of points $p_1, p_2, \dots, p_m$, where $U = \bigcup_{i=1}^m \la p_i \ra$.
If $M$ and $N$ are finite direct sums of such upsets, then $d_I(M,N) = 0$ if and only if $M\cong N$ if and only if $B(M) = B(N)$, and $d_I(M,N) = \infty$ if and only if $|B(M)| \neq |B(N)|$.
These cases are easy to check in polynomial time, and we disregard them in the theorem below for technical convenience.
\begin{theorem}
\label{thm_NP}
Let $c\geq 1$ be odd.
Given a polynomial time algorithm for one of the following computational problems, we can construct a polynomial time algorithm for the other.
\begin{itemize}
\item[(i)] Given upset decomposable modules $M\ncong N$ with $|B(M)|= |B(N)|$, return a number in the interval $[d_I(M,N), (c+2)d_I(M,N))$.
\item[(ii)] Given a CI problem $\Pp$, decide either that $\Pp$ does not have a solution or that the $c$-weakening of $\Pp$ has a solution.
\end{itemize}
\end{theorem}

\begin{proof}
The proof of \cref{lem_CI_to_mod} gives a polynomial time algorithm taking a CI problem $\Pp$ and returning two modules $M$ and $N$ with properties as described in \cref{lem_CI_to_mod}.
In particular, choosing $C$ large enough, $d_I(M,N)$ is either in $\{1,3,\dots,c\}$, or is at least $c+2$.
Assume we have a polynomial time algorithm giving us a number $r\in [d_I(M,N), (c+2)d_I(M,N))$.
If $r\geq c+2$, then $d_I(M,N)>1$, so $\Pp$ does not have a solution.
If $r<c+2$, then $M$ and $N$ are $c$-interleaved, so the $c$-weakening of $\Pp$ has a solution.
In both cases, we are able to solve (ii).

By \cref{lem_bij_CI}, whether $M$ and $N$ are $\epsilon$-interleaved depends only on the properties of the associated CI problem, which in turn depends only on whether $U\sse V(\epsilon)$ for $U$ and $V$ appearing in the barcodes of $M$ and $N$.
If $B(M)$ and $B(N)$ each have size $n$, there are $O(n^2)$ such pairs $(U,V)$, and the minimal $\epsilon$ such that $U\sse V(\epsilon)$ can be found in polynomial time.
Thus, we can write down a list $L$ of $O(n^2)$ numbers in polynomial time where $d_I(M,N)$ is the smallest value $\epsilon$ of these such that $M$ and $N$ are $\epsilon$-interleaved.
Assume we have a polynomial time algorithm solving (ii).
By \cref{lem_mod_to_CI}, given an $\epsilon\geq 0$, we can use this algorithm to find out either that $M$ and $N$ are not $\epsilon$-interleaved (negative answer), or that they are $c\epsilon$-interleaved (positive answer).
We run this test for each $\epsilon\in L$.
Since $M$ and $N$ are $\epsilon$-interleaved for the largest $\epsilon$ in the list, we must get a positive answer at least once.
Let $\delta$ be the smallest element of $L$ for which we get a positive answer.
Then $d_I(M,N)\in [\delta, c\delta]$, so returning $c\delta$ solves (i).
\end{proof}

We now get to the first of the main results of the section, showing the equivalence between \cref{conj_main} for upset decomposable modules, \cref{conj_upset} and \cref{conj_CI}.
We prove it by putting together results from earlier in the section.
\begin{theorem}
\label{thm_eq_conjs}
For any $c\geq 1$, the following three statements are equivalent:
\begin{itemize}
\item[(i)] Any pfd $\epsilon$-interleaved upset decomposable modules $M$ and $N$ have a common $c\epsilon$-refinement $Q$,
\item[(ii)] for any pfd $\epsilon$-interleaved upset decomposable modules $M$ and $N$, there is a $c\epsilon$-matching between $M$ and $N$,
\item[(iii)] for any CI problem $\Pp$ with a solution, the $\lfloor c \rfloor^{\text{odd}}$-weakening of $\Pp$ has a matching.
\end{itemize}
\end{theorem}
\begin{proof}
By \cref{lem_match_eq_to_refine}, (i) and (ii) are equivalent.
It remains to show that (ii) holds if and only if (iii) does.

If:
Suppose $M$ and $N$ are a counterexample to (ii), so there is an $\epsilon>0$ such that $M$ and $N$ are $\epsilon$-interleaved, but there is no $c\epsilon$-matching between $M$ and $N$.
We can assume $M$ is of the form $\bigoplus_{i=1}^n U_i$ and $N$ of the form $\bigoplus_{j=1}^n V_j$
By \cref{lem_bij_CI}, $\Pp_{(U_1,\dots,U_n),(V_1,\dots,V_n), \epsilon}$ has a solution, but by \cref{lem_mod_to_CI}, its $c$-weakening does not have a simple solution.

Only if:
Suppose $\Pp$ has a solution, but its $\lfloor c \rfloor^{\text{odd}}$-weakening does not have a simple solution.
Then by \cref{lem_CI_to_mod} with $C\geq c$, there are upset decomposable modules $M$ and $N$ that are $1$-interleaved, but do not allow a $c$-matching.
\end{proof}

\subsection{Extending the conjecture to more general classes of intervals}
\label{subsec_benign}

So far in this section, we have worked with upsets.
We now show \cref{thm_benign}, which says that the equivalence between \cref{conj_upset} and \cref{conj_CI} holds if we replace upsets in \cref{conj_upset} with any \emph{benign} collection of intervals.
This extension to benign classes of intervals is of particular interest in the mentioned applications of relative homological algebra to multipersistence.
By \cref{lem_hooks_upsets}, there is a benign set of intervals that contains all hooks and rectangles, and another one containing all upsets.
Thus, \cref{thm_benign} shows that a proof of \cref{conj_CI} would give strong stability results for hook, rectangle and upset decomposable modules, which have been studied in \cite{blanchette2021homological, botnan2024signed, botnan2024bottleneck, chacholski2024koszul, oudot2024stability}.
Furthermore, it would immediately improve some known stability results about homological algebra applications to multipersistence, like \cite[Theorem 1.1]{oudot2024stability}, (the term $n^2-1$ would become linear) and \cite[Theorem 6.1]{botnan2024bottleneck} ($(2n-1)^2$ would become linear).

For an interval $I\sse \R^d$, let $U(I)$ be the smallest upset containing $I$, and let $E(I) = U(I)\setminus I$.
We have $I = U(I) \setminus E(I)$.
Define a poset structure $\leq$ on the set of all intervals\footnote{This poset relation is different from the subquotient relation (\cref{def_subquotient}) on the modules corresponding to the intervals: for instance, $[1,2) \leq [0,1)$, but neither of the interval module supported on $[1,2)$ and the one supported on $[0,1)$ is a subquotient of the other.} by $I\leq J$ if $U(I)\sse U(J)$ and $E(I)\sse E(J)$.
Note that if $I\leq J$, then
\[
I\cap J = (U(I)\setminus E(I)) \cap (U(J)\setminus E(J)) = U(I)\setminus E(J).
\]

\begin{definition}
Let $\I$ be a set of intervals of $\R^d$ closed under shift by any $\epsilon\in \R$.
We call $\I$ \textbf{benign} if for every $I,J\in \I$ such that there exists a nonzero morphism from $I$ to $J$, we have that $I\cap J$ is an interval and $I\leq J$.
\end{definition}

The following lemma shows that this poset structure behaves nicely with respect to morphisms between intervals belonging to the same benign set, which is what allows us to prove \cref{thm_benign}.

\begin{lemma}
\label{lem_benign_poset}
Let $I$ and $J$ be intervals.
\begin{itemize}
\item[(i)] If $I\leq J$, then there is a morphism $f\colon I\to J$ with $f_p = \id_k$ for all $p\in U(I)\setminus E(J) = I\cap J$.
\item[(ii)] Suppose that $I\cap J$ is an interval, and that there is a nonzero morphism $f\colon I\to J$.
Then there is a $0\neq a\in k$ such that for all $p\in I\cap J$, $f_p$ is given by multiplication by $a$.
In this case, write $w(f) = a$.
(We also define $w(g)=0$ for any zero morphism $g$ between intervals.)
\end{itemize}
Let $\I$ be benign and let $I,J,K\in \I$.
\begin{itemize}
\item[(iii)] If there are nonzero morphisms $f\colon I\to J$ and $g\colon J\to K$, then either $U(I)\sse E(K)$, or $g\circ f\neq 0$ with $w(g\circ f) = w(g)w(f)$.
\item[(iv)] If $I\leq J(\epsilon)$ and $J\leq I(\epsilon)$, then $I$ and $J$ are $\epsilon$-interleaved.
\end{itemize}
\end{lemma}

\begin{proof}
(i): Let $f_p = \id_k$ for all $p\in U(I)\setminus E(J)$ and $f_p = 0$ otherwise.
We need to check that $J_{p\to q}\circ f_p = f_q \circ I_{p\to q}$ for all $p\leq q$.
This can only fail if $I_p$ and $J_q$ are both nonzero, so $p\in I$ and $q\in J$.
But since $U(I)\sse U(J)$ and $E(I)\sse E(J)$, this means that $p,q\in I\cap J$, in which case $I_{p\to q}$, $J_{p\to q}$, $f_p$ and $f_q$ are all the identity.

(ii): Let $p,q\in I\cap J$ with $q\geq p$.
Then
\[
f_q = f_q\circ I_{p\to q} = J_{p\to q} \circ f_p = f_p,
\]
where we have used that $I_{p\to q}$ and $J_{p\to q}$ are both the identity on $k$, since $p$ and $q$ are in both $I$ and $J$.
By the connectedness property of intervals (the second point in \cref{def_interval}), the equality $f_p = f_q$ extends to all $q\in I\cap J$.
Since $f$ is nonzero, there must be a $p\in I\cap J$ and $0\neq a\in k$ such that $f_p\colon k\to k$ is given by multiplication by $a$.
By the above, the same holds for $f_q$ for all $q\in I\cap J$.

(iii): Suppose $f$ and $g$ are nonzero and $U(I)\nsubseteq E(K)$.
Pick $p\in U(I)\setminus E(K)$.
By benignness, $U(I)\sse U(J)\sse U(K)$ and $E(I)\sse E(J)\sse E(K)$, which gives $p\in I\cap J\cap K$.
By (ii), $f_p$ and $g_p$ are given by multiplication by $w(f)$ and $w(g)$, respectively, so $g_p\circ f_p = (g\circ f)_p$ is given by multiplication by $w(f)w(g)$.
By (ii) again, this means that $w(g\circ f) = w(f)w(g)$.

(iv):
By (i), we have morphisms $\phi\colon I\to J(\epsilon)$ and $\psi\colon J\to I(\epsilon)$ with $\phi$ and $\psi(\epsilon)$ the pointwise identity in $U(I)\setminus E(J(\epsilon))$ and $U(J(\epsilon))\setminus E(I(2\epsilon))$, respectively.
Since $U(I)\sse U(J(\epsilon))\sse U(I(2\epsilon))$ and $E(I)\sse E(J(\epsilon))\sse E(I(2\epsilon))$, $\psi(\epsilon)\circ\phi$ is the pointwise identity in $U(I)\setminus E(I(2\epsilon))$.
By symmetry, the same holds for $\phi(\epsilon)\circ \psi$ and $J$, so $\phi$ and $\psi$ form an interleaving.
\end{proof}

Recall that for $p\in \R^d$, we have $\langle p\rangle = \{q\in \R^d \mid q\geq p\}$.
The following lemma shows that the set of all intervals with a minimum element is benign.
This set of intervals is of particular interest, since it contains all rectangles (intervals of the form $[a_1,b_1)\times \dots \times [a_d,b_d)$, with $a_i\in \R$ and $b_i\in \R\cup \{\infty\}$) and hooks (intervals of the form $\langle p\rangle\setminus \langle q\rangle$).
We also show that the collection of upsets is benign, tying together benignness and our work in the previous subsections.

\begin{lemma}
\label{lem_hooks_upsets}
Let $\I$ be the set of all intervals $I\sse \R^d$ such that $U(I)$ is of the form $\langle p\rangle$ (where $p$ depends on $I$), and let $\J$ be the set of all upsets of $\R^d$.
Both $\I$ and $\J$ are benign.
\end{lemma}
\begin{proof}
It is immediate that $\I$ and $\J$ are closed under shifts.
Suppose there is a nonzero morphism $f\colon I\to J$, where $I, J\in \I$.
Let $p$ and $q$ be defined by $U(I) = \langle p \rangle$ and $U(J) = \langle q \rangle$.
We have that $f_x$ is nonzero for some $x\in I\cap J$.
Since $p\leq x$, the composition $f_x\circ I_{p\to x}$ is well-defined, and equal to $f_x\neq 0$.
We have
\begin{align*}
f_x\circ I_{p\to x} &= J_{p\to x}\circ f_p,
\end{align*}
so both $J_{p\to x}$ and $f_p$ are nonzero.
We get that $p\in J$.
Thus, $q\leq p$, so $U(I)\sse U(J)$.

Suppose $y\in E(I)$.
Since $q\leq p\leq y$, we have $y\in U(J)$.
Since $y\notin I$, we have $I_{p\to y} = 0$.
Thus,
\begin{align*}
0 &= f_y\circ I_{p\to y} = J_{p\to y}\circ f_p.
\end{align*}
From the previous paragraph, we know that $f_p\neq 0$, so we must have $J_{p\to y}=0$.
Since $p\in J$, it follows that $y\notin J$.
Thus, $y\in E(J)$, so we conclude that $E(I)\sse E(J)$, and therefore $I\leq J$.

It remains to show that $I\cap J$ is an interval.
Since $p\in I\cap J$ and $I\sse \langle p\rangle$, $I\cap J$ is nonempty, and every $x,y\in I\cap J$ are connected by $x\geq p \leq y$.
Suppose that $x\leq y\leq z$ with $x,z\in I\cap J$.
We have $x\in U(I)\cap U(J)$ and $z\notin E(I)\cup E(J)$.
By using the properties of upsets, we get $y\in U(I)\cap U(J)$ and $y\notin E(I)\cup E(J)$, and thus $y\in I\cap J$.

We now show that $\J$ is benign.
For every $I\in \J$, $E(I)$ is empty, and $U(I) = I$.
The intersection of two upsets $I$ and $J$ in $\J$ is an upset, which is an interval, and $E(I)\sse E(J)$ is trivially true.
Suppose that $U(I)\nsubseteq U(J)$, and pick $p\in U(I)\setminus U(J) = I\setminus J$ and $q\in I\cap J$ with $q\geq p$.
This gives
\[
0 = J_{p\to q}\circ f_p = f_q\circ I_{p\to q} = f_q,
\]
since $J_{p\to q}$ is zero and $I_{p\to q}$ is the identity.
By \cref{lem_benign_poset} (ii), we get $f=0$, which finishes the proof.
\end{proof}

We now show that \cref{conj_CI}, which is (i) in the theorem below, is equivalent to a stability result for any set of benign intervals.
\begin{theorem}
\label{thm_benign}
Let $c\geq 1$.
Then the following statements are equivalent.
\begin{itemize}
\item[(i)] For any CI problem $\Pp$, if $\Pp$ has a solution, then the $\lfloor c\rfloor^{\text{odd}}$-weakening of $\Pp$ has a simple solution.
\item[(ii)] For any benign set $\I$ and for any $\epsilon$-interleaved modules $M$ and $N$ such that $B(M)$ and $B(N)$ are finite and only contain intervals in $\I$, there is a $c\epsilon$-matching between $M$ and $N$.
\end{itemize}
\end{theorem}

\begin{proof}
Assume that (ii) is true, and pick $\I$ equal to the set of upsets, which is benign by \cref{lem_hooks_upsets}.
By the implication (ii) $\implies$ (iii) in \cref{thm_eq_conjs}, (i) follows.

Thus, it remains to prove that (i) implies (ii), and we can assume that $c$ is an odd integer.
Fix a benign set $\I$ and modules $M$ and $N$ as in (ii), with interleaving morphisms $\phi\colon M\to N(\epsilon)$ and $\psi\colon N\to M(\epsilon)$.
We can assume that $M = \bigoplus_{i=1}^m I_i$ and $N = \bigoplus_{i=1}^n J_j$ for some $m$ and $n$, where $I_i$ and $J_j$ are in $\I$ for all $i$ and $j$.
We can write $\phi$ as a sum of morphisms $\phi_{i,j}\colon M_i \to N_j(\epsilon)$, and define $\psi_{j,i}$ similarly.
Define an $n\times m$ matrix $A$ and an $m\times n$ matrix $B$ by letting $A_{i,j} = w(\phi_{i,j})$ and $B_{j,i} = w(\psi_{j,i})$, using the notation from \cref{lem_benign_poset} (ii).

Consider the pair
\begin{align*}
C=\begin{bmatrix}
A & \id_n\\
\id_m - BA & -B
\end{bmatrix},
\quad
D=\begin{bmatrix}
B & \id_m\\
\id_n - AB & -A
\end{bmatrix}
\end{align*}
of $(m+n)\times (m+n)$-matrices, where $\id_r$ is the $r\times r$ identity matrix, and let $(P,Q)$ be the CI problem where $P$ and $Q$ are equal to $C$ and $D$, respectively, except that all nonzero elements have been replaced with $*$.
An easy calculation shows that $CD$ is the identity matrix, so $(C,D)$ is a solution of $(P,Q)$.
Thus, assuming (i) in the theorem, the $c$-weakening of $(P,Q)$ has a simple solution.
Label the rows of $P$ by $I_1,\dots, I_m, J_1(\epsilon)\dots, J_n(\epsilon)$ and the columns by $J_1,\dots, J_n, I_1(\epsilon)\dots, I_m(\epsilon)$.
Define multisets
\begin{align*}
B_M &= \{I_1,\dots, I_m\}, & B_N &= \{J_1,\dots, J_n\},\\
B_M^\epsilon &= \{I_1(\epsilon),\dots, I_m(\epsilon)\}, & B_N^\epsilon &= \{J_1(\epsilon),\dots, J_n(\epsilon)\}.
\end{align*}
We view $(P,Q)$ as a directed bipartite graph $\Gg_{(P,Q)}$ on the vertex set $(B_M\sqcup B_N^\epsilon) \sqcup (B_N\sqcup B_M^\epsilon)$.
The graph associated to the $c$-weakening of $(P,Q)$ has the same set of vertices as $\Gg_{(P,Q)}$, and an edge from $u\in B_M\sqcup B_N^\epsilon$ to $v\in B_N\sqcup B_M^\epsilon$ if there is a path from $u$ to $v$ of length at most $c$ in $\Gg_{(P,Q)}$, and vice versa.
Thus, the simple solution of the $c$-weakening of $(P,Q)$ gives a bijection
\begin{align*}
\sigma\colon B_M\sqcup B_N^\epsilon \to B_N\sqcup B_M^\epsilon
\end{align*}
with paths of length at most $c$ between $X$ and $Y$ in $\Gg_{(P,Q)}$ whenever $\sigma(X) = Y$.

We will show that the following defines a $c\epsilon$-matching $\sigma'$ between $B(M)$ and $B(N)$.
If $\sigma(I_i) =J_j$, then we let $\sigma'(I_i) = J_j$.
Any $I_i$ such that $\sigma(I_i)\in B_M^\epsilon$ is left unmatched, and so is any $J_j$ such that $\sigma^{-1}(J_j)\in B_M^\epsilon$.

We make the following claims:
\begin{itemize}
\item[(i)] If $\Gg_{(P,Q)}$ has an edge from $X$ to $Y$, then $X\leq Y(\epsilon)$.
\item[(ii)] If $X\in B_M\sqcup B_N$ and $Y\in B_M^\epsilon \sqcup B_N^\epsilon$, and $\Gg_{(P,Q)}$ has an edge from $X$ to $Y$, then $U(X)\sse E(Y(\epsilon))$.
\end{itemize}
Suppose these claims are true.
If $\sigma(I_i) = J_j$, then there is a path $I_i = X_0\to X_1 \to \dots \to X_r = J_j$ for some $r\leq c$, so by claim (i), $I_i\leq X_1(\epsilon)\leq \dots \leq  J_j(r\epsilon)\leq J_j(c\epsilon)$.
We also have $J_j\leq I_i(c\epsilon)$, so by \cref{lem_benign_poset} (iv), $I_i$ and $J_j$ are $c\epsilon$-interleaved.
Thus, matching $I_i$ with $J_j$ is allowed in a $c\epsilon$-matching.

Now suppose $\sigma(I_i) = J_j(\epsilon)$ for some $j$, in which case $I_i$ is not matched by $\sigma'$.
In this case, we have paths
\begin{align*}
I_i &= X_0\to X_1 \to \dots \to X_r = J_j,\\
J_j &= X'_0\to X'_1 \to \dots \to X'_{r'} = I_i
\end{align*}
with $r,r'\leq c$.
By claims (i) and (ii), we get
\begin{align}
\label{eq_path}
U(I_i) &=U(X_0)\sse \dots \sse U(X_\ell(\ell\epsilon)) \sse E(X_{\ell+1}((\ell+1)\epsilon))\sse \dots \sse E(X_r(r\epsilon)) = E(J_j((r+1)\epsilon)),\\
E(J_j(\epsilon)) &=E(X'_0)\sse \dots \sse E(X'_r(r'\epsilon)) = E(I_i(r'\epsilon))
\end{align}
and thus, $U(I_i)\sse E(I_i((r+r')\epsilon))\sse E(I_i(2c\epsilon))$.
Thus, $I_i\cap I_i(2c\epsilon)$ is empty, so $I_i$ is $c\epsilon$-interleaved with the zero module, which means that we are allowed to leave it unmatched in a $c\epsilon$-matching.
By symmetry, any $J_j$ that is unmatched by $\sigma'$ is also allowed to be unmatched in a $c\epsilon$-matching.
We conclude that $\sigma'$ is a valid $c\epsilon$-matching.

We now prove the claims.
We will assume $X\in B_M\sqcup B_M^\epsilon$; the cases $X\in B_N\sqcup B_N^\epsilon$ are entirely analogous.

If there is an edge in $\Gg_{(P,Q)}$ from $I_i$ to $J_j$, that means that $A_{i,j}\neq 0$, which means that $\phi_{i,j}\neq 0$.
By definition of benignness, we get $I_i\leq J_j(\epsilon)$.
Similarly, the existence of an edge from $I_i(\epsilon)$ to $J_j(\epsilon)$ implies $-A_{i,j}\neq 0$, which gives $I_i(\epsilon)\leq J_j(2\epsilon)$.
Moreover, we only have an edge from $I_i(\epsilon)$ to $I_{i'}$ if $i=i'$, in which case $I_i(\epsilon)\leq I_{i'}(\epsilon)$ is immediate.

This leaves us with the case of an edge going from $I_i$ to $I_{i'}(\epsilon)$, when we have $(\id_m-BA)_{i,i'}\neq 0$.
We want to show that $U(I_i)$ and $E(I_i)$ are both contained in $E(I_{i'})(2\epsilon)$.
Since $E(I_i)\sse U(I_i)$, it is enough to show this for $U(I_i)$.
Suppose $U(I_i)$ is not contained in $E(I_{i'})(2\epsilon)$.
By \cref{lem_benign_poset} (iii), $w(\psi_{j,i'}(\epsilon)\circ \phi_{i,j}) = w(\psi_{j,i'}) w(\phi_{i,j})$ for all $j$.
This gives $(DC)_{i,i'} = w(\sum_{j=1}^m \psi_{j,i'}(\epsilon)\circ \phi_{i,j})$, which is zero if $i\neq i'$, and $1$ if $i=i'$, since $\phi$ and $\psi$ form an interleaving.
Thus, $(BA)_{i,i'} = (\id_m)_{i,i'}$, which contradicts $(\id_m-BA)_{i,i'}\neq 0$.
This means that our assumption that $U(I_i)$ is not contained in $E(I_{i'})(2\epsilon)$ is false.
\end{proof}

\section{Conclusion}

The goal of this paper is to start building a stability theory for decomposition of multiparameter modules.
The centerpiece of the paper is \cref{thm_main}, which shows that it is possible to obtain positive decomposition stability results despite the difficulties described in \cref{sec_motivation}.
To even be able to state this theorem, we had to introduce $\epsilon$-refinements, which is itself a simple, but powerful contribution to the theory.
Previously, except for very special cases like certain types of interval decomposable modules, it was only known that decomposition of modules is wildly unstable in a naive sense, and no positive stability results for general modules existed to our knowledge.
The stability result \cref{thm_main}, the instability example in \cref{thm_counterex}, \cref{conj_main} and the idea of $\epsilon$-prunings together tell a story about the (in)stability properties of decompositions of general modules that is much more nuanced and enlightening than what was previously known.
Given the importance of decompositions in theoretical mathematics and the potential of approximate decompositions as a tool to study data, this is in our opinion a big step forward for multiparameter persistence theory.

\cref{thm_main} serves as a motivation for a new multiparameter pipeline using prunings, which can act as stabilized barcodes.
We define the pruning distance, a version of which could be a long missing stable multiparameter bottleneck distance, connect our work on decompositions of general modules with stability questions in simpler settings relevant for computational complexity and applications of homological algebra to multipersistence, and identify an elegant conjecture (\cref{conj_CI}) that seems to sit at the center of the theory of multiparameter stability.
Thus, our work connects practical applications, stability, computational complexity and distances in a way that we believe can give us a deeper understanding of all these facets of multipersistence.

We are hopeful that prunings can turn out to be a very useful tool for data analysis.
With their ability to split apart almost-decomposable modules, they may allow us to systematically separate parts of a module expressing different features of a data set, which is a significant challenge in multipersistence.
A natural next step is to try to develop efficient algorithms for computing prunings.
Given such an algorithm, we would like to investigate the stability properties of prunings in realistic examples, to find out if the $\epsilon$-pruning of $M$ is usually an $r'\epsilon$-refinement of $M$ for $r'$ much smaller than $2r$, or if we cannot do much better than $2r$ in practice.
We would also like to continue developing prunings from a theoretical point of view, by investigating the structure of the family $\{\Pru_\epsilon(M)\}_{\epsilon\geq 0}$ for modules $M$, and by proving (a version of) \cref{conj_pruning}, showing stability of the pruning distance.

\bibliography{stable-dec}

\end{document}